\documentclass[11pt]{article}
\usepackage{amsmath}
\usepackage{amssymb}
\usepackage{amscd}
\usepackage{amsthm}

\newtheorem{theorem}{Theorem}[section]
\newtheorem{proposition}[theorem]{Proposition}
\newtheorem{lemma}[theorem]{Lemma}
\newtheorem{claim}[theorem]{Claim}
\newtheorem{corollary}[theorem]{Corollary}

\newtheorem{assumption}[theorem]{Assumption}
\newtheorem{D}[theorem]{Definition}
\newenvironment{definition}{\begin{D} \rm }{\end{D}}
\newtheorem{R}[theorem]{Remark}
\newenvironment{remark}{\begin{R}\rm }{\end{R}}
\newtheorem{E}[theorem]{Example}
\newenvironment{example}{\begin{E}\rm }{\end{E}}

\def\Zee{\mathbb{Z}}
\def\Q{\mathbb{Q}}
\def\Ar{\mathbb{R}}
\def\Cee{\mathbb{C}}
\def\Pee{\mathbb{P}}
\def\Id{\operatorname{Id}}
\def\Ker{\operatorname{Ker}}
\def\Hom{\operatorname{Hom}}
\def\Aut{\operatorname{Aut}}
\def\Pic{\operatorname{Pic}}
\def\im{\operatorname{Im}}
\def\Spec{\operatorname{Spec}}

\def\scrO{\mathcal{O}}
\def\spcheck{^{\vee}}

\title{On the geometry of anticanonical pairs}
\author{Robert Friedman}

\begin{document}

\maketitle

\section*{Introduction}

Let $Y$ be a smooth rational surface and let $D\subseteq Y$ be an anticanonical divisor, i.e.\ $D\in |-K_Y|$. It has long been known that such pairs $(Y,D)$ have a rich geometry. Classically, one studied the case where $Y$ is a del Pezzo surface and, typically, $D$ is a smooth divisor. In modern times, such pairs $(Y,D)$ have arisen in the study of deformations of certain elliptic Gorenstein singularities as well as (in a somewhat related manner) in the study of degenerations of $K3$ surfaces, and were first investigated systematically in 1981 in a fundamental paper by Looijenga \cite{Looij}. Recently, Gross, Hacking and Keel have introduced new ideas into this subject \cite{GHK0}, \cite{GHK} with a view toward understanding mirror symmetry for the pair $(Y,D)$.

If $D$ is reduced, there are only three possibilities for $D$: it is either (i) a smooth elliptic curve, (ii) a cycle of $r$ smooth rational curves for $r\geq 2$ (or an irreducible nodal curve in case $r=1$), or (iii) satisfies one of three exceptional cases (irreducible cuspidal, two smooth components meeting along a tacnode, or three smooth components meeting at a point, pairwise transversally).  This survey is concerned with  case (ii), although some of the results carry over with minor modifications to the other cases. Although we shall not make this assumption in what follows, the most interesting case from the point of view of singularity theory is where the intersection matrix of $D$ is negative definite. In this case, the divisor $D$ can be analytically contracted in $Y$, leading to a normal complex analytic surface $\overline{Y}$ which has trivial dualizing sheaf $\omega_{\overline{Y}}$. Thus $\overline{Y}$ is a singular analogue of a $K3$ surface, and many of the familiar properties of $K3$ surfaces carry through to this case as well. For example, the deformation space of $\overline{Y}$ is smooth and the locus of $\overline{Y}$ which are in fact projective is a countable union of proper subvarieties which is Zariski dense (but is not a dense open subset of the moduli space in the classical topology). 

Unlike the case of $K3$ surfaces, there are infinitely many different families of such surfaces $\overline{Y}$, or more generally pairs $(Y,D)$,  which are not deformation equivalent or even topologically the same.  For example, there are infinitely many possibilities for the number $r$ of components of the cycle $D$ and for the possible self-intersections $d_i= D_i^2$ of the components of the cycle. Moreover, the information of $r$ and the $d_i$ does not always specify the deformation type of $(Y,D)$. In this sense, the situation is more like the case of (smooth) Calabi-Yau threefolds, where it is currently unknown whether there are finitely many or infinitely many deformation types. As we shall describe more explicitly, all of the families of pairs $(Y,D)$ are related by birational operations, and some are related by deformations as well. One can thus regard the study of families of pairs $(Y,D)$ as a toy model for the ``Clemens-Reid fantasy" for Calabi-Yau threefolds \cite{Reid}.  Despite the elementary construction of all families of pairs $(Y,D)$, their study quickly leads to challenging combinatorial and lattice-theoretic questions, and they have an intricate and largely unexplored associated geometry.

Our goal in this partly expository paper is to survey some of the main results about such pairs $(Y,D)$, both old and new, with a view toward giving fairly complete proofs as far as possible. In Section 1, we collect some preliminary definitions and elementary, well-known results. Section 2 deals with minimal models and the birational geometry of anticanonical pairs. In Section 3, we discuss the deformation theory of a pair $(Y,D)$, both under the assumption that the deformation is locally trivial for the singularities of $D$ and in general, and begin the analysis of the period map. In particular, we identify the differential of the period map and sketch the elementary argument that the period map is surjective. Section 4 describes the ample (or equivalently, nef) cone of $Y$ and gives a proof of the analogue of Mayer's theorem for $K3$ surfaces on the possible base locus and fixed components of a nef divisor; these results were originally proved in \cite{Fried1} and \cite{Harbourne}. In Section 5, we define the generic ample cone of a pair $(Y,D)$, which is the ample cone of a very general deformation of $(Y,D)$, and describe its relevance to the geometry, deformation theory, and smooth topology of $(Y,D)$. In Section 6, we  describe the Looijenga roots of a pair $(Y,D)$; roughly speaking, these are elements of square $-2$ in the orthogonal complement in $H^2(Y;\Zee)$ to the classes of the components of $D$ which become the classes of smooth rational curves of self-intersection $-2$ disjoint from $D$ under some deformation of $(Y,D)$. While both the generic ample cone and the roots already appear in Looijenga's Annals paper, their significance is somewhat obscured in the cases considered in \cite{Looij} by the presence of a large reflection group, and the crucial role that they play in the general case was clarified in \cite{GHK}. Finally, we give a different characterization of the roots: Aside from the case where $Y$ is either $\mathbb{F}_0$ or $\mathbb{F}_2$, a root is always the difference of the classes of two disjoint exceptional curves. Sections 7, 8, and 9 are an account of the work of Gross-Hacking-Keel \cite{GHK}. Section 7 is concerned with a somewhat technical result concerning automorphisms of the pair $(Y,D)$. In Section 8, we prove various incarnations of the Torelli theorem for pairs $(Y,D)$ as formulated in \cite{GHK}. Finally, in Section 9, we apply the Torelli theorem to characterize automorphisms of $(Y,D)$. Here, the role of integral isometries of $H^2(Y;\Zee)$ which preserve the classes $[D_i]$ as well as the generic ample cone of $Y$, while again implicit in \cite{Looij}, was first made explicit in the general case in \cite{GHK}.

\medskip
\noindent\textbf{Acknowledgements.} It is a pleasure to thank Mark Gross, Paul Hacking and Sean Keel as well as Philip Engel, Radu Laza, and Viacheslav Nikulin  for many helpful discussions and correspondence.

\section{Preliminaries}

We work over $\Cee$. Throughout this paper,  $Y$ denotes a smooth rational surface with  $-K_Y = D$. Unless otherwise noted, $D= \sum_{i=1}^rD_i$ is a (reduced) cycle of rational curves, where $i$ is taken mod $r$, i.e.\ each $D_i$ is a smooth rational curve and $D_i$ meets  $D_{i\pm 1}$ transversally,  except for $r=1$, in which case $D_1=D$ is an irreducible nodal curve, i.e.\ that we are in case (ii) of the introduction. We note again, however, that many of the results in this paper can be generalized to the case where $D\in |-K_Y|$ is not assumed to be a cycle. In particular, we shall denote Case (i) of the introduction, i.e.\ the case where $D$ is a smooth elliptic curve,  as the \textsl{elliptic case}, and Case (iii) of the introduction, i.e.\ the case where $D$ is cuspidal, tacnodal, or has the analytic type of three concurrent lines,  as the \textsl{triangle case}. For a component $D_i$ of $D$, we define $D_i^{\text{int}}\cong \mathbb{G}_m$ to be $D_i -\bigcup_{j\neq i}D_j$ in case $r\geq 2$ and to be the smooth locus of $D_1 = D$ in case $r=1$ and set $D_{\text{reg}} = \bigcup_iD_i^{\text{int}}$. The group $\Pic Y$ is isomorphic to $H^2(Y;\Zee)$ via the first Chern class; given $\alpha \in H^2(Y;\Zee)$, the unique line bundle whose first Chern class is $\alpha$ will be denoted by $L_\alpha$. If $C$ is a curve or divisor class on $Y$, the corresponding cohomology class will be denoted by $[C]$. Intersection pairing on divisor classes, line bundles, or $H^2(Y;\Zee)$ will be denoted by $\cdot$ (multiplication).

 The integer $r=r(D)$ is called the \textsl{length} of $D$. If the components of $D$ are indexed as above, we call such a $(Y,D)$ a \textsl{labeled anticanonical pair}. It is often more natural to consider the $D_i$ up to cyclic permutation. An \textsl{orientation} of $D$ is an orientation of the dual graph (with appropriate modifications in  case $r=1$). An orientation determines and is determined by the labeling of the components of $D$ as above up to cyclic permutation  if $r\geq 3$, and in this case we always choose a labeling, or a labeling up to cyclic permutation, which is compatible with the orientation in the obvious sense. We shall abbreviate the data of the surface $Y$ and the oriented cycle $D$ by $(Y,D)$ and refer to it as a \textsl{anticanonical pair}. 

Another very concrete way to think about orientations is as follows. Given the cycle $D$, with normalization $\widetilde{D}$, we consider the possible labelings of   points $q$ which are the preimages in $\widetilde{D}$ of singular points by either $0$ or $\infty$, subject to the requirements that (i) If $q_1$ and $q_2$ are preimages of the same point $p$ of $D_{\text{sing}}$, then exactly one of the $q_i$ is labeled by $0$ and hence the other is labeled by $\infty$; (ii) If $\widetilde{D}_i$ is a component of $\widetilde{D}$ (and hence $\widetilde{D}_i = D_i$ unless $r=1$) and $q_1$ and $q_2$ are the two distinct points of $\widetilde{D}_i$ whose images in $D$ are singular, then exactly one of the $q_i$ is labeled by $0$ and hence the other is labeled by $\infty$. It is easy to see  that, if we fix a   $q\in \widetilde{D}$ which is the preimage of a singular point, then the choice of labeling of $q$ by either $0$ or $\infty$ determines the set of labelings of all of the preimages of singular points as above, and that the two different choices correspond to the two possible orientations of $D$.

An \textsl{isomorphism} $\varphi\colon  (Y,D)\to (Y', D')$ of two labeled  anticanonical pairs  of the same length is an isomorphism $\varphi\colon Y \to Y'$ such that $\varphi (D_i) = D_i'$ for all $i$, and which preserves the orientation if $r\leq 2$ (this is automatically true if $r\geq 3$). For a single (unlabeled) anticanonical pair $(Y,D)$, an \textsl{automorphism} $\varphi$ of $(Y,D)$ is an automorphism $\varphi$ of $Y$ such that $\varphi(D_i) = D_i$ for every $i$, and such that moreover $\varphi$ preserves the orientation (which again is automatic if $r\geq 3$).

If the intersection matrix $(D_i\cdot D_j)$ is negative definite, we say that $(Y,D)$ is a \textsl{negative definite anticanonical pair} or that $D$ is \textsl{negative definite}. Negative semidefinite is defined similarly. We say that $(Y,D)$ is \textsl{strictly negative semidefinite} if $(Y,D)$ is negative semidefinite but not negative definite.  Note that, if no $D_i$ is exceptional, then $(Y,D)$ is negative definite $\iff$ $D$ is irreducible and $D^2<0$ or $D$ is reducible, $D_i^2 \leq -2$ for all $i$ and there exists a $j$ such that $D_j^2\leq -3$, and $(Y,D)$ is strictly negative semidefinite $\iff$ $D$ is irreducible and $D^2=0$ or $D$ is reducible and $D_i^2=-2$ for all $i$, in which case $r\leq 9$. We define the sequence $(D_1^2, \dots, D_r^2)$ to be the \textsl{self-intersection sequence} of $(Y,D)$;  it is well-defined up to cyclic permutation, and well-defined if $(Y,D)$ is labeled.  
The following is a  useful  numerical invariant of the pair $(Y,D)$:

\begin{definition} The \textsl{charge} $Q(Y,D)$ of the anticanonical pair $(Y,D)$ is defined as
$$Q(Y,D) = 12 -D^2 -r(D).$$
\end{definition}

\begin{lemma}\label{calcQ} With $Q(Y,D)$ as defined above,
$$Q(Y,D) = 2 + b_2(Y) - r(D) = \chi_{\rm{top}}(Y-D),$$
where $b_2(Y)$ is the second Betti  number of $Y$ and $\chi_{\rm{top}}$  denotes the topological Euler characteristic.
\end{lemma}
\begin{proof} For an arbitrary rational surface $Y$, $b_2(Y)  = 10 -K_Y^2$. Thus, for an anticanonical pair $(Y,D)$, $b_2(Y)  = 10 -D^2$, which is equivalent to the statement that
$$Q(Y,D) =12 -D^2 -r(D) =  2 + b_2(Y)  - r(D).$$
To see the second equality, let $\widetilde{D}$ be the normalization of $D$ and let $D_{\rm{sing}}$ be the set of double points of $D$. There are $r(D)$ components of $\widetilde{D}$, each isomorphic to $\Pee^1$ and hence with $\chi_{\rm{top}} = 2$, and $\#(D_{\rm{sing}}) = r(D)$ as well.  By standard results (cf.\ \cite[(3.2.4)]{Deligne}), 
\begin{align*}
\chi_{\rm{top}}(Y-D) &= \chi_{\rm{top}}(Y) - \chi_{\rm{top}}(\widetilde{D}) + \chi_{\rm{top}}(D_{\rm{sing}})\\
&= 2 + b_2(Y)  -2r(D) + r(D) = 2 + b_2(Y)  - r(D),
\end{align*} 
completing the proof.
\end{proof}

\begin{corollary} If $(Y,D)$ is negative definite, then $Q(Y,D) \geq 3$. 
\end{corollary}
\begin{proof} If $(Y,D)$ is negative definite, then the classes $[D_i]$ are independent in $H^2(Y;\Zee)$. Since intersection pairing on $H^2(Y,\Zee)$ is nondegenerate and indefinite,  $b_2(Y) \geq 1+ r(D)$, and hence  $Q(Y,D) = 2 + b_2(Y)  - r(D) \geq 3$.
\end{proof}

A basic lattice-theoretic invariant of the pair $(Y,D)$ is defined as follows::

\begin{definition}\label{defLambda} Let $\Lambda = \Lambda(Y,D) \subseteq H^2(Y; \Zee)$ be the orthogonal complement of the lattice spanned by the classes $[D_i]$.
\end{definition}

\begin{lemma}\label{rankQ} The lattice $\Lambda(Y,D)$ is free. Its rank is $Q(Y,D) -2$ if the classes $[D_1], \dots, [D_r]$ are linearly independent. More generally, if $s$ is the rank of the kernel of the homomorphism $\bigoplus_i\Zee[D_i] \to  H^2(Y; \Zee) $, then the rank of $\Lambda(Y,D)$ is
$$\operatorname{rank} \Lambda(Y,D) = Q(Y,D) -2 + s.$$
\end{lemma}
\begin{proof} There is an exact sequence
$$0\to \Q^s \to \bigoplus_i\Q[D_i] \to  H^2(Y; \Q)  \to \Lambda(Y,D)\otimes _\Zee\Q \to 0,$$
and hence the rank of $\Lambda(Y,D)$ is equal to $b_2(Y)  - r(D) +s = Q(Y,D) -2 + s$.
\end{proof}

Finally, we collect some well-known results about line bundles on cycles $D =D_1+\dots D_r$ of rational curves. Recall that, if $L$ is a line bundle on $D$, then the \textsl{multidegree}  of $L$ is the ordered $n$-tuple of integers 
$$(\deg (L|D_1), \dots, \deg (L|D_r)).$$
The neutral component $\Pic^0D$ of $\Pic D$ is then the group of line bundles on $D$ of multidegree $(0, \dots, 0)$. Equivalently, $\Pic^0D$ is the subgroup of $\Pic D$ consisting of line bundles $L$ such that the pullback of $L$ to the normalization $\widetilde{D}$ is trivial. We have the following standard result:

\begin{lemma}\label{Pic0isomG} The choice of an orientation of $D$ gives a canonical isomorphism $\psi \colon \Pic^0D \cong \mathbb{G}_m$. Via this isomorphism, if $p, q \in D_i^{\text{\rm{int}}}$ and we choose coordinates so that $p$ corresponds to $1$ and $q$ to $\lambda \in \mathbb{G}_m$, then 
$$\psi(\scrO_D(q-p)) =\lambda^{-1}.$$
\end{lemma}
\begin{proof} If $\nu\colon \widetilde{D} \to D$ is the normalization map, then we have the exact sequence
$$1 \to \scrO_D^* \to \nu_*\scrO_{\widetilde{D} }^* \to \nu_*\scrO_{\widetilde{D} }^*/\scrO_D^* \to 1,$$
where a local calculation identifies $\nu_*\scrO_{\widetilde{D} }^*/\scrO_D^*$ with $\mathbb{G}_m^r$. Taking global sections, we see that there is an exact sequence
$$1 \to \mathbb{G}_m \to \mathbb{G}_m^r \to \mathbb{G}_m^r \to H^1(D; \scrO_D^*) \to \bigoplus_iH^1(\widetilde{D}_i; \scrO_{\widetilde{D}_i }^*) \cong \Zee^r \to 0,$$
where the last homomorphism is just the multidegree. Thus there is an exact sequence
$$1 \to \mathbb{G}_m \xrightarrow{\delta_1} \mathbb{G}_m^r\xrightarrow{\delta_2} \mathbb{G}_m^r \to \Pic^0D \to 0,$$
where $\delta_1(t) = (t, \dots, t)$.
Concretely, given the line bundles $L_i=L|\widetilde{D}_i $ on $\widetilde{D}_i $ and explicit isomorphisms $\sigma_i \colon L_i \cong \scrO_{\widetilde{D}_i }$, and denoting the marked points $0$ and $\infty$ in $\widetilde{D}_i$ by $0_i$ and $\infty_i$ respectively, we recover the line bundle $L$  by giving the identifications $\mu_i \colon (L_i)_{0_i} \cong (L_{i-1})_{\infty_{i-1}}$. The trivializations $\sigma_i$ then identify both $(L_i)_{0_i}$ and $(L_{i-1})_{\infty_{i-1}}$ with $\Cee$ and $\mu_i$ with an element of $\mathbb{G}_m$. Replacing  $\sigma_i$ by $t_i\sigma_i$ for some $t_i\in \mathbb{G}_m$ replaces  $ \mu_i  \in \mathbb{G}_m$ by $t_i^{-1}t_{i-1}\mu_i$. Moreover,  $\delta_2(t_1, \dots, t_r) = (t_1^{-1}t_r, \dots, t_r^{-1}t_{r-1})$. If $p\colon \mathbb{G}_m^r\to \mathbb{G}_m$ is the homomorphism
$$p(\mu_1, \dots, \mu_r) = \mu_1\cdots \mu_r,$$
then it is easy to check that $p$ is surjective and 
$$\Ker p = \{(t_1^{-1}t_r, \dots, t_r^{-1}t_{r-1}): t_i\in \mathbb{G}_m\}=\im \delta_2,$$
 so that $p$ identifies $\Pic^0D$ with $\mathbb{G}_m$.
 
 Finally consider the line bundle $\scrO_D(q-p)$ as in the last statement of the lemma. Running through the procedure described above, and using the nowhere vanishing global section $s_i= (t-1)/(t-\lambda)$ of $L_i = \scrO_{\widetilde{D}_i}(q_i-p_i)$ to identify $L_i$ with $\scrO_{\widetilde{D}_i}$, it is easy to calculate that, in the above notation, $\mu_i=\lambda^{-1}$ and $\mu_k = 1$ for $k\neq i$. Thus $\psi(\scrO_D(q-p)) =\lambda^{-1}$.
\end{proof}

The proof of the following is straightforward and is left to the reader, and holds in the elliptic or triangle cases as well. 

\begin{lemma}\label{Loncycle} Let $D = D_1+\dots D_r$ be a cycle of rational curves and let $L$ be a line bundle on $D$ of multidegree $(e_1, \dots, e_r)$.
\begin{enumerate}
\item[\rm(i)] If $e_i \leq 0$ for all $i$, then $h^0(D;L) =0$ unless $L\cong \scrO_D$, in which case $h^0(D; L) = 1$.
\item[\rm(ii)] If $e_i \geq 0$ for all $i$ and $e_i >0$ for some $i$, then there exists a section $s\in H^0(D;L)$ such that, for every $i$, the restriction $s|D_i$ is not identically zero.
\item[\rm(iii)] If $e_i \geq 0$ for all $i$ and $\sum_ie_i \geq 2$, then $L$ has no base points, i.e.\ for every $p\in D$, there exists a section $s\in H^0(D;L)$ such that the image of $s$ in the fiber $L_p$ is not zero.
\item[\rm(iv)] If $e_i =1$ and $e_j = 0$ for $j\neq i$, then there exists a unique point $q\in D_i^{\text{\rm{int}}}$ such that $L \cong \scrO_D(q)$. In this case $h^0(D;L) =1$ and $q$ is the unique base point for $L$.
\item[\rm(v)] Fix $i$, $1\leq i\leq r$ and fix a point $p\in D_i^{\text{\rm{int}}}$. Then,   for 
 every line bundle $L$ on $D$ of multidegree $(0,\dots, 0)$ there exists a unique $q\in D_i^{\text{\rm{int}}}$ such that $L\cong \scrO_D(q-p)$.  \qed
\end{enumerate}
\end{lemma}

\begin{corollary}\label{Pic0toDint} Suppose that $p \in D_i^{\text{\rm{int}}}$. Define $\tau_p\colon \Pic^0D \to D_i^{\text{\rm{int}}}$ via: $\tau_p(L) = q$, where, for $L\in \Pic^0D$, we have  $L\otimes \scrO_D(p) = \scrO_D(q)$ for a unique point $q\in D_i^{\text{\rm{int}}}$ by {\rm{(iv)}} of Lemma~\ref{Loncycle}. Then $\tau_p$ is an isomorphism from $\Pic^0D$ to $D_i^{\text{\rm{int}}}$. 
\end{corollary}
\begin{proof} This is immediate from  (v)  of Lemma~\ref{Loncycle} (and the easily checked fact that $\tau_p$ is a morphism).
\end{proof}

One can also describe $\Pic^0D$ via divisors of multidegree $(0,\dots, 0)$ modulo principal divisors.  Let $\mathbf{d}=\sum_{i=1}^r\sum_{j=1}^{n_i}(q_{ij} -p_{ij})$ be a divisor of multidegree $(0,\dots, 0)$, where the $p_{ij}, q_{ij}\in D_i^{\text{\rm{int}}}$ for every $i$. Then $\mathbf{d}$ defines a line bundle $\scrO_D(\mathbf{d})$ of multidegree $0$ in the usual way:
$$\scrO_D(\mathbf{d}) = \scrO_D\left(\sum_{i=1}^r\sum_{j=1}^{n_i}(q_{ij} -p_{ij})\right).$$
The divisor $\mathbf{d}$ is \textsl{principal} if $\scrO_D(\mathbf{d}) \cong \scrO_D$ is the trivial line bundle.  

\begin{lemma}\label{princdiv} The map $\mathbf{d} \mapsto \scrO_D(\mathbf{d})$ is a surjective homomorphism from the group of divisors of multidegree $0$ supported on $D_{\text{\rm{reg}}}$ to $\Pic^0D$ and its kernel is the group of principal divisors. If $\mathbf{d}=\sum_{i=1}^r\sum_{j=1}^{n_i}(q_{ij} -p_{ij})$ is a divisor of multidegree $(0,\dots, 0)$ supported on $D_{\text{\rm{reg}}}$, then $\mathbf{d}$ is principal $\iff$ 
$$\prod_{i,j}\frac{p_{ij}}{q_{ij}} = 1.$$
\end{lemma}
\begin{proof} Clearly $\mathbf{d} \mapsto \scrO_D(\mathbf{d})$ is a homomorphism. It is surjective by (v) of Lemma~\ref{Loncycle} or directly, and by definition its kernel is the subgroup of principal divisors. The divisor $\sum_{i=1}^r\sum_{j=1}^{n_i}(q_{ij} -p_{ij})$ is principal $\iff$ there exist meromorphic functions $s_i$ on $D_i$ with $(s_i) = \sum_j(q_{ij} -p_{ij})$ and, for all $i\bmod r$, $s_i(0_i) = s_{i-1}(\infty_{i-1})$. Let $\tilde{s}_i$ be the meromorphic function on $D_i$ defined by
$$\tilde{s}_i(t) =\prod_j\frac{t-p_{ij}}{t-q_{ij}}.$$
Then $(\tilde{s}_i) = \sum_j(q_{ij} -p_{ij})$, and every meromorphic function on $D_i$ with this property is of the form $c_i\tilde{s}_i$ for some $c_i\in \Cee^*$. Moreover, $\tilde{s}_i(0_i) =\prod_jp_{ij}/q_{ij}$ and $\tilde{s}_i(\infty_i)=1$. The condition that there exist $c_i$ such that $c_i\tilde{s}_i(0_i) = c_{i-1}\tilde{s}_{i-1}(\infty_{i-1})$ for every $i$ is then equivalent to the condition that $\prod_{i,j} p_{ij}/q_{ij}  = 1$. 
\end{proof}

We turn next to automorphisms of $D$.

\begin{definition} Let $\Aut^0D$ be the neutral component of the group $\Aut D$ (i.e.\ the connected component of $\Aut D$ containing the identity).
\end{definition}

The following lists some properties of $\Aut^0D$ which will be needed in the proof of the Torelli theorem.

\begin{lemma}\label{autisom} With $\Aut^0D$ as above,
\begin{enumerate}
\item[\rm(i)] There is a canonical isomorphism $F\colon \Aut^0D \cong \mathbb{G}_m^r$. 
\item[\rm(ii)] $\Aut^0D$ acts simply transitively on $D_1^{\text{\rm{int}}} \times \cdots \times D_r^{\text{\rm{int}}}$. 
\item[\rm(iii)] The action of $\Aut^0D$ on $\Pic^0D$ is trivial.
\item[\rm(iv)] Given $p_i\in D_i^{\text{\rm{int}}}$, $1\leq i\leq r$, then, with $\psi$ as in Lemma~\ref{Pic0isomG},
$$F(\phi) = (\psi(\scrO_D(-\phi(p_1)+ p_1)), \dots, \psi(\scrO_D(-\phi(p_r)+ p_r))).$$ 
\item[\rm(v)] With $\tau_p$ as in Corollary~\ref{Pic0toDint}, for all $\phi\in \Aut^0D$,
$$\tau_{\phi(p)}= \phi\circ \tau_p.$$ 
\end{enumerate}
\end{lemma}
\begin{proof}(i) and (ii) are clear from the fact that the automorphism group  of $\Pee^1$ fixing $0$ and $\infty$ is canonically identified with $\mathbb{G}_m$ with the usual action of $\mathbb{G}_m$ on $\Pee^1-\{0, \infty\} \cong  \mathbb{G}_m$ by left multiplication. (iii) follows from the explicit description of $\Pic^0D$ given in Lemma~\ref{Pic0isomG}. To see (iv), suppose that $F(\phi) = (\lambda_1, \dots, \lambda_r)$, i.e.\ that $\phi$ corresponds to multiplication by $\lambda_i$ on $D_i^{\text{\rm{int}}}$. We can choose coordinates in each $D_i^{\text{\rm{int}}}$ so that $p_i$ corresponds to $1$,  and hence $\phi(p_i)$ corresponds to $\lambda_i$. By Lemma~\ref{Pic0isomG}, 
$$(\psi(\scrO_D(\phi(p_1)- p_1)), \dots, \psi(\scrO_D(\phi(p_r)- p_r))) = (\lambda_1^{-1}, \dots, \lambda_r^{-1}),$$
and  hence  $(\psi(\scrO_D(-\phi(p_1)+ p_1)), \dots, \psi(\scrO_D(-\phi(p_r)+ p_r))) = (\lambda_1  , \dots, \lambda_r )= F(\phi)$ as claimed.

Finally, with $p, q\in D_i^{\text{\rm{int}}}$ and $L\in \Pic^0D$, $\tau_p(L) = q$ $\iff$ $L\cong \scrO_D(q-p)$ $\iff$ $\phi(L) = \scrO_D(\phi(q)-\phi(p))$ $\iff$ $\tau_{\phi(p)}(\phi(L)) = \phi(q)$, where $\phi(L) = (\phi^*)^{-1}(L)$. Since $\phi(L) =L$ by (iii), this says that $\tau_{\phi(p)}= \phi\circ \tau_p$.
\end{proof}

We describe the relationship between $\Pic^0D$ and the generalized Jacobian $J(D)$. 
There is a unique $\omega_0 \in H^0(D;\omega_D)$ whose residue at the point $0_i\in D_i$ is $-1/2\pi \sqrt{-1}$, and hence whose  residue at the point $\infty_i\in D_i$ is $1/2\pi \sqrt{-1}$.  If $\sigma$ is a closed curve in $\bigcup_iD^{\text{int}}_i=D_{\text{\rm{reg}}}$, then $\displaystyle \phi \mapsto \int_\sigma \phi$ defines an element $p_\sigma$ of $H^0(D; \omega_D)\spcheck$, and every $p_\sigma$ is an integral multiple of $p_{\mathbf{c}}$, where $\mathbf{c}$ is a simple closed curve in $D_i^{\text{int}}$ for some fixed $i$ whose winding number around $\infty_i$ is $1$. We define $J(D) = H^0(D; \omega_D)\spcheck/\Zee p_{\mathbf{c}} \cong \Cee^*$. Note that $p_{\mathbf{c}}(\omega_0) = 1$, and hence under the identification of $H^0(D; \omega_D)\spcheck$ with $\Cee$ given by the choice of $\omega_0$, $p_{\mathbf{c}}$ is identified with $1\in \Cee$ and $\Zee p_{\mathbf{c}}$ with $\Zee \subseteq \Cee$. Then via $\exp(2\pi \sqrt{-1}\cdot )$, $J(D)$ is identified with $\Cee^* =\mathbb{G}_m$; we denote this explicit isomorphism by $\psi'$. 
 
\begin{definition}\label{defAbelJac}
Define the \textsl{Abel-Jacobi map} $\alpha$ from divisors of multidegree $(0, \dots, 0)$ on the smooth locus of $D$ to $J(D)\cong \Cee^*$ as follows: if $\mathbf{d}$ is a divisor of multidegree $(0, \dots, 0)$ on $D_{\text{\rm{reg}}}$, let
$$\alpha(\mathbf{d}) =\exp\left(2\pi \sqrt{-1}\int_\mathbf{d}\omega_0 \right).$$
Here the notation $\displaystyle\int_\mathbf{d}$ means that, if $\mathbf{d}=\sum_{i=1}^r\sum_{j=1}^{n_i}(q_{ij} -p_{ij})$, then we choose paths $\sigma_{ij}$ in $D_i^{\text{\rm{int}}}$ from $p_{ij}$ to $q_{ij}$, and set 
$$\int_\mathbf{d}\omega_0 = \int_\gamma\omega_0,$$
where $\gamma = \sum_{i,j}\sigma_{ij}$. For example, if $\lambda_1, \lambda_2$ are two points in $D_i^{\text{\rm{int}}}\cong \Cee^*$, then
$$\int_{\lambda_2-\lambda_1}\omega_0 = -\frac{1}{2\pi \sqrt{-1}}\int_{\lambda_1}^{\lambda_2}\frac{dz}{z} = -\frac{1}{2\pi \sqrt{-1}}(\log \lambda_2 - \log \lambda_1).$$
In particular, $\alpha(t\lambda_2-t\lambda_1) = \alpha(\lambda_2- \lambda_1)=\lambda_1/\lambda_2$ and $\displaystyle\alpha(\lambda-1) =  \lambda^{-1}$. It is easy to see that $\alpha$ is independent of the orderings of the $p_{ij}$ and $q_{ij}$ as well as the choices of the paths $\sigma_{ij}$.     
\end{definition}

\begin{lemma}\label{Lemma1.12} The homomorphism $\alpha$  induces an isomorphism $\bar{\alpha} \colon \Pic^0D \to J(D)$. Moreover, the following diagram commutes:
$$\begin{CD}
\Pic^0D @>{\bar\alpha}>> J(D) \\
@VV{\psi}V @V{\psi'}VV\\
\Cee^* @= \Cee^*, 
\end{CD}$$
 where $\psi \colon \Pic^0D \cong \mathbb{G}_m=\Cee^*$ is the isomorphism given in Lemma~\ref{Pic0isomG} and  $\psi'\colon J(D) \to \Cee^*$ is the isomorphism described above. 
 \end{lemma}
 \begin{proof} From the explicit calculations above, if $\mathbf{d}=\sum_{i=1}^r\sum_{j=1}^{n_i}(q_{ij} -p_{ij})$, then $\alpha(\mathbf{d}) = \prod_{i,j}p_{ij}/q_{ij}$. By Lemma~\ref{princdiv},  $\mathbf{d}$ is principal $\iff$ $\alpha(\mathbf{d}) = 1$. It follows that $\alpha$  induces a  homomorphism $\bar{\alpha} \colon \Pic^0D \to J(D)$. The fact that $\bar{\alpha}$ is an isomorphism and the commutativity of the diagram then follow from  (v)  of Lemma~\ref{Loncycle} and Lemma~\ref{Pic0isomG}.
 \end{proof}

\section{Moduli: birational geometry}

Let $(Y,D)$ be an anticanonical pair. We begin this section by describing the possibilities for an exceptional curve on $Y$:

\begin{definition}\label{defcurves} An irreducible curve $E$ on $Y$ is  an \textsl{interior exceptional curve} if $E\cong \Pee^1$, $E^2 = -1$, and $E \neq D_i$ for any $i$. Every exceptional curve on $Y$ is thus either an interior exceptional curve or a component of $D$. 
\end{definition}

For $p\in D$, let $\widetilde{Y}$ be the blowup of $Y$ at $p$. If $p$ is a smooth point of $D$, let $\widetilde{D} = \sum _i\widetilde{D}_i$, where  $\widetilde{D}_i$ is the proper transform of $D_i$. If $p$ is a singular point of $D$,  define $\widetilde{D} = \sum _i\widetilde{D}_i+E$, where  $\widetilde{D}_i$ is the proper transform of $D_i$ and $E$ is the exceptional divisor. Then $(\widetilde{Y}, \widetilde{D})$ is again an anticanonical pair. If $p$ is a  smooth point of $D$, we say that $(\widetilde{Y}, \widetilde{D})$ is an \textsl{interior blowup} of $(Y,D)$, and if $p$ is a singular point of $D$, we say that $(\widetilde{Y}, \widetilde{D})$ is a  \textsl{corner blowup} of $(Y,D)$. The following lemma describes how the self-intersection sequence and the charge are affected by a blowup.

\begin{lemma}\label{effectofblowup}
Let $(Y,D)$ be a labeled anticanonical pair with self-intersection sequence $(d_1, \dots, d_r)$ and charge $Q(Y,D)$.
\begin{enumerate}
\item[\rm(i)] If $\widetilde{Y}$ is an interior blowup of $Y$ at the point $p\in D_i^{\text{\rm{int}}}$, then $r(\widetilde{D})=r(D)$, and, under the natural labeling of  $\widetilde{D}$, the self-intersection sequence of $(\widetilde{Y}, \widetilde{D})$ is $(d_1, \dots, d_{i-1}, d_i-1, d_{i+1}, \dots, d_r)$.
\item[\rm(ii)] If $\widetilde{Y}$ is a corner blowup of $Y$ at the point $p\in D_i\cap D_{i+1}$, then $r(\widetilde{D})=r(D)+1$. If $r(D) =1$, i.e.\ $D$ is irreducible, then the self-intersection sequence of $(\widetilde{Y}, \widetilde{D})$ is $(d_1-4, -1)$. If $r(D) \geq 2$ and for an appropriate labeling of the components of $\widetilde{D}$,  the self-intersection sequence of $(\widetilde{Y}, \widetilde{D})$ is 
$$(d_1, \dots, d_{i-1}-1,  -1, d_{i+1}-1, \dots, d_r).$$
\item[\rm(iii)]  $Q(\widetilde{Y}, \widetilde{D}) = Q(Y,D) + 1$ if $\widetilde{Y}$ is an interior blowup of $Y$ and $Q(\widetilde{Y}, \widetilde{D}) = Q(Y,D)$ if $\widetilde{Y}$ is a corner blowup of $Y$. \qed
 \end{enumerate}
 \end{lemma}

To reverse the process of blowing up a point of the anticanonical divisor, let $(Y,D)$ be an anticanonical pair and let $E$ be an exceptional curve  on $Y$.  Then  either $E$ is an interior exceptional curve or $E$ is a component of $D$. Let $(\overline{Y}, \overline{D})$ be the pair resulting from blowing down $E$. It is easy to see that $(\overline{Y}, \overline{D})$ is an anticanonical pair. If $E$ is an interior exceptional curve, then $(Y,D)$ is an interior blowup of $(\overline{Y}, \overline{D})$, and if $E$ is a component of $D$ then $(Y,D)$ is a corner blowup of $(\overline{Y}, \overline{D})$.  We call the pair $(\overline{Y}, \overline{D})$ the \textsl{interior blowdown},  resp.\  the \textsl{corner blowdown}
 of $(Y,D)$.

In particular, we see that  interior exceptional curves arise by taking an  anticanonical pair $(\overline{Y}, \overline{D})$ and blowing up smooth points on $\overline{D}$. Blowing up distinct points on $\overline{D}$ leads to distinct interior exceptional curves. If however we blow up infinitely near points, in other words blow up $p\in \overline{D}_i$, then blow up the intersection point of the new exceptional curve $E$ with the proper transform of $\overline{D}_i$, and repeat $b-1$ times, we are led to the definition of a generalized exceptional curve:

\begin{definition}\label{defgenexcep} A \textsl{generalized exceptional curve} on $Y$ is a divisor $C_1+ \cdots + C_{b-1}+E$,  where $b\geq 1$, $C_i\cong \Pee^1$ is a smooth curve of self-intersection $-2$ disjoint from $D$ (a $-2$-curve in the terminology of Definition~\ref{defminustwo}) for $i\leq b-1$, $E$ is an interior exceptional curve,  $C_i \cdot C_j = 1$ if $j=i\pm 1$ and $0$ otherwise, and $E\cdot C_{b-1}= 1$ and $E\cdot C_i = 0$, $i\neq b-1$. 

If $C_1+ \cdots + C_{b-1}+E$ is a generalized exceptional curve, then so is  $E_i = C_i+ \cdots + C_{b-1}+E$ for $1\leq i \leq b$. Moreover, suppose that $(\overline{Y}, \overline{D})$ is an anticanonical pair, $p\in \overline{D}$ is a smooth point, and that $(Y,D)$ is obtained from $(\overline{Y}, \overline{D})$ by making  $b$ infinitely near blowups at $p$, in other words, that we blow up $p$, then the intersection of the new exceptional divisor with the proper transform of $\overline{D}$, and so on. Denote the sequence of blowups by
$$Y = Y_b \xrightarrow{\pi_b} Y_{b-1}\xrightarrow{\pi_{b-1}} \dots \xrightarrow{\pi_1} Y_0= \overline{Y}.$$
 Then the exceptional divisor of the birational morphism $Y \to \overline{Y}$ is a generalized exceptional curve, the divisors $E_i$ are also generalized exceptional curves, $E_i^2 = -1$, $E_i \cdot E_j = 0$ for $i\neq j$, and the $E_i$ are the pullbacks to $Y$ of the exceptional divisors of the $\pi_i$.   
\end{definition}

Clearly, every anticanonical pair $(Y,D)$ dominates an anticanonical pair $(\overline{Y}, \overline{D})$, where $\overline{Y}$ is a minimal rational surface; we will call such a pair \textsl{minimal}. Minimal pairs  $(\overline{Y}, \overline{D})$ are easy to classify (Lemma 3.2 of \cite{FriedmanMiranda}):

\begin{theorem}\label{minimalist} Let $(\overline{Y}, \overline{D})$ be a minimal anticanonical pair. Then exactly one of the following holds:
\begin{enumerate} 
\item[\rm(i)] $\overline{Y} \cong \Pee^2$, and $D$ is either three lines in general position, a line and a conic meeting transversally, or an irreducible nodal cubic. Equivalently, the possible self-intersection sequences are $(1,1,1)$, $(1,4)$, or $(9)$. The corresponding values of $Q(\Pee^2,D)$ are: $Q(\Pee^2,D) = 0$ if $D$ has self-intersection sequence  $(1,1,1)$; $Q(\Pee^2,D) = 1$ if $D$ has self-intersection sequence  $(1,4)$; $Q(\Pee^2,D) = 2$ if $D$ has self-intersection sequence  $(9)$.
\item[\rm(ii)] $\overline{Y} \cong  \mathbb{F}_N$, $N\neq 1$, and $D$ is the union of the negative section $\sigma_0$ and
\begin{enumerate}  
\item[\rm(a)] A section $\sigma$ with $\sigma^2 = N$ and two fibers $f_1$ and $f_2$. In this case, the self-intersection sequence is $(-N, 0, N,0)$ and  $Q(\mathbb{F}_N,D) = 0$.
\item[\rm(b)]  A section $\sigma$ with $\sigma^2 = N+2$ meeting $\sigma_0$ transversally and one fiber $f$ not passing through the intersection points of $\sigma_0$ and $\sigma$. In this case, the  self-intersection sequence is $(-N, N+2,0)$ up to orientation and $Q(\mathbb{F}_N,D) = 1$.
\item[\rm(c)]  A section $\sigma$ with $\sigma^2 = N+4$ meeting $\sigma_0$ transversally. In this case, the self-intersection sequence is $(-N, N +4)$ and $Q(\mathbb{F}_N,D) = 2$.
\end{enumerate}
\item[\rm(iii)] $\overline{Y} \cong  \mathbb{F}_N$, $N=0,2$, and $D$ is either an irreducible nodal bisection of self-intersection $8$, with   $Q(\mathbb{F}_N,D) = 3$, or $D$ is the union of two sections of self-intersection $2$, with $Q(\mathbb{F}_N,D) = 2$. \qed
\end{enumerate}
\end{theorem}

\begin{remark} In Lemma 3.2 of \cite{FriedmanMiranda}, the non-minimal case of $\mathbb{F}_1$ with self-intersection sequence $(0,4)$ is strangely omitted.
\end{remark}

\begin{remark}\label{nonewint} Let $(Y,D) \to (\overline{Y}, \overline{D})$ be a corner blowup and let $E$ be an interior exceptional curve on $\overline{Y}$. Then the proper transform of $E$ is still an exceptional curve on $Y$. Put slightly differently, a corner blowdown cannot create any new interior exceptional curves.

In particular, if we start with $(Y,D)$ and serially contract interior exceptional curves until there are no longer any such, the only remaining exceptional curves will be components of the anticanonical cycle, and this statement will remain true after any sequence of contracting such components. Hence every anticanonical pair $(Y,D)$  arises from a minimal pair $(\overline{Y}, \overline{D})$, not necessarily in a unique way, by first making an iterated series of corner blowups of $(\overline{Y}, \overline{D})$, and then making a series of interior blowups. 
\end{remark}

Instead of using a minimal model, we can also consider blowups $(Y,D) \to (\overline{Y}, \overline{D})$, where $(\overline{Y}, \overline{D})$ is toric, i.e.\  $T= \overline{Y}-\overline{D}\cong \mathbb{G}_m^2$ is an algebraic torus and multiplication of $T$ on itself extends to an action of $T$ on $\overline{Y}$. Note that the only toric examples in Theorem~\ref{minimalist} are $\overline{Y} \cong\Pee^2$ with $D$ equal to three lines in general position in Case (i), or $\overline{Y} \cong  \mathbb{F}_N$, $N\neq 1$, with $D =\sigma_0  +f_1+ \sigma + f_2$  in Case (ii)(a). Clearly, a toric anticanonical pair has no interior exceptional curves, a corner blowup or blowdown of a toric pair is again toric, and every toric anticanonical pair $(\overline{Y}, \overline{D})$ arises by making a series of corner blowups starting with one of the minimal examples described above. 

\begin{lemma}\label{Qzero}
If $(Y,D)$ is an anticanonical pair, then  $Q(Y, D) \geq 0$ and $(Y,D)$ is toric $\iff$ $Q(Y,D) = 0$.  \end{lemma}
\begin{proof} By inspection, for the pairs $(Y,D)$ of Theorem~\ref{minimalist},  $Q(Y, D) \geq 0$ and $(Y,D)$ is toric $\iff$ $Q(Y,D) = 0$. The result then follows from Lemma~\ref{effectofblowup} and the fact that, if $(Y,D)$ is toric, then $Q(Y,D) = \chi_{\text{top}}(Y-D) = 0$.
\end{proof}

\begin{remark} Another characterization of toric pairs is as follows: If $s$ is the rank of the kernel of the homomorphism $\bigoplus_i\Zee[D_i] \to  H^2(Y; \Zee) $ as in Lemma~\ref{rankQ}, then $s\leq 2$, and $s=2$ $\iff$ $(Y,D)$ is toric. See also Lemma~\ref{adefthcor}(iv)  for another interpretation  of the invariant $s$. 
\end{remark}

\begin{definition} An anticanonical pair $(Y,D)$ is \textsl{combinatorially toric} if there exists a toric anticanonical pair $(\overline{Y}, \overline{D})$ of the same length such that $D_i^2 = (\overline{D}_i)^2$ for all $i$ (with our usual convention that $D_i$ only meets $D_{i\pm 1}$, and similarly for $\overline{D}_i$).  
\end{definition}

\begin{lemma} Let $(Y,D)$ be a combinatorially toric anticanonical pair. Then $(Y,D)$ is toric. Moreover, $(Y,D)$ is obtained from $(\mathbb{F}_0, D)$ by a sequence of corner blowups and blowdowns, where $D = f_1+\sigma_1 + f_2 + \sigma_2$ as in Case (ii)(a) of Theorem~\ref{minimalist}.
\end{lemma}
\begin{proof} If $(Y,D)$ is combinatorially toric, then $Q(Y,D) =0$, and hence $(Y,D)$ is toric by Lemma~\ref{Qzero}. After a sequence of corner blowdowns, either $Y=\mathbb{F}_N$  with self-intersection sequence $(-N ,  0 , N, 0)$, or  $Y=\Pee^2$ and the self-intersection sequence is $(1,1,1)$. In the first case, after a sequence of corner blowups and blowdowns (viewed as elementary transformations of a ruled surface), we can assume $N =0$ and $Y\cong \mathbb{F}_0 =\Pee^1\times \Pee^1$. In the second case, after a corner blowup, the self-intersection sequence becomes $(-1,0,1,0)$ and after a further corner blowup and blowdown the sequence becomes $(0,0,0,0)$, and again we are in the case of the toric pair $(\mathbb{F}_0, D)$.
\end{proof}

\begin{remark} 
The proof shows that a combinatorially toric anticanonical pair is specified by its self-intersection sequence. However, two general anticanonical pairs with the same self-intersection sequence need not be deformation equivalent.
\end{remark}

We also have the following \cite[Proposition 1.3]{GHK0}:

\begin{lemma}\label{toricmodels} Let $(Y,D)$ be an anticanonical pair. Then there exists an anticanonical pair $(\widetilde{Y}, \widetilde{D})$ obtained from $(Y,D)$ by a sequence of corner blowups, and a morphism $\pi\colon (\widetilde{Y}, \widetilde{D}) \to (\overline{Y}, \overline{D})$, where $(\overline{Y}, \overline{D})$ is toric and $\pi$ is a sequence of interior blowups at not necessarily distinct points of $\overline{D}$.
\end{lemma}
\begin{proof}
Using Remark~\ref{nonewint}, it is enough to check for the minimal examples of Theorem~\ref{minimalist}, for which this is tedious but straightforward. For example, let $(Y,D) =(\Pee^2, D)$, where $D$ is an irreducible nodal cubic. Blow up the node of $D$, and then the two intersection points of the exceptional curve with the proper transform of $D$. We obtain a new anticanonical pair $(\widetilde{Y}, \widetilde{D})$, with self-intersection sequence $(3, -1, -3, -1)$. Also, the proper transforms of the tangent lines to the two branches of $D$ at the node are exceptional curves on $\widetilde{Y}$ meeting the two components of $\widetilde{D}$. Blowing these down leads to  self-intersection sequence $(3,0, -3, 0)$, so that, by the previous lemma, $(\widetilde{Y}, \widetilde{D})$ is an interior blowup of the toric pair $(\mathbb{F}_3, \sigma_0 + f_1 + \sigma + f_2)$ as in Case (ii)(a) of Theorem~\ref{minimalist}. The cases where $Y=\mathbb{F}_N$ and the self-intersection sequence is $(-N, N+2,0)$ or $(-N, N+4)$ are similar, by blowing up $D$ at the one or two points of $\sigma_0\cap \sigma$, in the notation of Theorem~\ref{minimalist} and noting that the proper transforms of the fibers passing through $\sigma_0\cap \sigma$ are exceptional. Finally, for the case $Y\cong \Pee^2$, with self-intersection sequence $(1,4)$, after blowing up the points of intersection of the line and the conic, the self-intersection sequence becomes $(-1, -1, -1, 2)$ and hence we get an anticanonical cycle which is a corner blowup of the case $Y=\mathbb{F}_0$ with self-intersection sequence $(0,0,2)$, so we are done by the previous case.
\end{proof}

\begin{corollary}\label{blowupanddown} Let $(Y,D)$ be an anticanonical pair. Then there exists a sequence of anticanonical pairs $(Y_i, D_i)$, $0\leq i\leq 2n$, such that $(Y_0,D_0) = (Y,D)$ and such that $(Y_{2n}, D_{2n}) = (\mathbb{F}_0, D)$ is toric, with $D = f_1+\sigma_1+f_2+\sigma_2$, and morphisms
$$\begin{matrix}
 (Y_0,D_0) &{}&{}&{}& \dots       &{} &{}&{}  &(Y_{2n},D_{2n})\\
{}&\searrow &{} &\swarrow&{}     &\searrow& {}&\swarrow &{} \\
 {}&{}&(Y_1,D_1)& {} &{}       &{}&(Y_{2n-1},D_{2n-1}) {}&{}&
 \end{matrix}$$
 where each $(Y_{2i}, D_{2i}) \to (Y_{2i+1}, D_{2i+1})$ is the identity, a corner blowdown, or an interior blowdown, and each $(Y_{2i}, D_{2i}) \to (Y_{2i-1}, D_{2i-1})$ is either the identity or a corner blowdown. \qed
 \end{corollary}
 
As an alternative to toric pairs, one can also consider pairs satisfying the following condition:
 
 \begin{definition} A (labeled) anticanonical pair $(Y,D)$ is \textsl{taut} if, given any labeled anticanonical pair $(Y',D')$ with the same self-intersection sequence as $D$, there is an isomorphism of (labeled) pairs $\phi\colon (Y', D') \to (Y, D)$ such that  $\phi(D_i') = D_i$. 
 \end{definition}
 
 For example, a toric anticanonical pair is taut. On the other hand, $Y=\mathbb{F}_0$ with self-intersection sequence $(0,4)$ is not taut, since there exists an anticanonical divisor on $\mathbb{F}_1$ which also has self-intersection sequence $(0,4)$.  Nonetheless, up to isomorphism, there is a unique anticanonical pair $(Y,D)$ with self-intersection sequence $(0,4)$ such that  $Y =\mathbb{F}_0$.

 \begin{lemma}\label{taut}
 \begin{enumerate}
 \item[\rm(i)] If $(Y,D)$ is taut and $(Y,D) \to (\overline{Y}, \overline{D})$ is an interior blowdown, then $(\overline{Y}, \overline{D})$ is taut.
 \item[\rm(ii)] If $(Y,D) \to (\overline{Y}, \overline{D})$ is a corner blowdown, then $(Y,D)$ is taut $\iff$ $(\overline{Y}, \overline{D})$ is taut.
 \item[\rm(iii)] The minimal models of Theorem~\ref{minimalist} which are not taut are exactly the ones given in Case {\rm(iii)} or $\mathbb{F}_0$ with self-intersection sequence $(0,4)$.  
 \end{enumerate}
 \end{lemma}
 \begin{proof} (i) and (ii) are straightforward. (iii) is a standard if tedious argument using the description of the automorphism group of a rational ruled surface and the fact that, except in the cases mentioned, the self-intersection sequence forces $Y=\mathbb{F}_N$ with $N$ determined by the self-intersection sequence.
 \end{proof}
 
 \begin{remark}\label{blowdowntotaut} Suppose that $(\overline{Y}, \overline{D})$ is $\mathbb{F}_0$ or $\mathbb{F}_2$ and the self-intersection sequence is either $(8)$ or $(2,2)$. If $(Y,D)$ is either an interior blowup or a corner blowup of $(\overline{Y}, \overline{D})$, it is easy to check  that there is a blowdown of $Y$ to $\Pee^2$, and hence to a taut minimal pair. A similar result holds in case $(\overline{Y}, \overline{D})$ is $\mathbb{F}_0$ or $\mathbb{F}_1$ and the self-intersection sequence is $(0,4)$. As a consequence, every anticanonical pair $(Y,D)$ is the blowup of a taut minimal pair, except for the case where $Y$ is $\mathbb{F}_0$, $\mathbb{F}_1$,  or $\mathbb{F}_2$ and the self-intersection sequence is  $(8)$,   $(2,2)$, or $(0,4)$.
 \end{remark}
 
 \begin{remark} For the applications later in this paper, a  weaker notion of tautness will be sufficient. See Remark~\ref{tautremark} for more details. 
 \end{remark}
 
  \begin{remark} In the elliptic case, every $(Y,D)$ is either a blowup of $(\Pee^2, E)$, where $E$ is a smooth cubic curve in $\Pee^2$, or $Y\cong \mathbb{F}_0$ or $\mathbb{F}_2$. In particular, if $D^2 < 8$, for example if $D^2$ is negative, then $(Y,D)$ is  a blowup of $(\Pee^2, E)$.
  
  In the triangle case, there is a short list of minimal pairs $(Y,D)$ analogous to that given in Theorem~\ref{minimalist}, and every pair is an interior blowup of a pair on that list. Moreover, if $(Y,D)$ is negative definite, then $(Y,D)$ blows down to $(\Pee^2, \overline{D})$, where $\overline{D}$ is a cubic curve in $\Pee^2$ of the corresponding type (either cuspidal, a line and a conic meeting at a tacnode, or three concurrent lines). See for example the proof of \cite[Theorem I.1.1]{Looij}. 
 \end{remark}

 \section{Moduli: deformations and the period map}
 
 Let $S$ be a scheme, analytic space, or of the form $\Spec A$, where $A$ is an Artin local $\Cee$-algebra. In the usual way, a \textsl{family} of anticanonical pairs consists of a scheme or analytic space $\mathcal{Y}$ together with  a smooth proper morphism $\pi \colon  \mathcal{Y}\to S$ and a relative Cartier divisor $\mathcal{D}$ with normal crossings on $\mathcal{Y}$, such that $\pi|\mathcal{D}$ is locally trivial, $\mathcal{D} =\bigcup_{i=1}^r\mathcal{D}_i$, where $\mathcal{D}_i$ is smooth if $r >1$, and such that each geometric fiber is an anticanonical pair. We shall work in the analytic category and shall only consider the case where $S$ is connected. We assume that the local system $R^1(\pi|\mathcal{D})_*\Zee$ is trivial, i.e.\ is isomorphic to the constant local system $\underline{\Zee}$, which is automatic if $r\geq 3$. An orientation of one fiber then induces an orientation of all fibers, and we shall always assume that such an orientation has been chosen. Families of labeled pairs are defined similarly. A \textsl{deformation} of the pair (or labeled pair) $(Y,D)$ over $S$ is a family  of anticanonical pairs or labeled anticanonical pairs  over $S$ such that one fiber is isomorphic to $(Y,D)$.  Two anticanonical pairs $(Y,D)$ and $(Y', D')$, or two labeled pairs, are \textsl{deformation equivalent} if they are both isomorphic to fibers of a family of anticanonical pairs or labeled anticanonical pairs (over a connected, not necessarily irreducible base). Deformation equivalence is easily seen to be an equivalence relation. Note that, in (iii) of Theorem~\ref{minimalist}, the pairs $(\mathbb{F}_0, D)$ and $(\mathbb{F}_2, D)$ are deformation equivalent in case $D$ is irreducible, and similarly for the case where $D$ is a union of two sections of self-intersection $2$, and these are the only examples of deformation equivalent minimal pairs. A \textsl{deformation type} of anticanonical pairs, or labeled anticanonical pairs, is an equivalence class for the relation of deformation equivalence. Then:
 
 \begin{theorem} Given $(d_1, \dots, d_r)\in \Zee^r$, there are only finitely many deformation types of anticanonical pairs $(Y,D)$ with self-intersection sequence  $(d_1, \dots, d_r)$.
 \end{theorem}
 \begin{proof}  The proof is by induction on $-D^2$. Note that $-D^2\geq -9$, $-D^2=-9$ $\iff$ $Y\cong \Pee^2$, and $-D^2$ increases by $1$ after an interior or corner blowup. The theorem is clearly true for $-D^2=-9$. Proceeding by induction, given a pair $(Y,D)$ with self-intersection sequence  $(d_1, \dots, d_r)$, either $(Y,D)$ is obtained from a pair $(Y', D')$ by an interior blowup or a corner blowup, or $(Y,D)$ is minimal. Suppose that $(Y,D)$ is obtained from a pair $(Y', D')$ by an interior blowup at a point $p\in (D_i')^\text{int}$. Then $(Y', D')$ has self-intersection sequence  $(d_1, \dots, d_i-1, \dots, d_r)$. By the inductive hypothesis, there are only finitely many deformation types with this self-intersection sequence. Fixing one such type, all pairs $(Y,D)$ obtained by blowing up a point $p\in (D_i')^\text{int}$ for some pair $(Y', D')$ of the given type are deformation equivalent. A similar argument handles the case where $(Y,D)$ is obtained from a pair $(Y', D')$ by a corner blowup. Finally, if $(Y,D)$ is minimal, then either $Y\cong \Pee^2$, $Y\cong \mathbb{F}_0$ or $\mathbb{F}_2$, or $Y\cong \mathbb{F}_N$ where $N$ is determined by the self-intersection sequence. A glance at the list in Theorem~\ref{minimalist} shows that there are only finitely many possibilities for the deformation type of $(Y,D)$.
 \end{proof}
 
Next we recall some standard results in deformation theory (cf.\ \cite{Looij} for example). Let $\mathbf{Def}_{Y; D_1, \dots,D_r}$ be the deformation functor for the pair $(Y, D_1+\cdots +D_r)$ viewed as a normal crossing divisor in $Y$. Then the Zariski tangent space to $\mathbf{Def}_{Y; D_1, \dots,D_r}$ is $H^1(Y; T_Y(-\log D))$ and the corresponding obstruction space is $H^2(Y; T_Y(-\log D))$, where $T_Y(-\log D)$ is the sheaf of vector fields on $Y$ tangent to the smooth curves $D_i$, and is dual to $\Omega^1_Y(\log D)$, the sheaf of $1$-forms with logarithmic poles along the normal crossings divisor $D$. Moreover, $H^0(Y; T_Y(-\log D))$ is the Lie algebra of the group $\Aut(Y,D)$ of automorphisms of the pair $(Y,D)$. As a functor on germs of analytic spaces, $\mathbf{Def}_{Y; D_1, \dots,D_r}$ is represented by a Kuranishi family $(S,0)$, which is smooth of dimension $\dim H^1(Y; T_Y(-\log D))$ by (iii) of Lemma~\ref{adefthlemma} below. 
 
 \begin{lemma}\label{adefthlemma} With $T_Y(-\log D)$ and $\Omega^1_Y(\log D)$ as above,
 \begin{enumerate}
 \item[\rm(i)] $T_Y(-\log D)\cong \Omega^1_Y(\log D)$.
  \item[\rm(ii)] There is an exact sequence 
  $$0\to \Omega^1_Y \to \Omega^1_Y(\log D) \to \bigoplus_i\scrO_{D_i} \to 0.$$
 \item[\rm(iii)] $H^2(Y; T_Y(-\log D)) =0$. 
 \item[\rm(iv)] Let $\partial \colon \bigoplus _i\Cee[D_i] \to H^2(Y; \Cee)$ be the natural homomorphism sending $[D_i]$ to its class in $H^2(Y; \Cee)$.  Then
   $H^0(Y; T_Y(-\log D))  \cong \Ker \partial$ and  $H^1(Y; T_Y(-\log D)) \cong \operatorname{Coker} \partial$.
 \end{enumerate}
 \end{lemma}
 \begin{proof} (i): The pairing $\Omega^1_Y(\log D) \otimes \Omega^1_Y(\log D) \to \Omega^2_Y(\log D) = K_Y(D)\cong \scrO_Y$ is perfect, and hence the dual $T_Y(-\log D) =(\Omega^1_Y(\log D))\spcheck \cong \Omega^1_Y(\log D)$.
 
 \smallskip
 \noindent (ii): This is the usual Poincar\'e residue sequence. 
 
 \smallskip
 \noindent (iii), (iv): These follow by taking the long exact cohomology sequence associated to the short exact sequence in (ii) and the identification of the coboundary homomorphism $\partial \colon \bigoplus_iH^0(\scrO_{D_i}) = \bigoplus _i\Cee[D_i] \to H^1(Y; \Omega^1_Y) \cong H^2(Y;\Cee)$ with the fundamental class map.
 \end{proof} 
 
 \begin{corollary}\label{adefthcor} {\rm(i)}   $\mathbf{Def}_{Y; D_1, \dots,D_r}$ is unobstructed. 
 
 \smallskip
 \noindent {\rm(ii)} The dimension of $\mathbf{Def}_{Y; D_1, \dots,D_r}$, i.e.\ $\dim H^1(Y; T_Y(-\log D))$, is equal to    $\dim H^2(Y; \Cee)/\Cee[D_1] + \cdots + \Cee[D_r]=\operatorname{rank}\Lambda(Y,D)$.
 
  \smallskip
 \noindent {\rm(iii)}  $(Y,D)$ is infinitesimally rigid, i.e.\  $H^1(Y; T_Y(-\log D)) = 0$ $\iff$ the classes $[D_1], \dots, [D_r]$ span $H^2(Y; \Cee)$ $\iff$ $\Lambda(Y,D) = 0$.
 
  \smallskip
 \noindent {\rm(iv)} The dimension of the group of automorphisms of the pair $(Y,D)$ is $s= \dim \Ker \partial$. \qed
 \end{corollary}
 
 \begin{remark}\label{tautrigid} Let $(Y,D)$ be a taut pair. Then it is easy to see that $(Y,D)$ is infinitesimally rigid, although the converse does not always hold. Hence the classes $[D_1], \dots, [D_r]$ span $H^2(Y; \Cee)$, or equivalently, $\Lambda(Y,D) = 0$. This can also be checked directly via the classification in  \S2.
 \end{remark}

 We can also consider the functor $\mathbf{Def}_{Y; D}$, consisting of deformations of the pair $(Y,D)$, keeping $D$ as an effective Cartier divisor, but such that the corresponding deformations of $D$ are not necessarily locally trivial.
 
 \begin{proposition} The functor $\mathbf{Def}_{Y; D}$  is unobstructed and the natural morphism of functors $\mathbf{Def}_{Y; D} \to \mathbf{Def}_D$ is smooth, where $\mathbf{Def}_D$ is the functor of deformations of the singular curve $D$.
 \end{proposition}
 \begin{proof} Let $\mathcal{C}^\bullet$ be the complex
 $$T_Y \to N_{D/Y},$$
 where $T_Y$ is in degree $0$ and $N_{D/Y}$, the normal sheaf to $Y$ in $D$, is in degree one.
 By standard results (cf.\ \cite{Sernesi}) the Zariski tangent space to $\mathbf{Def}_{Y; D}$ is the hypercohomology group $\mathbb{H}^1(Y; \mathcal{C}^\bullet)$. A direct cocycle calculation shows that  the obstruction space to $\mathbf{Def}_{Y; D}$ is $\mathbb{H}^2(Y; \mathcal{C})$. Moreover, the Zariski tangent space to $\mathbf{Def}_D$ is $\mathbb{T}^1_D = \operatorname{Ext}^1(\Omega^1_D, \scrO_D)$, where $\Omega^1_D$ is the sheaf of K\"ahler differentials on $D$. Let $\mathcal{D}^\bullet$ be the complex
 $$T_Y|D \to N_{D/Y},$$
 the complex dual to the conormal complex $I_D/I_D^2 \to \Omega^1_Y|D$, which is a resolution of $\Omega^1_D$. Then $\operatorname{Ext}^1(\Omega^1_D, \scrO_D) = \mathbb{H}^1(D; \mathcal{D}^\bullet)$, and the map on tangent spaces corresponding to the map of functors $\mathbf{Def}_{Y; D} \to \mathbf{Def}_D$ is the natural restiction map $\mathbb{H}^1(Y; \mathcal{C}^\bullet) \to \mathbb{H}^1(D; \mathcal{D}^\bullet)$. Thus, to prove the proposition, we must show that $\mathbb{H}^1(Y; \mathcal{C}^\bullet) \to \mathbb{H}^1(D; \mathcal{D}^\bullet)$ is surjective and that $\mathbb{H}^2(Y; \mathcal{C}) = 0$.
 
 By definition, there is a short exact sequence
 $$0 \to T_Y(-D) \to \mathcal{C}^\bullet \to \mathcal{D}^\bullet \to 0,$$
 where $T_Y(-D)$ is in degree $0$. Thus there is a long exact sequence 
 $$\mathbb{H}^1(Y; \mathcal{C}^\bullet) \to \mathbb{H}^1(D; \mathcal{D}^\bullet) \to H^2(Y; T_Y(-D)) \to \mathbb{H}^2(Y; \mathcal{C}^\bullet) \to \mathbb{H}^2(D; \mathcal{D}^\bullet).$$
 Since the cohomology sheaf $\mathcal{H}^1(\mathcal{D}^\bullet)$ is supported at a finite number of points, $\mathbb{H}^2(D; \mathcal{D}^\bullet) =0$. Thus, to prove the proposition, it suffices to prove that $H^2(Y; T_Y(-D)) =0$. But, as $T_Y(-D) =T_Y\otimes K_Y$, $H^2(Y; T_Y(-D))$ is Serre dual to $H^0(Y; \Omega^1_Y) =0$. 
  \end{proof}
  
 Note that, if $(Y',D')$ is a deformation of an anticanonical pair in the above sense, then $D'$ remains an effective curve linearly equivalent to $-K_{Y'}$. The following is then essentially equivalent to a theorem of Karras-Laufer-Wahl (see for example \cite[Theorem 5.4]{Wahl}):
  
  \begin{corollary}\label{cor3.5} If $(Y,D)$ is an anticanonical pair and $T$ is a subset of $D_{\text{\rm{sing}}}$, there exists a deformation of the pair $(Y,D)$ over a smooth connected base, such that the general fiber $(Y',D')$ is an anticanonical pair where $D'$ is the smoothing of $D$ at the set $T$ of nodes, or is a smooth elliptic curve in case $T = D_{\text{\rm{sing}}}$. \qed
  \end{corollary}
  
  In particular, if $(Y,D)$ has self-intersection sequence $(d_1, \dots,  d_r)$ and, say, $r\geq 3$, then we can deform $(Y,D)$ to a pair $(Y', D')$ of length $r-1$ and self-intersection sequence $(d_1, \dots, d_{r-2}, d_{r-1} +d_r + 2)$ and charge $Q(Y', D') = Q(Y, D) +1$.
  
 We next define the period map: 

\begin{definition} Let $\Lambda=\Lambda(Y,D)$ be as in Definition~\ref{defLambda}. Fixing the identification $\psi \colon \Pic^0D \cong \mathbb{G}_m$ of Lemma~\ref{Pic0isomG} defined by the orientation of the cycle $D$, we define the \textsl{period homomorphism} $\varphi_Y \colon \Lambda \to \mathbb{G}_m$ via: if $\alpha \in \Lambda$ and $L_\alpha$ is the corresponding line bundle, then $\varphi_Y(\alpha) = \psi(L_\alpha|D)\in \mathbb{G}_m$. Clearly $\varphi_Y$ is a homomorphism. The \textsl{period map} is the map that assigns to the pair $(Y,D)$ the homomorphism $\varphi_Y\in \Hom(\Lambda, \mathbb{G}_m)$.
\end{definition}

\begin{remark}
Let $\pi\colon (\widetilde{Y}, \widetilde{D}) \to (Y,D)$ be a corner blowup. Then there is a natural identification of $\Lambda(Y,D)$ with $\Lambda(\widetilde{Y}, \widetilde{D})$ via $\pi^*$. Moreover, the surjective morphism $\pi\colon \widetilde{D} \to D$ which collapses the exceptional component induces an isomorphism $\pi^*\colon \Pic^0D \to \Pic^0\widetilde{D}$. The period homomorphisms for $(Y,D)$ and $(\widetilde{Y}, \widetilde{D})$ are then compatible, in the sense that $\pi^*\circ \varphi_Y = \varphi_{\widetilde{Y}}\circ \pi^*$.
\end{remark}

 To define the period map for a family, we need to trivialize the cohomology:

\begin{definition}\label{defmarking} Let $\pi\colon (\mathcal{Y}, \mathcal{D}) \to S$ be a family of anticanonical pairs over a reduced connected complex space $S$. A \textsl{marking} of the family $(\mathcal{Y}, \mathcal{D})$  is an isomorphism of local systems $\theta\colon R^2\pi_*\Zee \to \underline{\widehat{\Lambda}}$, where $\widehat{\Lambda}$ is a lattice isomorphic to $H^2(Y_s;\Zee)$ for some $s\in S$, and $\underline{\widehat{\Lambda}}$ is the trivial local system over $S$ with fiber $\widehat{\Lambda}$. Clearly, a marking exists $\iff$ the monodromy homomorphism of the family $(\mathcal{Y}, \mathcal{D})$ is trivial. A \textsl{marked family} $(\mathcal{Y}, \mathcal{D}, \theta)$ is then a family of anticanonical pairs together with the choice of a marking.

Given a marked family as above, fix a base point $s_0\in S$, and identify  $\Lambda(Y_s, D_s)$ with the fixed lattice $\Lambda(Y_{s_0}, D_{s_0})=\Lambda$. The \textsl{period map} $\Phi_S$ is then the function $\Phi_S\colon S \to  \Hom(\Lambda, \mathbb{G}_m)$ defined by $\Phi_S(s) =\varphi_{Y_s}$. 
\end{definition}

\begin{remark} As we shall see, we will need to impose a stronger condition (admissibility) on markings.
\end{remark} 

The period map is then holomorphic in the following sense.

\begin{theorem}\label{perholom} Let $S$ be a reduced connected analytic space and let $(\mathcal{Y}, \mathcal{D},\theta)$ be a marked family of anticanonical pairs with trivial monodromy. Then the period map $\Phi_S \colon S \to \Hom(\Lambda, \mathbb{G}_m)$ is then a holomorphic map. \qed
\end{theorem} 
\begin{proof} For a fixed $\alpha\in \Lambda$, let $L_\alpha$ be the corresponding line bundle over $Y_{s_0}$. After shrinking $S$, we may assume that $L_\alpha$ extends to a line bundle $\mathcal{L}_\alpha$ over $\mathcal{Y}$, and that $\mathcal{D} \cong S\times D$. Then the restriction $\mathcal{L}_\alpha|\mathcal{D} \cong S\times D$ defines a holomorphic map $S \to \Pic^0D\cong \mathbb{G}_m$. Doing this for a basis of $\Lambda$ shows that $\Phi_S$ is holomorphic.
\end{proof}

We shall sketch another proof of Theorem~\ref{perholom} below.

There is a dual description of the period homomorphism $\varphi_Y$ which explains the name ``period map." There is a unique section $\omega$ of $H^0(Y; \omega_Y) = H^0(Y; \Omega^2_Y(\log D))$ such that  $2\pi\sqrt{-1}\operatorname{Res}_D\omega=\omega_0$. Here $\operatorname{Res}_D\omega$ denotes the residue on $D$, and  $\omega_0$ is the unique section  of $H^0(D; \omega_D)$ described in \S1 whose residue at the point $0_i\in D_i$ is $-1/2\pi \sqrt{-1}$, and hence whose  residue at the point $\infty_i\in D_i$ is $1/2\pi \sqrt{-1}$. Given $\xi \in \Lambda$, we can represent $\xi$ by a class $\Sigma = \sum_k\pm [C_k]$, where the $C_k$ are smooth curves in $Y$ meeting $D$ transversally at distinct points of $D_{\text{reg}}$. Thus we may write the signed intersection of $\Sigma$ with $D_i$ as $\sum_{j=1}^{n_i}(q_{ij}-p_{ij})$. For each $i$ and $j$, join $p_{ij}$ to $q_{ij}$ by a (real) simple curve $\sigma_{ij} \subseteq D_i^{\text{int}}$; we can further assume that the $\sigma_{ij}$ are disjoint. If $\tau(\sigma_{ij})$ is the tube over $\sigma_{ij}$ in a small normal neighborhood of $D_i^{\text{int}}$, we can glue the $\tau(\sigma_{ij})$ into the complement in $\Sigma$ of small disks around the $p_{ij}$ and $q_{ij}$, to obtain a new cycle $\Sigma'$, contained in $Y-D$ and homologous to $\Sigma$ in $Y$  (since $\Sigma '-\Sigma$ is the boundary of the cylinders corresponding to the $\tau(\sigma_{ij})$). It is easy to see that 
$\displaystyle \int_{\Sigma'}\omega \bmod \Zee$ is well-defined (independent of the choices of the numberings  of the $p_{ij}$ and $q_{ij}$ and the choice of the $\sigma_{ij}$). In fact, if $\Sigma''$ is another cycle contained in $Y-D$ such that the image of $[\Sigma'']\in H_2(Y-D; \Zee)$ is equal to the image of $[\Sigma']$, then $[\Sigma'']-[\Sigma'] \in \Ker (H_2(Y-D; \Zee) \to H_2(Y;\Zee))$. It is easy to check that 
$$ \Ker (H_2(Y-D; \Zee) \to H_2(Y;\Zee)) =\Zee\cdot \gamma,$$
where $\displaystyle \int_\gamma \omega = 1$ (\cite[I (5.1)]{Looij} or \cite[\S1]{EngelFriedman}). Hence
$$\int_{\Sigma''}\omega - \int_{\Sigma'}\omega \in \Zee.$$

Thus we have a well-defined homomorphism $\varphi'_Y\colon \Lambda \to \Cee^*\cong \mathbb{G}_m$ defined by 
$$\xi\mapsto \exp\left(2\pi\sqrt{-1}\int_{\Sigma'}\omega \right).$$
According to Carlson's theory of extensions of mixed Hodge structures  \cite{Carlson}, the mixed Hodge structure on $H^2(Y-D)$ is classified by the homomorphism $\varphi'_Y$. Moreover we have the following:

\begin{proposition} $\varphi'_Y =  \varphi_Y$, so that the period homomorphism $\varphi_Y$ describes the mixed Hodge structure on $Y-D$. 
\end{proposition} 
\begin{proof} Since $\omega$ is holomorphic on $Y-D$, the restriction of $\omega$ to $C_k - D$ is $0$.  A standard Stokes' theorem calculation shows that, with $\xi\in \Lambda$ and $\Sigma'$ as above,
$$\int_{\Sigma'}\omega = 2\pi \sqrt{-1}\sum_{i,j}\int_{\sigma_{ij}}\operatorname{Res}_D\omega = \sum_{i,j}\int_{\sigma_{ij}}\omega_0,$$
in the notation of Definition~\ref{defAbelJac}.  Thus, if $\mathbf{d}$ is the divisor $\sum_{i,j}q_{ij}-p_{ij}$, then, by Lemma~\ref{Lemma1.12}
$$\int_{\Sigma'}\omega = \alpha(\mathbf{d}) = \bar{\alpha}(\scrO_D(\mathbf{d})) = \psi(\scrO_D(\mathbf{d})).$$
By definition,
$$\psi(\scrO_D(\mathbf{d})) =\psi(L_\xi|D) = \varphi_Y(\xi).$$
Thus $\varphi'_Y(\xi) =  \varphi_Y(\xi)$. 
\end{proof}  

The differential of the period map  has been computed in  \cite{Looij} (see also \cite{Fried2}). We give an alternate formulation, whose proof may be found in \cite{EngelFriedman}.  First, a calculation in local coordinates gives:

\begin{lemma}\label{locexactseq} Let $\nu\colon \widetilde D=\coprod_i\widetilde{D}_i \to Y$ be the composition of normalization and inclusion.  Then there is an exact sequence
$$0 \to \Omega^1_Y(\log D)(-D) \to \Omega^1_Y \to \nu_*\Omega^1_{\widetilde{D}} \to 0. \qed$$ 
\end{lemma}

Given a line bundle $L$ on $Y$ such that $\deg (L|D_i) = 0$ for every $i$,  the Chern class $c_1(L)\in H^1(Y;\Omega^1_Y)$ maps to $0$ in $\bigoplus _iH^1(\widetilde{D}_i; \Omega^1_{\widetilde{D}_i})$. It follows from the exact sequence
$$\bigoplus _iH^0(\widetilde{D}_i; \Omega^1_{\widetilde{D}_i})=0 \to H^1(Y;\Omega^1_Y(\log D)(-D)) \to H^1(Y; \Omega^1_Y) \to  \bigoplus _iH^1(\widetilde{D}_i; \Omega^1_{\widetilde{D}_i})$$
that $c_1(L)$ lifts to a unique element of $H^1(Y;\Omega^1_Y(\log D)(-D))$, which we denote by $\hat{c}_1(L)$.  Let $(S,0)$ be the Kuranishi family which represents the functor $\mathbf{Def}_{Y; D_1, \dots,D_r}$. The differential $(d\Phi_S)_0$ of the period map at the point $0$ corresponding to $(Y,D)$, which is defined invariantly as a homomorphism from $H^1(Y; T_Y(-\log D))$ to $\Hom(\Lambda, H^1(D; \scrO_D))$, is then described as follows:

\begin{theorem}\label{diffcomp} Let $\partial \colon H^1(D; \scrO_D) \to H^2(Y; \scrO_Y(-D))$ be the coboundary map arising from the exact sequence
$$0 \to \scrO_Y(-D) \to \scrO_Y \to \scrO_D \to 0,$$
which is an isomorphism since $H^1(Y; \scrO_Y) =  H^2(Y; \scrO_Y) = 0$.  Then, for all $\alpha \in \Lambda$ and $\xi\in H^1(Y; T_Y(-\log D))$, 
$$\partial \circ (d\Phi_S)_0 (\xi)(\alpha) = \xi \smallsmile \hat{c}_1(L_\alpha),$$
where $\smallsmile$ denotes the cup product 
$$H^1(Y; T_Y(-\log D))\otimes H^1(Y;\Omega^1_Y(\log D)(-D)) \to H^2(Y; \scrO_Y(-D)).\qed$$
\end{theorem}

\begin{corollary} $(d\Phi_S)_0$ is an isomorphism.
\end{corollary}
\begin{proof} This is clear since $\scrO_Y(-D) = \omega_D$ and by Serre duality.
\end{proof}
 
Note that 
\begin{align*}
&\Hom(H^1(Y;\Omega^1_Y(\log D)(-D)) , H^2(Y; \scrO_Y(-D))) \\
\cong&\Hom(H^2(Y; \scrO_Y(-D))\spcheck, (H^1(Y;\Omega^1_Y(\log D)(-D))\spcheck) \\
\cong&\Hom(H^0(Y; \Omega^2_Y(\log D)) , H^1(Y; \Omega^1_Y(\log D))),
 \end{align*}
where we have used Serre duality. Thus we obtain the usual description of the differential of the period map, which describes the variation of the line $F^2 = H^0(Y; \Omega^2_Y(\log D))$ in $H^2(Y-D; \Cee) = \mathbb{H}^2(Y; \Omega^\bullet_Y(\log D))$. The differential is then a homomorphism from $H^1(Y; T_Y(-\log D))$, the  tangent space of the Kuranishi family, to $\Hom(F^2,H^2(Y-D; \Cee)/F_2)$. Using  $F^2=H^0(Y; \Omega^2_Y(\log D))$ and  $H^2(Y-D; \Cee)/F_2 = H^1(Y; \Omega^1_Y(\log D))$, we then have: 

\begin{theorem} The differential of the period map at a point $(Y,D)$ is given by cup product:
$$H^1(Y; T_Y(-\log D)) \otimes H^0(Y; \Omega^2_Y(\log D)) \to H^1(Y; \Omega^1_Y(\log D)).$$
Since $\Omega^2_Y(\log D) = \Omega^2_Y(D) = \scrO_Y$ and contraction against a nowhere vanishing section of $\Omega^2_Y(\log D)$ induces an isomorphism of sheaves from $T_Y(-\log D)$ to $\Omega^1_Y(\log D)$, the differential of the period map is an isomorphism. \qed
\end{theorem}

In particular, the period map is locally an isomorphism. 
By \cite{Looij}, \cite{FriedmanScattone}, \cite{Fried2}, we have:

\begin{theorem}\label{surjper} The period map is surjective. More precisely, given $Y$ as above and given an arbitrary homomorphism $\varphi \colon \Lambda \to\mathbb{G}_m$, there exists a deformation of the pair $(Y,D)$ over a smooth connected base, which we can take to be a product of $\mathbb{G}_m$'s, such that the monodromy of the family is trivial and there exists a fiber of the deformation, say $(Y', D')$ such that, under the induced identification of $\Lambda(Y',D')$ with $\Lambda$, $\varphi_{Y'} = \varphi$. \qed
\end{theorem}

We sketch a proof of  Theorem~\ref{perholom} and Theorem~\ref{surjper}. Let $(Y,D)$ be an anticanonical pair. The idea is to construct a family $\pi\colon (\mathcal{Y}^{\text{u}}, \mathcal{D}^{\text{u}}) \to (\mathbb{G}_m)^N$ for some $N$, satisfying:
\begin{enumerate}
\item The pair $(Y,D)$ is a fiber of $\pi$ over the point $t_0\in (\mathbb{G}_m)^N$;
\item The monodromy of the family is trivial, so that in particular, for all $t\in (\mathbb{G}_m)^N$, we can identify all of the groups $\Lambda(Y_t, D_t)$ with $\Lambda =\Lambda(Y,D)$;

\item The family $\pi$ is versal at $t_0$, i.e.\   the Kodaira-Spencer map $T_{(\mathbb{G}_m)^N, t_0} \to H^1(Y; T_{Y}(-\log D))$ is surjective;
\item The period map $\Phi\colon (\mathbb{G}_m)^N \to \Hom (\Lambda, \mathbb{G}_m)$ is an affine map, i.e.\ the composition of a homomorphism and a translation.
\end{enumerate} 
To prove Theorem~\ref{perholom}, suppose that we are given a family $(\mathcal{Y}, \mathcal{D}) \to S$ as in the statement of the theorem, and a point $s\in S$ corresponding to a pair $(Y,D)$. By versality, there is a neighborhood $U$ of $s$ and a holomorphic map $f\colon U \to (\mathbb{G}_m)^N$ such that the family $(\mathcal{Y}, \mathcal{D})|U$ is the pullback of the corresponding family on $(\mathbb{G}_m)^N$. Then, in $U$, the period map is given by the composition  $\Phi\circ f$, which is clearly holomorphic. Since it is an elementary fact that an affine map between two algebraic tori whose differential is surjective at a point is surjective, the existence of the family $(\mathcal{Y}^{\text{u}}, \mathcal{D}^{\text{u}})$ gives a proof of Theorem~\ref{surjper}.

To construct the family, let us assume that   $(Y,D)$ is given by a sequence of interior blowups of a taut pair $(\overline{Y}, \overline{D})$, say at points $q_1, \dots, q_N \in \bigcup_iD_i^{\text{int}}$, possibly infinitely near. (The remaining case, Case (iii) of Theorem~\ref{minimalist}, where $Y =\mathbb{F}_0$ or $\mathbb{F}_2$, is easily handled by a direct argument.) Let $a\colon \{1, \dots, N\} \to \{1, \dots, r\}$ be the function defined by:  $q_i \in  D_{a(i)}^{\text{int}}$. Identifying $D_{a(1)}^{\text{int}} \times \cdots \times D_{a(N)}^{\text{int}}$ with $(\mathbb{G}_m)^N$ by choosing an origin in each $D_i^{\text{int}}$, we define the family $\pi\colon (\mathcal{Y}^{\text{u}}, \mathcal{D}^{\text{u}}) \to (\mathbb{G}_m)^N$ in the natural way, so that the fiber $(Y_t, D_t)$ over $t=(p_1, \dots, p_N)$ is the blowup of $(Y,D)$ at the points $p_1, \dots, p_N$. (Note: the biregular type of $\mathcal{Y}$ depends on the ordering of the points $p_i$.) Then $(\mathcal{Y}^{\text{u}}, \mathcal{D}^{\text{u}})$ clearly satisfies (1) and (2) above. Using the stability of exceptional divisors under small deformations, it is easy to see that $\pi$ is versal at some point (in fact, it is versal at every point, but fails to be locally semi-universal if the automorphism group of $(\overline{Y}, \overline{D})$ is positive-dimensional). 

The remaining point to check is (4). Fix $\alpha \in \Lambda$. Then we can write $\alpha = \delta + \sum_{i=1}^Nn_ie_i$, where $\delta$ is the pullback of a class in $H^2(\overline{Y}; \Zee)$ and the $e_i$ are the classes of the exceptional curves corresponding to the blowdown $(Y,D) \to (\overline{Y}, \overline{D})$. If $(Y,D)$ and $(Y_t, D_t)$ are as above, it is easy to see that 
$$\varphi_{Y_t}(\alpha) = \varphi_Y(\alpha)\otimes \scrO_D(\sum_in_i(p_i-q_i)).$$
Thus, $\varphi_{Y_t}(\alpha) = \varphi_Y(\alpha)$ up to tensoring with  $\scrO_D(\sum_in_i(p_i-q_i))$, which is the composition of a translation and a homomorphism $(\mathbb{G}_m)^N \to \Pic^0D$. Doing this for a basis $\alpha_1, \dots, \alpha_n$ of $\Lambda$ then establishes (4).

\begin{remark} We can define the period map in the elliptic and triangle cases as well. Let $(Y,D)$ satisfy either Case (i) or Case (iii) of the introduction, and let $\Lambda(Y,D)$ be the orthogonal complement in $H^2(Y; \Zee)$ of the classes of the components of $D$. In its simplest form, the period point of $(Y,D)$ is then the homomorphism $\psi_Y\colon \Lambda(Y,D) \to J^0(D)$ defined by:  $\psi_Y(\alpha) = L_\alpha|D$. Note however that, in the triangle case, $J^0(D) \cong \mathbb{G}_a$ but there exists a subgroup of $\Aut D$ isomorphic to $\mathbb{G}_m$ whose action on  $J^0(D)$ corresponds to the usual action of $\mathbb{G}_m$ on $\mathbb{G}_a$.
\end{remark}

\section{The ample cone and nef divisors}

\begin{definition}
Let $\mathcal{C}=\mathcal{C}(Y)$ be the positive cone of $Y$, i.e.\ 
$$\mathcal{C} = \{x\in H^2(Y; \Ar): x^2 >0\}.$$
 Then $\mathcal{C}$ has two components, and exactly one of them, say $\mathcal{C}^+=\mathcal{C}^+(Y)$, contains the classes of ample divisors. Let $\mathcal{A}(Y)$ be the ample (nef, K\"ahler) cone of $Y$ and let $\overline{\mathcal{A}}(Y)\subseteq  \mathcal{C}^+ \subseteq H^2(Y; \Ar)$ be the  closure of $\mathcal{A}(Y)$  in $\mathcal{C}^+$. By definition, $\overline{\mathcal{A}}(Y)$ is closed in $\mathcal{C}^+$ but not in general in $H^2(Y; \Ar)$.
\end{definition}

Let $\alpha \in  H^2(Y; \Zee)$. Then $\alpha$ defines the \textsl{oriented wall} $W^\alpha$, which by definition is the hyperplane $\{x\in H^2(Y; \Ar): x\cdot \alpha=0\}$, together with the preferred half-space $\{x\in H^2(Y; \Ar): x\cdot \alpha \geq 0\}$. If $\Omega\subseteq \mathcal{C}^+$ is a locally polyhedral convex set with nonempty interior, then $W^\alpha$ is a \textsl{face} of $\Omega$ if $\Omega \subseteq \{x\in H^2(Y; \Ar): x\cdot \alpha \geq 0\}$ and  $\overline{\Omega} \cap W^\alpha$ contains a nonempty open subset of $W^\alpha$, where  $\overline{\Omega}$ denotes the closure of $\Omega$ in $\mathcal{C}^+$. (Thus, in this paper, all faces have codimension one.) An \textsl{interior} point $x$ of $\Omega$ is then a point $x\in \Omega$ not lying on a face. 

By a standard result, the faces of $\overline{\mathcal{A}}(Y)$ correspond to curves of negative self-intersection on $Y$ (see for example \cite{Fried3}):

\begin{proposition}\label{constructnef} Let $X$ be a smooth projective surface and let $G_1, \dots , G_n$ be irreducible curves on $X$ such that the intersection matrix $(G_i\cdot G_j)$ is negative definite.
Then there exists a nef and big divisor $H$ on $X$ such that $H\cdot G_j = 0$ for all $j$ and, if $C$ is an irreducible curve such that $C \neq G_j$ for any $j$, then  $H\cdot C >0$. In fact, the set of nef and big $\Ar$-divisors on $X$ which are orthogonal to $\{G_1, \dots, G_n\}$ is a nonempty open subset of $\{G_1, \dots, G_n\}^\perp \otimes \Ar$. \qed
\end{proposition}

\begin{corollary} If $X$ is a smooth projective surface, then the faces of $\overline{\mathcal{A}}(X)$ are exactly the walls $W^\alpha$, where $\alpha = [C]$ is the class of an irreducible  curve of negative self-intersection. Two classes $\alpha_1=[C_1]$ and $\alpha_2 = [C_2]$  define the same face of $\overline{\mathcal{A}}(X)$ $\iff$ $C_1 = C_2$. Finally, if $W^{\alpha_1}, \dots , W^{\alpha_n}$ are two faces of $\overline{\mathcal{A}}(X)$ corresponding to distinct irreducible curves  $C_1, \dots , C_n$, then $W^{\alpha_1}\cap \cdots \cap W^{\alpha_n}\cap \overline{\mathcal{A}}(X)\neq \emptyset$ $\iff$  the intersection matrix $(C_i\cdot C_j)$ is negative definite. \qed 
\end{corollary}

Returning to an anticanonical pair $(Y,D)$, it is  easy to describe all of the curves of negative self-intersection on $Y$:

\begin{definition}\label{defminustwo}
An irreducible curve $C$ on $Y$ is  a \textsl{$-2$-curve}  if $C\cong \Pee^1$,  $C^2 = -2$, and $C \neq D_i$ for any $i$. Note that, if $C$ is a $-2$-curve, then $C \cap D =\emptyset$.
\end{definition}

An easy exercise in adjunction shows:

\begin{lemma} Let $(Y,D)$ be an anticanonical pair and let $C$ be an irreducible curve on $Y$. If $C^2 < 0$, then either $C$ is an interior exceptional curve, $C$ is a $-2$-curve, or $C$ is a component of $D$.
\qed
\end{lemma}

\begin{corollary} $\overline{\mathcal{A}}(Y)$ is the set of all $x\in \mathcal{C}^+$ such that  $x\cdot [D_i]\geq 0$, $x\cdot [E] \geq 0$ for all exceptional curves $E$  and $x\cdot [C] \geq 0$ for all $-2$-curves $C$. Moreover, if   $\alpha$ is the class associated to an exceptional or $-2$-curve, or $\alpha =[D_i]$ for some $i$ such that $D_i^2< 0$  then $W^\alpha$ is a face of $\overline{\mathcal{A}}(Y)$, and if $\alpha, \beta$ are two such classes, $W^\alpha = W^\beta$ $\iff$ $\alpha =\beta$.  \qed
\end{corollary}

A fact that we shall use repeatedly is:

\begin{lemma}\label{performinustwo} Let  $\beta =[C]$ be the class of a $-2$-curve on $Y$. Then $\varphi_Y(\beta) = 1$. 
\end{lemma}
\begin{proof} This is clear since $\scrO_Y(C)|D =\scrO_D$.
\end{proof}

The naive converse to Lemma~\ref{performinustwo}, that if $\beta \in \Lambda$ with $\beta^2=-2$ and $\varphi_Y(\beta) = 1$, then $\pm\beta$ is the class of a $-2$-curve on $Y$ (or a union of such) does not hold. The correct formulation of a converse to Lemma~\ref{performinustwo} is given in Theorem~\ref{mainprop}.

\begin{definition}\label{defgener} The pair $(Y,D)$ is \textsl{generic} if there does not exist a $-2$-curve on $Y$.
\end{definition}

 As a corollary of Lemma~\ref{performinustwo} and the local surjectivity of the period map, we have:
 
 \begin{corollary}\label{gener1} The image of the set of generic pairs $(Y,D)$ under the period map is contained in the complement in $\Hom(\Lambda, \mathbb{G}_m)$ of a countable union of proper subvarieties. Hence a very general small deformation of $(Y,D)$ is generic. \qed
 \end{corollary}

We now consider the behavior of nef divisors on $Y$, following \cite{Fried1} and Harbourne \cite{Harbourne}. The discussion here extends verbatim to the elliptic and triangle cases as well. The paper \cite{Harbourne} also deals with the case where $D\in |-K_Y|$ is not necessarily reduced.  

In general we will have to make the assumption that no component of $D$ is a fixed component of $L$. For nef and big divisors, this is a mild assumption:

\begin{lemma}\label{LrestrD} Let $L$ be a nef and big divisor on $Y$. Then $L$ is effective, i.e.\ $H^0(Y; L) \neq 0$,  and the restriction map
$$\rho\colon H^0(Y; L) \to H^0(D; L|D)$$
is surjective. Finally, $h^1(Y;L) = 1$ if $L|D =\scrO_D$ and $h^1(Y;L) = 0$ otherwise.
\end{lemma}
\begin{proof} The first statement follows from Riemann-Roch: since $L$ is nef, $H^2(Y; L) =0$ as it is Serre dual to $$H^0(Y; L^{-1} \otimes K_Y) =H^0(Y; L^{-1} \otimes \scrO_Y(-D)) = 0.$$
 Then $h^0(Y;L) -h^1(Y; L) = \frac12(L^2+L\cdot D) + 1>0$, and hence $h^0(Y;L)>0$.  To see the second statement, the cokernel of $\rho$ is contained in the group 
$$H^1(Y; L\otimes \scrO_Y(-D)) = H^1(Y; L\otimes K_Y),$$ which is Serre dual to $H^0(Y; L^{-1})$. By Ramanujam's vanishing theorem, since $L$ is nef and big, $H^1(Y; L^{-1}) =0$, and hence $\rho$ is surjective. Finally, from the exact sequence
$$0 = H^1(Y;L\otimes \scrO_Y(-D)) \to  H^1(Y;L) \to H^1(D; L|D) \to H^2(Y;L\otimes \scrO_Y(-D))$$
and the fact that $H^2(Y;L\otimes \scrO_Y(-D))$ is Serre dual to $H^0(Y; L^{-1})=0$, we see that $H^1(Y;L) \cong H^1(D; L|D)$. By Serre duality again, $H^1(D; L|D) \cong H^0(D; L^{-1}|D)$. But $L^{-1}|D$ has nonpositive degree on every component. By Lemma~\ref{Loncycle},  $h^0(D; L^{-1}|D) =1$ $\iff$ $L^{-1}|D =\scrO_D$ $\iff$ $L|D =\scrO_D$, and $h^0(D; L^{-1}|D) =0$ otherwise. Thus, the same holds for  $h^1(Y;L)$.
\end{proof}

\begin{corollary}\label{nocompD} If $L$ is nef and big and $\deg (L|D) > 0$, then no component of $D$ is a fixed component of $L$. If $L$ is nef and big and $\deg (L|D_i) =0$ for every $i$, then either $L|D \cong \scrO_D$, i.e.\ $\varphi_Y(L) = 1$, and there exists a section of $L$ which is nowhere vanishing along $D$, or $L|D$ is not trivial and every component of $D$ is a fixed component of $L$. 
\end{corollary}
\begin{proof} This is immediate from the previous lemma and Lemma~\ref{Loncycle}.
\end{proof}

\begin{theorem}\label{nefdivisors} Let $L$ be a nef and big line bundle on $Y$ and suppose that no component of $D$ is a fixed component of $L$. Then either $L$ has no fixed components or the possibilities for its fixed components as a divisor are given as follows:
\begin{enumerate}
\item[\rm(i)] The fixed component is a $-2$-curve $C$, and $L =\scrO_Y(kF+C)$, where $F$ is a smooth elliptic curve, $F\cdot D = 0$, and $k\geq 2$.
\item[\rm(ii)] The fixed component is a generalized exceptional curve in the sense of Definition~\ref{defgenexcep}, i.e.\ a chain of length $b\geq 1$ of curves $C_1, \dots, C_b$, where $C_i$ is a $-2$-curve for $i\leq b-1$, $C_b$ is an interior exceptional curve, and $C_i \cdot C_j = 1$ if $j=i\pm 1$ and $0$ otherwise. In this case, $L=\scrO_Y(G_0 + \sum_{i=1}^bC_i )$, where $G_0$ is either a smooth irreducible curve with $G_0^2 >0$, $G_0\cdot C_1 = 1$,   $G_0\cdot C_i =0$ for $i>1$, and  $G_0\cdot D =0$, or $G_0 = kF$, where $k\geq 1$, $F$ is a smooth elliptic curve, and $F\cdot C_1 = 1$, $F\cdot C_i  =0$ for $i>1$, and  $F^2=F\cdot D =0$.
\end{enumerate}
Finally, if $L$ has no fixed components, then $L$ has a base point $\iff$ $L\cdot D =1$, in which case the unique smooth point $p\in D$ such that $L|D \cong \scrO_D(p)$ is the unique base point of $L$. 
\end{theorem} 
\begin{proof} We begin with a series of lemmas on effective divisors on $Y$:

\begin{lemma}\label{firsttechlemma} Let $G$ be a nonzero effective divisor on $Y$.
\begin{enumerate}
\item[\rm(i)] $h^2(Y; \scrO_Y(G)) = 0$, and hence 
$$h^0(Y; \scrO_Y(G)) - h^1(Y; \scrO_Y(G)) = 1+ \frac12(G^2 + G \cdot D).$$
\item[\rm(ii)] If $G$ is reduced and connected, and $D\cap G$ is finite and nonempty, then $h^1(Y; \scrO_Y(G)) = 0$.
\item[\rm(iii)] If  $D\cap G = \emptyset$, then $h^1(Y; \scrO_Y(G)) \geq 1$ and $h^1(Y; \scrO_Y(G)) = 1$ if $G$ is reduced and connected.
\end{enumerate}
\end{lemma}
\begin{proof} (i) is clear since $h^2(Y; \scrO_Y(G)) = h^0(Y; \scrO_Y(-G-D)) =0$ and from Riemann-Roch. For (ii) and (iii), using the exact sequence
$$0 \to \scrO_Y \to \scrO_Y(G) \to \scrO_Y(G)|G \to 0,$$
we see that $h^1(Y; \scrO_Y(G)) = h^1(G; \scrO_Y(G)|G)$. By adjunction, $\scrO_Y(G)|G= \omega_G \otimes \scrO_Y(D)$, and hence, by Serre duality on $G$, 
$$h^1(G; \scrO_Y(G)|G) = h^0(G; \scrO_Y(-D)|G).$$
Under the hypotheses of (ii), $G$ is reduced and connected,   and $\scrO_Y(-D)$ has nonpositive degree on every component of $G$ and strictly negative degree on at least one component. Thus $h^0(G;\scrO_Y(-D)|G) =0$ and hence $h^1(Y; \scrO_Y(G)) =0$. If $D\cap G = \emptyset$, then $h^1(Y; \scrO_Y(G)) =h^0(G; \scrO_G) \geq 1$, and $h^1(Y; \scrO_Y(G)) =h^0(G; \scrO_G) =1$ if $G$ is reduced and connected.
\end{proof}

\begin{lemma}\label{secondtechlemma} Let $C$ be an irreducible curve on $Y$, $C\neq D$, such that $C^2 \geq 0$. Then the linear system $|C|$ has no fixed components. Moreover, $|C|$ has a base point $\iff$ $C$ is not a component of $D$ and  $C\cdot D =1$, and in this case $C\cap D$ is the unique base point of $|C|$.
\end{lemma}
\begin{proof} As $Y$ is regular, the  map $H^0(Y; \scrO_Y(C)) \to H^0(C; \scrO_Y(C)|C)$ is surjective. The lemma is then clear if $C$ is a smooth rational curve.  Thus we may assume that $p_a(C) \geq 1$ and that $C$ is not a component of $D$, i.e.\ that $C\cap D$ is finite. By adjunction, $\scrO_Y(C)|C =\omega_C \otimes \scrO_Y(D)|C$. By a general result on Gorenstein curves \cite{Catanese}, since $p_a(C) \geq 1$, the line bundle $\omega_C$ has no base points. Thus $|C|$ has no fixed components and the only possible base points are on $D$. If $C^2 >0$, then $C$ is nef and big, and hence the restriction map $H^0(Y; \scrO_Y(C)) \to H^0(D; \scrO_Y(C)|D)$ is surjective. In this case, by Lemma~\ref{Loncycle}(iv), $|C|$ has no base points on $D$ $\iff$ either  $C\cdot D = 0$, so that $\scrO_Y(C)|D=\scrO_D$, or $C\cdot D \geq 2$. The remaining case is $C^2= 0$. In this case, since $2p_a(C)-2 = -C\cdot D \leq 0$ and $C$ is not a smooth rational curve, we must have $C\cdot D = 0$ and $p_a(C) = 1$. Then $\omega_C =\scrO_C$ and we are done as before. 
\end{proof}

In the situation of Theorem~\ref{nefdivisors}, write $L =\scrO_Y(G)$ for an effective divisor $G$. We can write $G = G_m+G_f$, where $G_f$ and $G_m$ are effective, $G_m$ is the moving part of $|G|$, so that $|G_m|$ has no fixed components and $G_m$ is nef, and $G_f$ is the fixed component of $|G|$; in particular, $\dim |G_f| =0$. It will be useful to consider more general decompositions:

\begin{lemma}\label{thirdtechlemma} Let $G$ be a nef divisor, and suppose that $G=G_1+G_2$, where 
\begin{enumerate}
\item[\rm(a)] $G_1$ and $G_2$ are effective.
\item[\rm(b)] $h^0(Y; \scrO_Y(G)) = h^0(Y; \scrO_Y(G_1))$.
\item[\rm(c)] $h^0(Y; \scrO_Y(G_2))=1$ and $G_2\cap D$ is finite.
\end{enumerate}
Then the following hold:
\begin{enumerate}
\item[\rm(i)] If $(G_2)^2 = G_2\cdot D =0$, then $G_2 =0$.
\item[\rm(ii)] If $G_1\cdot G_2 = 0$, then $G_2 = 0$.
\item[\rm(iii)] Every connected component of $\operatorname{Supp} G_2$ has a nonempty intersection with $G_1$.
\end{enumerate}
\end{lemma}
\begin{proof} (i) Suppose that $G_2 \neq 0$. By (c) and the hypothesis that $G_2\cap D$ is finite, $G_2\cap D =\emptyset$. By (iii) of Lemma~\ref{firsttechlemma}, $h^1(Y;\scrO_Y(G_2)) \geq 1$.  By assumption (c),  $h^0(Y; \scrO_Y(G_2))=1$. Thus, 
$$h^0(Y; \scrO_Y(G_2)) - h^1(Y; \scrO_Y(G_2)) \leq 0.$$
On the other hand, by hypotheses and by (i) of Lemma~\ref{firsttechlemma},
$$h^0(Y; \scrO_Y(G_2)) - h^1(Y; \scrO_Y(G_2)) = 1+ \frac12(G_2^2+ G_2\cdot D) =1.$$
This is a contradiction, so we must have $G_2=0$.

\smallskip
\noindent (ii) If $G_1\cdot G_2 = 0$, then $G\cdot G_2 = G_2^2 \geq 0$, since $G$ is nef, and $G_2\cdot D \geq 0$ by hypothesis. Thus $1+ \frac12(G_2^2+ G_2\cdot D)\geq 1$. On the other hand, 
$$1+ \frac12(G_2^2+ G_2\cdot D) = h^0(Y; \scrO_Y(G_2)) - h^1(Y; \scrO_Y(G_2)) = 1- h^1(Y; \scrO_Y(G_2))\leq 1,$$
so that we must have $1+ \frac12(G_2^2+ G_2\cdot D)= 1$ and hence $(G_2)^2 = G_2\cdot D =0$. But then $G_2 =0$ by (i). Conversely, if $G_2\neq 0$, then $G_1\cdot G_2 > 0$. (Note: if $G$ is also assumed to be  big, then (ii) is automatic, since a nef and big divisor is numerically connected \cite[p.\ 24 Ex.\ 13]{Friedbook}.)

\smallskip
\noindent (iii) This follows from (ii) by replacing $G_2$ by a connected component of its support. 
\end{proof}

\begin{corollary}\label{techcor} If $G=G_1+G_2$ is a nef divisor satisfying the hypotheses (a)--(c) of Lemma~\ref{thirdtechlemma}, such that  $G_1$ is nef and $h^1(Y; \scrO_Y(G_1)) = 0$, or more generally  $h^1(Y; \scrO_Y(G_1)) \leq h^1(Y; \scrO_Y(G))$, then $G_2=0$.
\end{corollary}
\begin{proof} First we claim that 
$$h^1(Y; \scrO_Y(G_1)) - h^1(Y; \scrO_Y(G)) = \frac12((G_1\cdot G_2) + (G\cdot G_2) + (D\cdot G_2)).$$
To see this, using (i) of Lemma~\ref{firsttechlemma}, we have
\begin{align*}
h^0(Y; \scrO_Y(G)) - h^1(Y; \scrO_Y(G)) &= 1+ \frac12(G^2+ G\cdot D);\\
h^0(Y; \scrO_Y(G_1)) - h^1(Y; \scrO_Y(G_1)) &= 1+ \frac12(G_1^2+ G_1\cdot D).
\end{align*}
Subtract the second line from the first, using   $h^0(Y; \scrO_Y(G))=h^0(Y; \scrO_Y(G_1))$ and  $G = G_1+G_2$.  We see that $h^1(Y; \scrO_Y(G_1)) - h^1(Y; \scrO_Y(G))$ is equal to
\begin{gather*}
  \frac12(G_1^2+2G_1\cdot G_2 + G_2^2 + G_1\cdot D + G_2\cdot D-G_1^2- G_1\cdot D)\\
 =\frac12((G_1\cdot G_2) + (G\cdot G_2) + (D\cdot G_2)),
\end{gather*}
where we have used the fact that $2G_1\cdot G_2 + G_2^2 = G_1\cdot G_2  + G\cdot G_2$. This establishes the formula.

By (ii) of Lemma~\ref{thirdtechlemma}, in order to prove the corollary, it is enough to show that $G_1\cdot G_2 =0$. From the above formula, we have
$$0\geq   h^1(Y; \scrO_Y(G_1)) - h^1(Y; \scrO_Y(G)) = \frac12((G_1\cdot G_2) + (G\cdot G_2) + (D\cdot G_2)),$$
where all of the terms on the right hand side are nonnegative.
Thus $G_1\cdot G_2 =0$ as claimed.
\end{proof} 

Returning to the proof of Theorem~\ref{nefdivisors}, note that the decomposition $G = G_m+G_f$ satisfies the hypotheses of Lemma~\ref{thirdtechlemma} by Corollary~\ref{nocompD}. We divide the proof into three cases:

\medskip
\noindent \textbf{Case I:} There exists an irreducible member of $|G_m|$ and $G_m \cdot D > 0$. Then, by 
(ii) of Lemma~\ref{firsttechlemma}, $h^1(Y; \scrO_Y(G_m))=0$. Applying Corollary~\ref{techcor} to $G=G_m+G_f$ gives $G_f =0$. Thus $|G|$ has  no fixed components, and, by Lemma~\ref{secondtechlemma}, $|G|$ has a base point $\iff$ $G\cdot D = 1$.
 
\medskip
\noindent \textbf{Case II:} There exists an irreducible member of $|G_m|$ and $G_m \cdot D = 0$. Replacing $G_m$ by a general member of $|G_m|$, we may assume by Lemma~\ref{secondtechlemma} that $G_m$ is smooth. If $G_f=0$, we are done. Otherwise, by (ii) of Lemma~\ref{thirdtechlemma}, $G_m\cdot G_f > 0$. Let $C$ be an irreducible component of $G_f$ such that $G_m\cdot C \geq 1$. Then $C$ is a smooth rational curve and $C^2=-1$ or $-2$.

\begin{claim}\label{techclaim} $G_m\cdot C = 1$.
\end{claim}
\begin{proof}[Proof of the claim] Suppose that $G_m\cdot C = n\geq 2$. Then $|G_m+C|\subseteq |G|$ is the obvious sense. There is an exact sequence
$$0 \to \scrO_Y(G_m) \to  \scrO_Y(G_m+C) \to \scrO_C(k) \to 0,$$
where $k = n+ C^2 \geq 2+C^2$, and hence $k\geq 0$ if $C^2=-2$ and $k\geq 1$ if $C^2=-1$. By (iii) of Lemma~\ref{firsttechlemma}, $h^1(Y; \scrO_Y(G_m))=1$, and $h^1(Y; \scrO_Y(G_m+C)) =1$ if $C^2=-2$. If $C^2=-1$, then the image of $H^0(Y; \scrO_Y(G_m+C))$ in $H^0(C; \scrO_C(k))$ has codimension at most one in the $(k+1)$-dimensional vector space $H^0(C; \scrO_C(k))$, and in particular there exists an element of $|G_m+C|$ which does not contain $C$. This contradicts the fact that $C$ is a fixed component of $|G|$. Likewise, if $C^2=-2$, then the homomorphism $H^1(Y; \scrO_Y(G_m)) \to H^1(Y; \scrO_Y(G_m+C))$ is surjective, as $H^1(C; \scrO_C(k)) =0$, and hence is an isomorphism since $h^1(Y; \scrO_Y(G_m))=h^1(Y; \scrO_Y(G_m+C)) =1$. Then $H^0(Y; \scrO_Y(G_m+C))\to H^0(C; \scrO_C(k))$ is surjective. Since $k\geq 0$,   there is a member of $|G_m+C|$ which does not contain $C$, and we get a contradiction as before.
\end{proof} 

By (iii) of Lemma~\ref{thirdtechlemma}, every connected component of $G_f$ meets $G_m$. If $G\cdot D = 0$, then $h^1(Y; \scrO_Y(G_m)) = 1= h^1(Y; \scrO_Y(G))$  by  Lemma~\ref{firsttechlemma} and Lemma~\ref{LrestrD}. Then $G_f =0$ by Corollary~\ref{techcor}. Hence $G\cdot D = G_f\cdot D >0$.  Thus, there exists a connected component   of $(G_f)_{\text{red}}$ which meets both $G_m$ and $D$, and hence there exist distinct curves $C_1, \dots, C_b $ such that (i) $C_1\cdot G_m \neq 0$, hence $C_1\cdot G_m  =1$ by Claim~\ref{techclaim}; (ii) $C_i \cap C_{i+1} \neq \emptyset$ for $i = 1, \dots, b-1$;  and (iii) $C_b\cap D \neq \emptyset$, hence $C_b \cdot D =1$, and $C_i\cdot D =0$ for $i< b$, so that $C_i^2=-2$ if $i< b$ and $C_b^2=-1$. The proof of Claim~\ref{techclaim} then shows that $C_i \cdot (G_m + C_1 + \cdots + C_{i-1}) = 1$. Since $C_i\cdot C_{i-1} \geq 1$, it then follows that, for $i> 1$, $C_i\cdot C_{i-1} = 1$, $C_i \cdot C_j = 0$ for $j< i-1$ and $C_i \cdot G_m=0$. 

We now claim that $G = G_m + C_1+ \cdots + C_b$. In any case, we have a decomposition $G=G_1+G_2$, where $G_1= G_m + C_1+ \cdots + C_b$ is nef and $G_2$ is an effective divisor contained in $G_f$. Hence $G_1$ and $G_2$  satisfy the hypotheses of Lemma~\ref{thirdtechlemma}. We want to show that $G_2=0$. The divisor $G_1$ is reduced and connected and $G_1\cdot D > 0$. By 
(ii) of Lemma~\ref{firsttechlemma}, $h^1(Y; \scrO_Y(G_1))=0$, and we conclude as in Case I via Corollary~\ref{techcor}.

Finally, the general member of $|G_m|$ is smooth by Lemma~\ref{secondtechlemma}. If $G_m^2>0$, then the general element of $|G_m|$ is a smooth curve of genus at least $2$. If $G_m^2 =0$, then it is a smooth elliptic curve $F$. In either case, $L$ satisfies (ii) of Theorem~\ref{nefdivisors}, with $k=1$.

\medskip
\noindent \textbf{Case III:} The general element of $|G_m|$ is reducible. Then $|G_m|$ is composite with a pencil, necessarily rational, so that $G_m = kF$ for some irreducible curve $F$ on $Y$ and some $k>1$. We claim that the general element of $|F|$ is a smooth elliptic curve and that $F\cdot D =0$. First, $F^2=0$, for otherwise $|kF|$ has no base points, by Lemma~\ref{secondtechlemma} and the fact that $(kF)\cdot D$ is either $0$ or at least $2$.  But this is impossible if $F^2 >0$, since the general element of $|G_m|$ is of the form $F_1 +\cdots + F_k$ and $F_i\cdot F_j > 0$. So $F^2=0$, and either $F$ is a smooth elliptic curve and $F\cdot D =0$, or $F$ is a smooth rational curve and $F\cdot D = 2$. We claim that this last case is impossible. In any case,  $G_f\neq 0$ since $G^2>0$. By (ii) of Lemma~\ref{thirdtechlemma}, $G_m\cdot G_f > 0$ and hence $F\cdot G_f>0$. Let $C$ be an irreducible component of $G_f$ such that $F\cdot C \geq 1$. Then $C$ is a smooth rational curve and $C^2=-1$ or $-2$. Moreover, $n=(kF+C)\cdot C \geq k-2\geq 0$. From the exact sequence
$$0\to \scrO_Y(kF) \to  \scrO_Y(kF+C) \to \scrO_C(n)\to 0,$$
and the easy calculation that $H^1(Y; \scrO_Y(kF)) = 0$ if $F^2 =0$, $F\cdot D =2$, we see that the map $H^0(Y; \scrO_Y(kF+C)) \to H^0(C;\scrO_C(n))$ is surjective, contradicting the fact that $C$ is a component of $G_f$. Hence the case $F^2 =0$, $F\cdot D =2$ is impossible. 

Thus $G_m = kF$ where $F$ is a smooth elliptic curve, $F^2=0$, and $F\cdot D = 0$. An easy inductive argument shows that $\dim H^1(Y; \scrO_Y(kF)) = k$.  Arguing as in the preceding paragraph, there exists a smooth rational curve $C\subseteq \operatorname{Supp}  G_f$ such that $F\cdot C > 0$. By Lemma~\ref{firsttechlemma}, $\dim H^1(Y; \scrO_Y(kF+C)) = 1$ if $C^2=-2$ and $\dim H^1(Y; \scrO_Y(kF+C)) = 0$ if $C^2=-1$. Let $m= F \cdot C \geq 1$, so that $(kF+C) \cdot C = km-1$ if $C^2=-1$ and $=km-2$ if $C^2 = -2$. Suppose that $m\geq 2$, so that $(kF+C)\cdot C \geq 0$. Then it is easy to check that the image of $H^0(Y; \scrO_Y(kF+C))$ in $H^0(C; \scrO_Y(kF+C)|C)$  has dimension at least $km-1-k = k(m-1) -1> 0$ if $m\geq 2$, since $k\geq 2$.  Arguing as in the proof of Claim~\ref{techclaim},  this would mean that $C$ is not a fixed component of $kF+C$, a contradiction. Thus $F\cdot C = 1$.

Now either $G\cdot D =0$ or $G\cdot D > 0$. If $G\cdot D =0$, then clearly $C^2=-2$. Set $G_1 = kF+C$ and $G_2 = G-G_1$. Since $k\geq 2$, $G_1$ is nef and $G_1$ and $G_2$ satisfy the hypotheses of Lemma~\ref{thirdtechlemma}. Moreover, by (iii) of Lemma~\ref{firsttechlemma}, $h^1(Y;\scrO_Y(G_1)) = 1$  and $h^1(Y;\scrO_Y(G)) \geq 1$. Thus $h^1(Y; \scrO_Y(G_1)) - h^1(Y; \scrO_Y(G))\leq 0$, and so $G_2 = 0$ by Corollary~\ref{techcor}. Hence $G = kF+C$ and we are in Case (i) of Theorem~\ref{nefdivisors}.

Finally, if $G\cdot D >0$, then the argument of Case II shows that $G=kF+C_1+ \cdots +C_b$, where   $C_i$ is a $-2$-curve for $i\leq b-1$, $C_b$ is an interior exceptional curve, and $C_i \cdot C_j = 1$ if $j=i\pm 1$ and $0$ otherwise, $F\cdot C_1 = 1$, $F\cdot C_i  =0$ for $i>1$, and  $F\cdot D =0$. Thus $L$ satisfies Case (ii) of Theorem~\ref{nefdivisors} with $k\geq 2$. This completes the proof of Theorem~\ref{nefdivisors}.
\end{proof} 

Given Theorem~\ref{nefdivisors}, standard arguments (cf.\ \cite{Mayer} or \cite{Fried1}) show the following:

\begin{theorem}\label{3ample} Let $L$ be a nef and big line bundle on $Y$. 
\begin{enumerate}
\item[\rm(i)] For $n\geq 2$, $L^{\otimes n}$ has no fixed components or base points. 
\item[\rm(i)] For $n\geq 3$, the morphism defined by $L^{\otimes n}$ is birational onto its image, which is the normal projective surface obtained by contracting all of the curves $C$, necessarily $-2$-curves,   interior exceptional curves, or components of $D$, such that $L\cdot C =0$. \qed
\end{enumerate}
\end{theorem}

In case $L$ is nef but not big, we have the following:

\begin{theorem}\label{Lnotbig} Let $L=\scrO_Y(G)$ be a nef line bundle on $Y$ such that $G^2 =0$, where $G$ is effective and nonzero. Suppose that no component of $D$ is a fixed component of $L$. Then either $G = kC$, where $C$ is a smooth rational curve, $C^2=0$, $C\cdot D = 2$, and $k\geq 1$, or $G=kE$, where $E$ is a smooth elliptic curve, $E^2=E\cdot D =0$, and $k\geq 1$. Moreover, $|G|$ has no fixed components or base points.
\end{theorem}
\begin{proof} As in the proof of Theorem~\ref{nefdivisors}, write $G = G_m+ G_f$. Since $G^2=0$, we have
$$0 = G_m^2 + 2G_m\cdot G_f + G_f^2 = G_m^2 +  G_m\cdot G_f + G\cdot G_f.$$
Since all terms are nonnegative, $G_m\cdot G_f =0$. By (ii) of Lemma~\ref{thirdtechlemma}, $G_f =0$.  Hence $G = G_m$.

First suppose that the general element of $|G|$ is irreducible. By adjunction,
$$-2 \leq 2p_a(G) -2 = G^2 -G\cdot D = -G\cdot D.$$
 Hence, either $G\cdot D =0$ and $p_a(G) = 1$ or $G\cdot D =2$ and $p_a(G) = 0$. In the first case, by Lemma~\ref{secondtechlemma}, $|G|$ has no base points and the general element of $|G|$ is a smooth elliptic curve $E$. In the second case, every irreducible element of $|G|$, and hence the general element of $|G|$,  is a smooth rational curve. The proof of Lemma~\ref{secondtechlemma} shows that $|G|$ has no base points in this case as well.

Finally suppose that the general element of $|G|$ is not irreducible. Then the argument in the first paragraph of the proof of Case III of Theorem~\ref{nefdivisors} shows that $G = kE$ or $G=kC$, where $E$ and $C$ are as in the statement of the theorem and $k\geq 2$.
\end{proof}

\section{The generic ample cone}

If $Y$ is a del Pezzo surface, i.e.\ $D$ is ample, or more generally if $D$ is nef and there are no $-2$-curves on $Y$, then every class $\alpha\in H^2(Y; \Zee)$ such that $\alpha^2 =\alpha \cdot K_Y =-1$ is the class of an exceptional curve. However, in general it is hard to characterize the classes of exceptional curves. On the other hand, if we only care about effective classes, then one can say more:

\begin{definition} The class $\alpha\in H^2(Y; \Zee)$ is a \textsl{numerical exceptional curve} if $\alpha^2 =\alpha \cdot K_Y =-1$. A numerical exceptional curve $\alpha$ is \textsl{effective} if $\alpha$ is the class of an effective divisor.
\end{definition}

\begin{lemma}\label{chareff} The class $\alpha\in H^2(Y; \Zee)$ is an effective numerical exceptional curve $\iff$ for every nef   $\Ar$-divisor $x$ on $Y$ such that $x\cdot [D] > 0$, $x\cdot \alpha \geq 0$ $\iff$ there exists a nef   $\Ar$-divisor $x$ on $Y$ such that $x\cdot [D] > 0$ and $x\cdot \alpha \geq 0$.
\end{lemma}
\begin{proof} Clearly, if $\alpha$ is effective, then $x\cdot \alpha \geq 0$ for every nef  $\Ar$-divisor $x$ on $Y$, and hence for at least one such. Conversely, suppose that there exists a nef   $\Ar$-divisor $x$ on $Y$ with $x\cdot [D] >0$ such that $x\cdot \alpha \geq 0$. If $L_\alpha$ is the line bundle corresponding to $\alpha$, then by Riemann-Roch $\chi(Y; L_\alpha) = 1$. Since $h^2(Y;L_\alpha) = h^0(Y; L_\alpha^{-1}\otimes \scrO_Y(-D))$ and $x \cdot (-\alpha -[D])< 0$, $h^2(Y;L_\alpha) = 0$. Thus $h^0(Y;L_\alpha) \geq 1$ and $\alpha$ is effective.
\end{proof} 

\begin{corollary}\label{genample1}  If there do not exist any $-2$-curves on $Y$, then 
\begin{gather*}
\overline{\mathcal{A}}(Y) = \{x\in \mathcal{C}^+: x\cdot [D_i] \geq 0 \text{ and }
 x\cdot \alpha \geq 0  \\ \text{ for all   effective numerical exceptional curves $\alpha$}\}.
 \end{gather*}
\end{corollary}
\begin{proof} Since the class of an exceptional curve is an effective numerical exceptional curve, the right hand side above is contained in the left hand side. On the other hand, since a nef divisor has nonnegative intersection with every effective divisor, the left hand side is contained in the right hand side, and so they are equal. 
\end{proof}

\begin{definition} Define the \textsl{generic ample cone} $\overline{\mathcal{A}}_{\text{\rm{gen}}}= \overline{\mathcal{A}}_{\text{\rm{gen}}}(Y)$ to be the set defined by the right hand side of Corollary~\ref{genample1}, i.e.\ $\overline{\mathcal{A}}_{\text{\rm{gen}}}$ is the set of $x\in \mathcal{C}^+$ such that $x\cdot [D_i] \geq 0$ and  
 $x\cdot \alpha \geq 0$  for all   effective numerical exceptional curves $\alpha$. (This set is denoted $C_D{}^+{}^+$ in \cite{GHK}.) Let $\mathcal{A}_{\text{\rm{gen}}}(Y)$ be the interior of $\overline{\mathcal{A}}_{\text{\rm{gen}}}(Y)$ (in either $\mathcal{C}^+$ or $H^2(Y; \Ar)$). Clearly $\overline{\mathcal{A}}(Y) \subseteq \overline{\mathcal{A}}_{\text{\rm{gen}}}$,  $\overline{\mathcal{A}}(Y) = \overline{\mathcal{A}}_{\text{\rm{gen}}}$ $\iff$ there are no $-2$-curves on $Y$, and more generally 
 $$\overline{\mathcal{A}}(Y) = \{ x\in \overline{\mathcal{A}}_{\text{\rm{gen}}}: x\cdot [C] \geq 0 \text{ for all $-2$-curves $C$} \}.$$ 
\end{definition}

To explain the terminology ``generic ample cone," let $\pi \colon (\mathcal{Y}, \mathcal{D}) \to S$ be a deformation of the pair $(Y,D)$ (in the analytic category). If $S$ is contractible, then we can identify $H^2(Y_s; \Zee)$ with $H^2(Y; \Zee)$ for every $s\in S$. Possibly after shrinking $S$, we can also assume that there exists a relatively ample line bundle $\scrO_{\mathcal{Y}}(\mathcal{H})$ on $\mathcal{Y}$, with $\mathcal{H}|Y_s = H_s$ the corresponding ample divisor class. It follows that $\alpha$ is an effective 
numerical exceptional curve for $Y_s$ $\iff$ $\alpha$ is an effective numerical exceptional curve for $Y$, since the conditions $\alpha^2=\alpha \cdot K_{Y_s} =-1$ and $\alpha \cdot H_s \geq 0$ are independent of $s$. Hence $\overline{\mathcal{A}}_{\text{\rm{gen}}}(Y_s)= \overline{\mathcal{A}}_{\text{\rm{gen}}}(Y)$ under the identifications. Suppose moreover that $Y_s$ is a very general small deformation in the sense that there are no $-2$-curves on $Y_s$; note that such small deformations always exist by Corollary~\ref{gener1}. Then $\overline{\mathcal{A}}_{\text{\rm{gen}}}(Y_s) = \overline{\mathcal{A}}(Y_s)$, so that the generic ample cone of $Y$ is the ample cone of a very general small deformation of $Y$. The argument also shows:
 
\begin{lemma}\label{definv} The set of effective numerical exceptional curves and the set $\overline{\mathcal{A}}_{\text{\rm{gen}}}$ are  locally constant, and hence are invariant in a global deformation with trivial monodromy under the induced identifications. \qed
\end{lemma}

\begin{corollary}\label{monoinvar} The set $\overline{\mathcal{A}}_{\text{\rm{gen}}}$ is invariant under the global monodromy group of the pair $(Y,D)$. \qed
\end{corollary}

Corollary~\ref{monoinvar} leads to the ``correct" definition of allowable isometries and of  markings  in a family of pairs over a reduced, connected base $S$:

\begin{definition}\label{defadmmarking}  Let $(Y,D)$ and $(Y',D')$ be two labeled anticanonical pairs. An integral isometry $\gamma \colon H^2(Y';\Zee) \to H^2(Y; \Zee)$ is \textsl{admissible} if   $\gamma([D_i]) = [D_i']$ for all $i$ and  
$$\gamma (\overline{\mathcal{A}}_{\text{\rm{gen}}}(Y')) = \overline{\mathcal{A}}_{\text{\rm{gen}}}(Y),$$
where we denote   the natural extension of $\gamma$ to an isometry $H^2(Y'; \Ar) \to H^2(Y; \Ar)$ by $\gamma$ as well.

Let $\pi\colon (\mathcal{Y}, \mathcal{D}) \to S$ be a family of anticanonical pairs over a reduced connected complex space $S$, and let $\theta\colon R^2\pi_*\Zee \to \underline{\widehat{\Lambda}}$ be a marking, where $\widehat{\Lambda}=H^2(Y_s;\Zee)$ for some $s\in S$. Then $\theta$ is \textsl{admissible} if, for all $t\in S$, $\theta(\overline{\mathcal{A}}_{\text{\rm{gen}}}(Y_t)) = \overline{\mathcal{A}}_{\text{\rm{gen}}}(Y_s)$.
\end{definition}

We use the standard notation $\Lambda_\Ar$ for $\Lambda\otimes _\Zee\Ar$    and $\Lambda_\Q$ for $\Lambda\otimes _\Zee\Q$. The  following shows that the generic ample cone is determined by its intersection with the real subspace $\Lambda \otimes_\Zee \Ar$ of $H^2(Y; \Ar)$.

\begin{lemma}\label{detbyres} {\rm(i)} If $(Y,D)$ is not negative semidefinite, then $\overline{\mathcal{A}}_{\text{\rm{gen}}}\cap \Lambda_\Ar =\{0\}$. 

\smallskip
\noindent
{\rm(ii)} If $(Y,D)$ is strictly negative semidefinite, then $\overline{\mathcal{A}}_{\text{\rm{gen}}}\cap \Lambda_\Ar =\Ar^+[D]\cup\{0\}$.

\smallskip
\noindent  
{\rm(iii)} If $(Y,D)$ is negative definite, then $\overline{\mathcal{A}}(Y)\cap \Lambda_\Ar$ and $\overline{\mathcal{A}}_{\text{\rm{gen}}} \cap \Lambda_\Ar$ have nonempty interior in $\Lambda_\Ar$, $\overline{\mathcal{A}}(Y)\cap \Lambda_\Q$ is dense in $\overline{\mathcal{A}}(Y)\cap \Lambda_\Ar$, and similarly for $\overline{\mathcal{A}}_{\text{\rm{gen}}} \cap \Lambda_\Q$.

\smallskip
\noindent Finally, in all cases, $\overline{\mathcal{A}}_{\text{\rm{gen}}}$ determines and is determined by  $\overline{\mathcal{A}}_{\text{\rm{gen}}}\cap \Lambda_\Ar$. 
\end{lemma} 
\begin{proof}
(i) and (ii) are easy consequences of the Hodge index theorem, and (iii) follows from Proposition~\ref{constructnef}.  To see the final statement of the lemma, we first show:

\begin{lemma}\label{caseslemma} {\rm(i)} Suppose that $D$ is not negative semidefinite, and that no component of $D$ is an exceptional curve. Let $D_i$ be a component of $D$ such that $D_i^2 \geq 0$. Then a numerical exceptional curve $\alpha$ is effective $\iff$ $\alpha \cdot [D_i] \geq 0$.

\smallskip
\noindent {\rm(ii)} Suppose that $D$ is strictly negative semidefinite, and that no component of $D$ is an exceptional curve. Then a numerical exceptional curve $\alpha$ is effective $\iff$ $\alpha \cdot [D] \geq 0$.

\smallskip
\noindent {\rm(iii)} Suppose that $D$ is  negative  definite, and that no component of $D$ is an exceptional curve. Let $y\in \overline{\mathcal{A}}(Y)\cap \Lambda_\Ar$, and suppose that $y$ does not lie on any wall of $\overline{\mathcal{A}}(Y)$ except those corresponding to $[D_i]$, $1\leq i\leq r$. (Note that such $y$ exist by Proposition~\ref{constructnef}.) Finally, let $\alpha$ be a numerical exceptional curve which is not in the $\Zee$-span of the $[D_i]$. Then $\alpha$ is effective $\iff$ $\alpha\cdot y \geq 0$.
\end{lemma}
\begin{proof} (i): The hypothesis implies that either $D$ is irreducible and $D^2> 0$ or $D$ is not irreducible and there exists a component $D_i$ of $D$ with $D_i^2 \geq 0$ and hence $D_i\cdot D > 0$. The result then follows from Lemma~\ref{chareff}, by taking $x$ to be the nef divisor $D$ in the first case and $D_i$ in the second. 

\smallskip
\noindent (ii): The hypothesis implies that $D$ is nef and $D^2=0$. Thus, if $\alpha$ is effective, then $\alpha \cdot [D] \geq 0$. Conversely, suppose that $\alpha \cdot [D] \geq 0$. As in the proof of Lemma~\ref{chareff}, it suffices to show that $-\alpha -[D]$ is not the class of an effective divisor. However, as $[D]\cdot (-\alpha -[D]) = -1$ and $D$ is nef, $-\alpha -[D]$ cannot be the class of an effective divisor.

\smallskip
\noindent (iii): As above, if $\alpha$ is effective, then $\alpha\cdot y \geq 0$, and we 
we must show the converse. It again suffices to show that, if $\alpha\cdot y \geq 0$, then  $-\alpha -[D]$ is not the class of an effective divisor. But if $-\alpha -[D]$ is effective, then $y\cdot (-\alpha -[D]) =-\alpha \cdot y \geq 0$, and hence $\alpha \cdot y = 0$. Thus $y\cdot (-\alpha -[D]) =0$ as well. Hence, if $-\alpha -[D]   =\sum_j[C_j]$, where the $C_j$ are irreducible curves, $y\cdot [C_j]=0$ for every $j$. It follows that $C_j$ is a component of $D$ for every $j$, and hence that $\alpha = -\sum_j[C_j] -[D]$ is in the $\Zee$-span of the $[D_i]$. But this contradicts the hypothesis on $\alpha$.
\end{proof}

Returning to the  proof of  Lemma~\ref{detbyres}, we first reduce by induction on the number of interior exceptional curves to the case where no component of $D$ is an exceptional curve. For if $E$ is an exceptional curve which is a component of $D$, and $(\overline{Y}, \overline{D})$ is the pair obtained by blowing down $E$, then $\Lambda(\overline{Y}, \overline{D}) \cong \Lambda(Y,D)$ and 
$$\overline{\mathcal{A}}_{\text{\rm{gen}}}(\overline{Y})\cap \Lambda_\Ar(\overline{Y}, \overline{D}) = \overline{\mathcal{A}}_{\text{\rm{gen}}}(Y)\cap \Lambda_\Ar(Y,D).$$
   Assuming inductively that Lemma~\ref{detbyres} holds for $(\overline{Y}, \overline{D})$, it follows that $\overline{\mathcal{A}}_{\text{\rm{gen}}}(Y)\cap \Lambda_\Ar(Y,D)$ determines $\overline{\mathcal{A}}_{\text{\rm{gen}}}(\overline{Y})$. We may assume that $Y$ and $\overline{Y}$ are generic. An element $x$ in the interior of $\overline{\mathcal{A}}_{\text{\rm{gen}}}(\overline{Y})$   is then a nef $\Ar$-divisor on $\overline{Y}$  such that $x\cdot  [D] > 0$. By Lemma~\ref{chareff}, $x$ then determines the set of effective numerical exceptional curves on $Y$ and hence $\overline{\mathcal{A}}_{\text{\rm{gen}}}(Y)$.

Thus we may assume that no component of $D$ is an exceptional curve. Since $\overline{\mathcal{A}}_{\text{\rm{gen}}}(Y)$ is determined by the set of effective numerical exceptional curves, Lemma~\ref{caseslemma} completes the proof of Lemma~\ref{detbyres} in case $D$ is not negative semidefinite, or is strictly negative semidefinite. To deal with the negative definite case, an easy argument shows that $\overline{\mathcal{A}}_{\text{\rm{gen}}}(Y)$ is the set of $x\in \mathcal{C}^+$ such that $x\cdot [D_i] \geq 0$ and $x\cdot \alpha \geq 0$ for all effective numerical exceptional curves $\alpha$ which are not in the $\Zee$-span of the $[D_i]$, since, if $E$ is an exceptional curve, then $[E] =\sum_ir_i[D_i]$ with $r_i \in \Ar$ $\iff$ $E=D_i$ for some $i$. We then conclude by (iii) of Lemma~\ref{caseslemma}.
\end{proof}

In the negative definite case, the set $\overline{\mathcal{A}}(Y)\cap \Lambda_\Ar$ has the following interpretation:

\begin{proposition}\label{whenalg} Let $(Y,D)$ be a negative definite anticanonical pair and let $\overline{Y}$ be the normal complex surface obtained by contracting $Y$. Then the set of ample line bundles  on $\overline{Y}$ is naturally identified with the set of $\lambda \in \overline{\mathcal{A}}(Y)\cap \Lambda$ such that  $\lambda \cdot [E] > 0$ for all interior exceptional curves on $Y$, $\lambda \cdot [C] > 0$ for all $-2$-curves $C$ on $Y$, and  $\varphi_Y(\lambda) = 1$.
\end{proposition} 
\begin{proof} Given $\lambda$ as in the statement of the proposition, let $L$ be the line bundle corresponding to $\lambda$. Then by hypothesis the degree of $L$ on every exceptional curve or $-2$-curve is positive, and $L|D \cong \scrO_D$. Clearly, the degree of $L$ on any curve on $Y$ of nonnegative square is also positive. By Corollary~\ref{nocompD}, there exists a section of $L$ which is everywhere nonvanishing on $D$, and hence $L$ induces a line bundle $\overline{L}$ on $\overline{Y}$. By the Nakai-Moishezon 
criterion, $\overline{L}$ is ample. (In fact, Theorem~\ref{3ample} gives a more precise statement.)

Conversely, suppose that $\overline{L}$ is an ample line bundle on $\overline{Y}$. Let $L$ be the pullback of $\overline{L}$ to $Y$ and let $\lambda$ be the class of $L$ in $H^2(Y; \Zee)$. Then clearly $\lambda \in \Lambda $ and $\varphi_Y(\lambda) = 1$. By hypothesis $L$ has positive degree on every curve which is not a component of $D$, so that $\lambda \cdot \alpha > 0$ for every interior exceptional curve $\alpha$ and $\lambda \cdot [C] > 0$ for every $-2$-curve  $C$.
\end{proof}

\begin{remark} It is easy to give variants of the above where we relax some of the conditions on $\lambda$. For example, if we only require that $\lambda \cdot [C] \geq 0$ for all $-2$-curves $C$ on $Y$, then $L$ will induce a a birational morphism from $\overline{Y}$ to its image in some projective space which is an embedding in a neighborhood of the image of $D$, but which may also contract certain configurations of $-2$-curves to rational double points.  

Similarly, if  $\lambda \cdot [E] =0$ for some interior exceptional curve $E$,  then the line bundle $\overline{L}$ will not define an embedding of $\overline{Y}$ in a neighborhood of the image of $D$, but rather of a blown down version of $\overline{Y}$.
\end{remark}

\begin{remark} The set of $\varphi\in \Hom(\Lambda , \mathbb{G}_m)$ for which there exists an ample line bundle on $\overline{Y}$, i.e.\ which satisfy the hypotheses of Proposition~\ref{whenalg},  is Zariski dense in $\Hom(\Lambda , \mathbb{G}_m)$, but is not dense in the classical topology.   This is in contrast to the fact that algebraic $K3$ surfaces are dense in the moduli space (suitably interpreted) of all $K3$ surfaces. 
\end{remark}

The following  result highlights the importance  of $\overline{\mathcal{A}}_{\text{\rm{gen}}}$:

\begin{theorem}\label{defequiv} Two labeled pairs $(Y, D)$ and $(Y', D')$ with $r(D) = r(D')$ are deformation equivalent if and only if  there exists an admissible integral isometry $\gamma \colon H^2(Y'; \Zee) \to H^2(Y; \Zee)$.
\end{theorem}
\begin{proof} If $(Y, D)$ and $(Y', D')$   are deformation equivalent via a family over an irreducible base $S$ realizing the deformation equivalence, then passing to the universal cover of $S$ defines an isometry $\gamma \colon H^2(Y'; \Zee) \to H^2(Y; \Zee)$ satisfying: $\gamma([D_i']) = [D_i]$ for every $i$ and $\gamma(\overline{\mathcal{A}}_{\text{\rm{gen}}}(Y')) = \overline{\mathcal{A}}_{\text{\rm{gen}}}(Y)$, by Lemma~\ref{definv}. The general case, where the family $S$ is connected but not necessarily irreducible, follows easily from this special case.

Conversely, let $\gamma \colon H^2(Y'; \Zee) \to H^2(Y; \Zee)$ be  an admissible integral isometry. We may as well assume that both $Y$ and $Y'$ are generic, so that there are no $-2$-curves on $Y$ and $Y'$ and that $\overline{\mathcal{A}}_{\text{\rm{gen}}}(Y) = \overline{\mathcal{A}}(Y)$, and similarly for $Y'$. We argue by induction on the rank of $\Pic Y$, or equivalently on the number of blowups of $Y$ starting with a minimal model (or on $-D^2$). If $Y$ is minimal, then $Y\cong \Pee^2$ or $Y \cong \mathbb{F}_N$ for some $N\neq 1,2$. In this case, it is clear that the ample cone of $Y$ and the self-intersection sequence determine the deformation type of the pair $(Y,D)$.

If $Y$ is not minimal, then there exists an exceptional curve $E$ on $Y$, which is then either an interior exceptional curve or a corner exceptional curve. Suppose that $E$ is an interior exceptional curve, with $\alpha =[E]\in  H^2(Y; \Zee)$. Let $\alpha'$ be the corresponding element of $ H^2(Y'; \Zee)$, i.e.\ $\gamma(\alpha') = \alpha$. Since $W^\alpha$ is a face of $\overline{\mathcal{A}}(Y)$, $W^{\alpha'}$ is a face of $\overline{\mathcal{A}}(Y')$, and hence $\alpha'=[E']$ where $E$ is an exceptional curve on $Y'$. Note that, if $i$ is the unique integer such that $E\cdot D_i =1$, then $E'\cdot D_i'=1$ since $\gamma$ is an isometry. Let $(\overline{Y}, \overline{D})$ be the result of contracting $E$ on $Y$ and let $(\overline{Y}', \overline{D}')$ be the result of contracting $E'$ on $Y'$. Via the identification of $H^2(\overline{Y};\Zee)$ with $[E]^\perp=W^\alpha$, it is easy to see that $\overline{\mathcal{A}}_{\text{\rm{gen}}}(\overline{Y}) =  \overline{\mathcal{A}}_{\text{\rm{gen}}}(Y)\cap [E]^\perp$. Then $\gamma$ induces an isometry $\overline{\gamma}\colon \colon H^2(\overline{Y}'; \Zee) \to H^2(\overline{Y}; \Zee)$ such that $\overline{\gamma}([D_i']) = [D_i]$ for every $i$ and $\overline{\gamma}(\overline{\mathcal{A}}_{\text{\rm{gen}}}(\overline{Y}')) = \overline{\mathcal{A}}_{\text{\rm{gen}}}(\overline{Y})$. By the inductive hypothesis, $\overline{Y}$ and $\overline{Y}'$ are deformation equivalent as labeled anticanonical pairs. Since $Y$ is obtained from $\overline{Y}$ by blowing up a point on $\overline{D}_i^{\text{int}}$, and similarly for $Y'$, $Y$ and $Y'$ are deformation equivalent as well. The case where $E$ is a corner exceptional curve is similar (and simpler).
\end{proof}

Theorem~\ref{defequiv} has the following consequence for the smooth topology of pairs $(Y,D)$.

\begin{theorem}\label{diffisdef} Let $(Y, D)$ and $(Y', D')$ be two anticanonical pairs, with $r(D) = r(D') =r$, say, and suppose that there exists a diffeomorphism $f\colon Y\to Y'$ such that, for every $1\leq i\leq r$, $f^*[D_i'] = [D_i]$, where $[D_i]$ and $[D_i']$ denote the classes of $D_i$ and $D_i'$ in $H^2(Y; \Zee)$ and $H^2(Y'; \Zee)$ respectively. Then $f^*(\overline{\mathcal{A}}_{\text{\rm{gen}}}(Y')) = \overline{\mathcal{A}}_{\text{\rm{gen}}}(Y)$, and hence $(Y,D)$ and $(Y', D')$ are deformation equivalent.
\end{theorem}
\begin{proof} It follows from  Theorem~\ref{defequiv} that, if $f^*(\overline{\mathcal{A}}_{\text{\rm{gen}}}(Y')) = \overline{\mathcal{A}}_{\text{\rm{gen}}}(Y)$, then $(Y,D)$ and $(Y', D')$ are deformation equivalent. Also, we may clearly assume that $Y$ and $Y'$ are generic in what follows. There are two main steps to the proof that $f^*(\overline{\mathcal{A}}_{\text{\rm{gen}}}(Y')) = \overline{\mathcal{A}}_{\text{\rm{gen}}}(Y)$. In the first step, we assume that $D$ and hence $D'$ are irreducible. In this case, the result follows easily from a theorem in \cite{FriedmanMorgan}. The second step then shows that the result in general follows from the result in the irreducible case.

\medskip
\noindent \textbf{Step I:} Assume that $D$ and $D'$ are irreducible. We must show that,  in this case, $f^*(\overline{\mathcal{A}}_{\text{\rm{gen}}}(Y')) = \overline{\mathcal{A}}_{\text{\rm{gen}}}(Y)$.  Let $\mathbb{H}(Y)=\{x\in \mathcal{C}^+: x^2 =1\}$ be the hyperbolic space associated to the intersection form on $H^2(Y;\Zee)$; equivalently, $\mathbb{H}(Y)$ is the quotient of $\mathcal{C}^+(Y)$ by $\Ar^+$. Let $P_0(Y)=\overline{\mathcal{A}}_{\text{\rm{gen}}}(Y) \cap \mathbb{H}(Y)$. Equivalently, 
$$\overline{\mathcal{A}}_{\text{\rm{gen}}}(Y) = \Ar^+\cdot P_0(Y) \cup \{0\}.$$
The oriented walls of $P_0(Y)$ then correspond to the oriented walls of $\overline{\mathcal{A}}_{\text{\rm{gen}}}(Y)$.  In conformity with the notation of \cite{FriedmanMorgan} we denote $[D]$ by $\kappa(P_0(Y))$.  In the terminology of \cite{FriedmanMorgan}, $P_0(Y)$ is a $P$-cell. More generally, a  \textsl{$P$-cell} $P$ is a subset of $\mathbb{H}(Y)$ of the form $\gamma(P_0(Y))$, where $\gamma$ is an integral isometry of $H^2(Y; \Zee)$ and we continue to denote by $\gamma$ its extension to $H^2(Y; \Ar)$. We define $\kappa(P) =\gamma(\kappa(P_0(Y)))$. By \cite[Proposition 2.8]{FriedmanMorgan}, $\kappa(P)$ only depends on $P$. The $P$-cell $P_0(Y)$   is contained in a unique super $P$-cell $\mathbb{S}(P_0(Y))$, and similarly for an arbitrary $P$-cell. Here 
$$\mathbb{S}(P_0(Y)) =\bigcup_{\gamma\in \mathcal{R}(P_0(Y))}\gamma\cdot P_0(Y),$$
where $\mathcal{R}(P_0(Y))$ is the group of reflections generated by the reflections about all of the walls of $P_0(Y)$ other than $[D]$ (which are integral isometries since these walls are integral classes of square $-1$). Moreover, $\mathbb{S}(P_0(Y))$ is a convex subset of $\mathbb{H}(Y)$. The main result that we need from \cite{FriedmanMorgan} is the following:

\begin{theorem} Let $f\colon Y \to Y'$ be a diffeomorphism. Then $f^*\mathbb{S}(P_0(Y'))=\pm \mathbb{S}(P_0(Y))$. Moreover, if $f^*[D'] = [D]$, then in fact $f^*\mathbb{S}(P_0(Y'))= \mathbb{S}(P_0(Y))$.
\end{theorem}
\begin{proof} To see the first part, note the proof of Theorem 10A in \cite{FriedmanMorgan}, p.\ 355, which deals with the case $Y=Y'$, applies unchanged to the more general situation where we do not assume $Y=Y'$.

Suppose that  $f^*(\mathbb{S}(P_0(Y')))= - \mathbb{S}(P_0(Y))$. There exists a diffeomorphism $g\colon Y \to Y$ such that $g^* = -\Id$ by \cite[Theorem 10B]{FriedmanMorgan}. After precomposing $f$ with $g$, we obtain a diffeomorphism $h\colon Y \to Y'$ such that $h^*\mathbb{S}(P_0(Y'))= \mathbb{S}(P_0(Y))$ and $h^*(\kappa(P_0(Y'))= -\kappa(P_0(Y))$. But this contradicts the fact that $h^*(\kappa(P_0(Y'))$ is an oriented wall of $\mathbb{S}(P_0(Y))$.
\end{proof}

Next we claim:

\begin{lemma} Let $\gamma\colon H^2(Y'; \Zee) \to H^2(Y; \Zee)$ be an integral isometry such that $\gamma(\mathbb{S}(P_0(Y')))= \mathbb{S}(P_0(Y))$  and $\gamma(\kappa(P_0(Y')))= \kappa(P_0(Y))$. Then 
$$\gamma(P_0(Y')) = P_0(Y).$$
\end{lemma}
\begin{proof}  Clearly $\gamma(P_0(Y'))= P_1$ is some $P$-cell of $Y$ contained in $\mathbb{S}(P_0(Y))$. Moreover, $\kappa(P_1) = \kappa(P_0(Y))$. Choosing a geodesic path $\mu(t)$ in $\mathbb{H}(Y)$ connecting a general point of $P_0(Y)$ to a general point of $P_1$, let $\alpha_1, \dots, \alpha_n$ be the oriented walls crossed by $\mu$ in order. It is easy to check that 
$$\kappa(P_1) = \kappa(P_0(Y)) + 2\sum_i\alpha_i = \kappa(P_0(Y)),$$ and hence $2\sum_i\alpha_i = 0$. But, if $x$ is a general point of $P_0(Y)$, then $\alpha_i \cdot x > 0$ for every $i$. This is a contradiction unless  $n=0$, i.e.\ unless $P_1 = P_0(Y)$, and hence  $\gamma(P_0(Y')) = P_0(Y)$.
\end{proof}

In particular, if $f\colon Y \to Y'$ is a diffeomorphism such that $f^*([D']) = [D]$, then $f^*P_0(Y') = P_0(Y)$. Lifting back into $\mathcal{C}^+$, it follows that $f^* \overline{\mathcal{A}}_{\text{\rm{gen}}}(Y') = \overline{\mathcal{A}}_{\text{\rm{gen}}}(Y)$.

\medskip
\noindent \textbf{Step II:} The general case. Let $(Y,D)$ and $(Y', D')$ be general anticanonical pairs and let $f\colon Y\to Y'$ be a diffeomorphism  such that  $f^*[D_i'] = [D_i]$ for all $i$. We must show that $f^* \overline{\mathcal{A}}_{\text{\rm{gen}}}(Y') = \overline{\mathcal{A}}_{\text{\rm{gen}}}(Y)$. First, by Corollary~\ref{cor3.5}, there exists a small deformation of the pair $(Y,D)$, over a smooth base, to an anticanonical  pair $(Y_{\text{irr}},D_{\text{irr}})$, where $D_{\text{irr}}$ is an irreducible nodal curve, and we can similarly deform  $(Y'D')$ to an anticanonical  pair $(Y_{\text{irr}}',D_{\text{irr}}')$, where $D_{\text{irr}}'$ is irreducible. Over a small contractible  base, there is a diffeomorphism  from $Y$ to $Y_{\text{irr}}$, unique up to isotopy, and similarly for $Y'$ and $Y_{\text{irr}}'$. Thus, the diffeomorphism $f$ defines a diffeomorphism $Y_{\text{irr}} \to Y_{\text{irr}}'$, which we denote by $f_{\text{irr}}$. Clearly, identifying $H^2(Y;\Zee)$ with $H^2(Y_{\text{irr}}; \Zee)$ as well as $H^2(Y';\Zee)$ with $H^2(Y_{\text{irr}}'; \Zee)$, $[D_{\text{irr}}] = \sum_i[D_i]$ and $[D_{\text{irr}}'] = \sum_i[D_i']$.  Thus, $f_{\text{irr}}^*[D_{\text{irr}}'] = [D_{\text{irr}}]$. By Step I, $f_{\text{irr}}^*\overline{\mathcal{A}}_{\text{\rm{gen}}}(Y_{\text{irr}}') = \overline{\mathcal{A}}_{\text{\rm{gen}}}(Y_{\text{irr}})$. 

Thus, returning to $Y$ and $Y'$ and viewing $\overline{\mathcal{A}}_{\text{\rm{gen}}}(Y_{\text{irr}})$ as a subset of $H^2(Y; \Zee)$, and similarly for $\overline{\mathcal{A}}_{\text{\rm{gen}}}(Y_{\text{irr}}')$, we see that $f^*\overline{\mathcal{A}}_{\text{\rm{gen}}}(Y_{\text{irr}}') = \overline{\mathcal{A}}_{\text{\rm{gen}}}(Y_{\text{irr}})$ and that $f^*[D_i'] =[D_i]$. To conclude that $f^* \overline{\mathcal{A}}_{\text{\rm{gen}}}(Y') = \overline{\mathcal{A}}_{\text{\rm{gen}}}(Y)$, it is clearly enough to prove the following: 

\begin{lemma} In the above notation, 
$$\overline{\mathcal{A}}_{\text{\rm{gen}}}(Y) = \{x\in \overline{\mathcal{A}}_{\text{\rm{gen}}}(Y_{\text{\rm{irr}}}): x\cdot [D_i] \geq 0 \text{ for every $i$}\},$$
and similarly for $\overline{\mathcal{A}}_{\text{\rm{gen}}}(Y')$.
\end{lemma}
\begin{proof} Choose an ample divisor $H$ on $Y$. In the deformation of $Y$ to $Y_{\text{irr}}$, and identifying $H^2(Y;\Zee)$ with $H^2(Y_{\text{irr}};\Zee)$, we can assume (at least for a small deformation) that $[H]$ remains the class of an ample divisor. Now, by  Lemma~\ref{chareff}, the effective numerical exceptional curves on $Y$ are exactly the set 
$$\{\alpha \in H^2(Y;\Zee): \alpha^2 = \alpha\cdot K_Y = -1 \text{ and }  \alpha \cdot [H] \geq 0\}.$$
The identification of $H^2(Y;\Zee)$ with $H^2(Y_{\text{irr}};\Zee)$ then identifies the effective numerical exceptional curves on $Y$ with the effective numerical exceptional curves on $Y_{\text{irr}}$. Then by definition 
$\overline{\mathcal{A}}_{\text{\rm{gen}}}(Y)$ is the set of all $x\in \mathcal{C}^+$ such that $x\cdot [D_i] \geq 0$ for all $i$  and  
 $x\cdot \alpha \geq 0$ for all   effective numerical exceptional curves $\alpha$. Since $x\cdot [D_i] \geq 0$ for every $i$ $\iff$ $x\cdot [D_i] \geq 0$ for every $i$ and $x\cdot [D] \geq 0$, we see that 
$\overline{\mathcal{A}}_{\text{\rm{gen}}}(Y)= \{x\in \overline{\mathcal{A}}_{\text{\rm{gen}}}(Y_{\text{\rm{irr}}}): x\cdot [D_i] \geq 0 \text{ for every $i$}\}$ as claimed. 
\end{proof}
 
Summarizing, then, we have shown that $f^*$ is admissible, and hence that $(Y,D)$ and $(Y', D')$ are deformation equivalent.
\end{proof}

\begin{remark} It is natural to ask if there is an elementary approach to the proof of Theorem~\ref{diffisdef} (i.e.\ one that does not use gauge theory).
\end{remark}

\section{Roots}

We turn now to $-2$-curves on $Y$. In the case where $D$ is nef, it is well-known that any class $\beta \in \Lambda(Y,D)$ with $\beta^2=-2$ becomes the class of a $-2$-curve on some deformation of $(Y,D)$ over a connected base. In general, as first shown by Looijenga in \cite{Looij}, this is no longer the case. The following gives a nontrivial condition:

\begin{proposition}\label{reflect}  Let $C$ be a $-2$-curve on $Y$  and let $r_C\colon H^2(Y; \Ar) \to H^2(Y; \Ar)$ be the reflection in the class $[C]$. Then $r_C( \overline{\mathcal{A}}_{\text{\rm{gen}}}) = \overline{\mathcal{A}}_{\text{\rm{gen}}}$. 
\end{proposition}
\begin{proof} Since $r_C(\mathcal{C}^+) =\mathcal{C}^+$ and $r_C([D_i]) = [D_i]$ for every $i$, it suffices  to prove that $r_C$ permutes the set of effective numerical exceptional curves. Let $\alpha$ be a class such that $\alpha^2 =K_Y\cdot \alpha =-1$. Then clearly $r_C(\alpha)$ has these properties since $r_C(K_Y) = K_Y$, and hence $r_C(\alpha)$ is a numerical exceptional curve. By Proposition~\ref{constructnef}, there exists a nef and big divisor $H$ such that $H\cdot D_i > 0$ for every $i$ and such that $H\cdot C = 0$. Thus $r_C([H]) = [H]$, and hence $\alpha\cdot H \geq 0$ $\iff$ $r_C(\alpha) \cdot H \geq 0$. It follows from Lemma~\ref{chareff} that $\alpha$ is effective $\iff$ $r_C(\alpha)$ is effective, and thus that $r_C( \overline{\mathcal{A}}_{\text{\rm{gen}}}) = \overline{\mathcal{A}}_{\text{\rm{gen}}}$.
\end{proof}

\begin{definition} Let $\mathsf{W}({\Delta_Y})$ be the group  generated by the reflections in the classes in the set  $\Delta_Y$ of $-2$-curves on $Y$. Define $R^{\text{\rm{nod}}}_Y = R^{\text{\rm{nod}}}$, the set of \textsl{nodal classes} on $Y$, to be $\mathsf{W}({\Delta_Y})\cdot \Delta_Y$. By standard arguments, if $\mathsf{W}(R^{\text{\rm{nod}}}_Y)$ is the group generated by the reflections in the classes in $R^{\text{\rm{nod}}}_Y$, then $\mathsf{W}(R^{\text{\rm{nod}}}_Y)= \mathsf{W}({\Delta_Y})$.
\end{definition}

Note that $\{W^\beta \cap \overline{\mathcal{A}}_{\text{\rm{gen}}}: \beta \in R^{\text{\rm{nod}}}_Y\}$ is a locally finite closed subset of $\overline{\mathcal{A}}_{\text{\rm{gen}}}$, since $\{W^\beta \cap \mathcal{C}^+: \beta \in \Lambda, \beta^2=-2\}$ is a locally finite closed subset of $\mathcal{C}^+$. In particular, $\mathsf{W}(R^{\text{\rm{nod}}}_Y)=\mathsf{W}({\Delta_Y})$ acts properly discontinuously on  $\overline{\mathcal{A}}_{\text{\rm{gen}}}$.

\begin{corollary}\label{funddomain} The set $\overline{\mathcal{A}}(Y)$ is a fundamental domain for the action of the  group $\mathsf{W}({\Delta_Y})$ on $\overline{\mathcal{A}}_{\text{\rm{gen}}}$.
\end{corollary}
\begin{proof} This is a general result in the theory of groups generated by reflections \cite{Bour}. 
\end{proof}

We will need: 

\begin{lemma}\label{permapinvar} For all $w\in \mathsf{W}({\Delta_Y})$ and all $\alpha \in \Lambda$,  $\varphi_{Y}(w(\alpha)) = \varphi_{Y}(\alpha)$. 
\end{lemma}
\begin{proof} This is clear since, if $C\in \Delta_Y$, then $\varphi_Y([C]) =1$. Hence $\varphi_{Y}(r_C(\alpha)) = \varphi_{Y}(\alpha)$ for all $\alpha \in \Lambda$.
\end{proof}

\begin{definition}\label{defroots} Let $\beta \in \Lambda$, $\beta ^2 = -2$. Then $\beta$ is a \textsl{Looijenga root} (or briefly \textsl{root}) of $Y$ if $r_\beta(\overline{\mathcal{A}}_{\text{\rm{gen}}}) = \overline{\mathcal{A}}_{\text{\rm{gen}}}$, where $r_\beta$ is reflection in the class $\beta$. We denote the set of all roots by $R=R_Y$. Clearly $R^{\text{\rm{nod}}} \subseteq R$. Note that, as opposed to $R^{\text{\rm{nod}}}$, the set  $R$ is a deformation invariant of $Y$.
\end{definition}

\begin{theorem}\label{mainprop}  Let $\beta \in \Lambda$ with $\beta^2 = -2$. Then the following are equivalent:
\begin{enumerate}
\item[\rm(i)] $\beta$ is a root.
\item[\rm(ii)] The wall $W^\beta$ meets the interior of $\overline{\mathcal{A}}_{\text{\rm{gen}}}$.
\item[\rm(iii)] Let $Y_1$ be  a deformation of $Y$ with trivial monodromy such that $\varphi_{Y_1}(\beta) = 1$. Then $\beta\in R^{\text{\rm{nod}}}_{Y_1}$. In particular, if $Y_1$ is generic subject to the condition that $\varphi_{Y_1}(\beta) = 1$ (i.e.\ if $\Ker  \varphi_{Y_1}  = \Zee \cdot \beta$), then $\pm \beta = [C]$ for a $-2$-curve $C$.
\end{enumerate}
\end{theorem}
\begin{proof} (i) $\implies$ (ii): Clearly, if $r_\beta(\overline{\mathcal{A}}_{\text{\rm{gen}}}) = \overline{\mathcal{A}}_{\text{\rm{gen}}}$, then $W^\beta$ meets the interior of $\overline{\mathcal{A}}_{\text{\rm{gen}}}$.

\smallskip
\noindent (ii) $\implies$ (iii): Relabeling $Y_1$ by $Y$, suppose $x$ is an interior point of $\overline{\mathcal{A}}_{\text{\rm{gen}}}$ and that  $x\in W^\beta$. Thus $x\cdot \beta =0$ and $x\cdot D_i > 0$ for every $i$.  By Corollary~\ref{funddomain}, there exists a $w\in \mathsf{W}({\Delta_Y})$ such that  $w(x)\in \overline{\mathcal{A}}(Y)$. In particular, $w(x)\cdot C \geq 0$ for every irreducible curve $C$ on $Y$. Note that, by Lemma~\ref{permapinvar}, we still have $\varphi_Y(w(\beta)) = 1$. We may then replace $\beta$ by $w(\beta)$ and $x$ by $w(x)$.

First we claim that $\pm\beta$ is effective, so that $\pm \beta =\sum_in_i[C_i]$ where the $C_i$ are distinct irreducible curves and $n_i > 0$. In fact, suppose that $\beta$ is not the class of an effective divisor. With $L_\beta$ the line bundle associated to $\beta$, $h^2(Y; L_\beta) = h^0(Y; L_\beta^{-1} \otimes K_Y) = 0$ since $x\cdot (\beta - [D]) < 0$. By assumption, $h^0(Y; L_\beta) =0$. Hence, by Riemann-Roch,  $\chi(Y; L_\beta) = -h^1(Y;L_\beta) =0$, and hence $h^1(Y; L_\beta^{-1} \otimes \scrO_Y(-D)) =0$.  Since $\varphi_Y(\beta) =1$, $L_\beta^{\pm 1}|D =\scrO_D$. From the exact sequence
$$0 \to L_\beta^{-1} \otimes \scrO_Y(-D) \to L_\beta^{-1} \to L_\beta^{-1}|D \to 0,$$
and the fact that $h^1(Y; L_\beta^{-1} \otimes \scrO_Y(-D)) =0$, it follows that the map $H^0(Y; L_\beta^{-1}) \to H^0(D;L_\beta^{-1}|D)$ is surjective. But $L_\beta^{-1}|D =\scrO_D$, so that $H^0(Y; L_\beta^{-1})$ is nonzero. It follows that $-\beta$ is effective.

Writing $\pm \beta   =\sum_in_i[C_i]$, note that $x\cdot [C_i] \geq 0$ for all $i$, since $x$ is nef, and hence $x\cdot [C_i]=0$ for all $i$ since $x\cdot \beta = 0$. In particular, no $C_i$ is a component of $D$. Thus, $C_i \cdot D \geq 0$ for all $i$. But since $(\sum_in_i[C_i]) \cdot [D_i] =0$ for every $i$, $C_i$ and $D$ are disjoint. Finally, since $x\cdot [C_i]=0$, $C_i^2< 0$, so that $C_i$ is a $-2$-curve for every $i$. Since $\beta^2 = (\sum_in_i C_i)^2 = -2$, it follows that $\bigcup_iC_i$ is connected. It is then a standard fact about irreducible, simply laced root systems, that, for every choice of $i$, there exists a $w$ in the reflection group generated by the $r_{C_j}$ such that $w(\beta) = [C_i]$. Since the $r_{C_j} \in \mathsf{W}({\Delta_Y})$, it then follows that $\beta\in R^{\text{\rm{nod}}}_{Y_1}$.

\smallskip
\noindent (iii) $\implies$ (i): This follows by Proposition~\ref{reflect} in case $\beta = [C]$ is the class of a $-2$-curve, and hence in general since $r_\beta$ is a product of reflections in the classes of $-2$-curves.
\end{proof}

\begin{example} (i) Suppose that $E_1$ and $E_2$ are two disjoint interior exceptional curves on $Y$ which both meet the same component $D_i$ of $D$. Then $\beta =[E_1]-[E_2]\in \Lambda$, $\beta ^2=-2$, and $\beta$ is a root by (i) above, since  we can deform $Y$ until the two blowups on points of $D_i$ become infinitely near. As we shall see below, essentially all roots are of this type. 

\smallskip
\noindent (ii) It is very easy to give explicit examples of pairs $(Y,D)$ and elements $\beta \in \Lambda$, $\beta^2=-2$, which are not roots. In principle, such examples were essentially known to Du Val \cite{duV2}; see for instance, \cite[Example 2.19(ii)]{Fried3}. The main point is as follows: Let $(\overline{Y}, \overline{D})$ be an anticanonical pair and let $\alpha \in H^2(\overline{Y}; \Zee)$ be a numerical exceptional curve which is not the class of an exceptional curve. Suppose moreover that $\alpha \cdot D_i=1$ and that $\alpha \cdot D_j =0$ for $j\neq i$. Let $Y$ be the blowup of $\overline{Y}$ at a point of $D_i^{\text{int}}$, with exceptional curve $E$, and let $\beta = \alpha -[E]$, where we identify $H^2(\overline{Y}; \Zee)$ with a subgroup of $H^2(Y; \Zee)$ via pullback. Then $\beta\in \Lambda(Y,D)$ has square $-2$, but $r_\beta([E]) =\alpha$. Thus $r_\beta$ does not preserve the walls of $\overline{\mathcal{A}}_{\text{\rm{gen}}}(Y)$, and hence $r_\beta(\overline{\mathcal{A}}_{\text{\rm{gen}}}(Y)) \neq \overline{\mathcal{A}}_{\text{\rm{gen}}}(Y)$. It follows that $\beta$ is not a root. 
\end{example}

\begin{remark} We could have tried to define $R=R_Y$ to be the set  of all $\beta \in \Lambda$ such that $\beta ^2 = -2$ and such that there exists some deformation of $Y$ for which $\beta$ becomes the class of a $-2$-curve. This definition of $R$ is somewhat awkward, since there is no canonical identification of the cohomologies of the fibers along the deformation (by Corollary~\ref{monoinvar}, the choice of an identification will not in fact matter). In particular, if $\beta = [C]$ is a $-2$-curve on $Y$, then, for a nearby deformation  $Y'$ of $Y$ which is a smoothing of the ordinary double point on the contraction of $C$ on $Y$, the monodromy of the smoothing family sends $[C]$ to $-[C]$, and hence $-\beta \in R$ as well. To avoid this issue, it is simpler to define $R$ as in Definition~\ref{defroots}.
 \end{remark}
 
 The following is an easy consequence of Theorem~\ref{mainprop} and the definitions:

\begin{corollary}\label{preserveamp} \text{\rm{(i)}} If $f\colon H^2(Y; \Zee) \to H^2(Y; \Zee)$ is an admissible integral isometry, then $f(R) = R$.

\smallskip
\noindent \text{\rm{(ii)}} If $\mathsf{W}(R)$ is the reflection group generated by reflections in the elements of $R$, then $\mathsf{W}(R) \cdot R = R$ and $w(\overline{\mathcal{A}}_{\text{\rm{gen}}}) = \overline{\mathcal{A}}_{\text{\rm{gen}}}$ for all $w\in \mathsf{W}(R)$.

\smallskip
\noindent \text{\rm{(iii)}} The set $R^{\text{\rm{nod}}} =\{\beta \in R: \varphi_Y(\beta) = 1\}$.
\qed
\end{corollary} 

We then have the following more precise characterization of generic pairs $(Y,D)$ in the sense of Definition~\ref{defgener}:

\begin{lemma}\label{gener2} The set of generic pairs $(Y,D)$ corresponds to the set
$$\Hom(\Lambda, \mathbb{G}_m) - \bigcup_{\beta\in R}\{\varphi\in \Hom(\Lambda, \mathbb{G}_m): \varphi(\beta) = 1\}. \qed$$
\end{lemma}

\begin{remark} Although the set $\bigcup_{\beta\in R}\{\varphi\in \Hom(\Lambda, \mathbb{G}_m): \varphi(\beta) = 1\}$ is a union of countably many proper subvarieties of $\Hom(\Lambda, \mathbb{G}_m)$, it is not hard to check that it is in general not closed. Indeed, its closure can have a nonempty interior. For a more precise statement when $r\leq 5$, see \cite[II(1.5)]{Looij}.
\end{remark}

Our final goal in this section is to prove the following:

\begin{theorem}\label{maintheorem}
Let $(Y,D)$ be an anticanonical pair and let $\beta$ be a root. Then exactly one of the following holds:
\begin{enumerate}
\item[\rm(i)] The surface $Y$ is the minimal rational ruled surface $\mathbb{F}_0$ or $\mathbb{F}_2$ and $\Lambda$ has rank one. In this case, either   $Y=\mathbb{F}_0\cong \Pee^1\times \Pee^1$ and $\beta$ is  $\pm([f_1]-[f_2])$, where $f_1, f_2$ are the fibers of the two rulings, or $Y = \mathbb{F}_2$ and  $\beta =\pm [\sigma]$, where $\sigma$ is the negative section.
\item[\rm(ii)] For every very general deformation $(Y', D')$ of the pair $(Y,D)$, there exist disjoint exceptional curves $E_1$ and $E_2$ on $Y$, meeting the same component of $D'$, such that $\beta = [E_1]-[E_2]$.
\end{enumerate}
\end{theorem}

\begin{remark} In general, it is possible for a  root to be written as a difference of two exceptional curves in more than one way. For example, if $E_1+E_2$ and $F_1 +F_2$ are two reducible sections of a ruling, with $D$ an irreducible bisection, then $E_1+E_2\equiv F_1 +F_2$ and hence $[E_1] - [F_1] = [F_2] - [E_2]$ is  a  root.
\end{remark}

\medskip
\noindent \textbf{Step I: Preliminary reductions.} 
\medskip

We begin by using the surjectivity of the period map (Theorem~\ref{surjper}).  In particular, after a (not necessarily small) deformation of $(Y,D)$, we can assume that $\varphi_Y(\beta) = 1$ and that $\Ker \varphi_Y = \Zee\beta$. By Theorem~\ref{mainprop}, there exists a smooth rational curve $C$ on $Y$ disjoint from $D$ such that $\pm \beta = [C]$, and $C$ is the unique $-2$-curve on $Y$. Then every irreducible curve of negative self-intersection on $Y$ is either $C$, an interior exceptional curve, or a component  of $D$. We shall show that, if $Y$ is not $\mathbb{F}_0$ or $\mathbb{F}_2$, then there exists an interior exceptional curve $E$ on $Y$  such that $E \cdot C = 1$. In this case, take $E_1 = C+E$ and $E_2 = E$, so that $E_1, E_2$ are generalized exceptional curves, $E_1 \cdot E_2 = 0$, and $C = E_1- E_2$. Then it is easy to check that $H^1(Y; \scrO_Y(C+E)) = 0$, and hence the curve $C+E$ deforms to a very general small deformation of $(Y,D)$, where it is necessarily irreducible and hence an exceptional curve. Similarly the curve $E_2$ deforms to a very general small deformation of $(Y,D)$. Thus, replacing $(Y,D)$ by   such a   deformation, we see that $\beta = [E_1]-[E_2]$, where $E_1$ and $E_2$ are exceptional curves, necessarily disjoint as $E_1\cdot E_2 = 0$. The same will then be true on a generic, not necessarily small deformation of $(Y,D)$.

Thus we must show that, if $(Y,D)$ is an anticanonical pair with a unique smooth rational curve $C$ of self-intersection $-2$ disjoint from $D$, then either there exists an exceptional curve $E$ such that $C\cdot E =1$ or $Y= \mathbb{F}_2$.  
Let us make some easy reductions. First, if there exists an interior exceptional curve $F$ such that $F\cdot C =0$, then let $(\overline{Y}, \overline{D})$ be the anticanonical pair obtained by contracting $F$.   The image of $C$ is a $-2$-curve $\overline{C}$  on $\overline{Y}$, and it is still the unique such curve. Suppose that there exists an interior exceptional curve $\overline{E}$ on $\overline{Y}$  such that $ \overline{C}\cdot\overline{E}=1$. Note that   $\overline{E}$ does not contain the point $p$ of $\overline{Y}$ which is the image of $F$, since otherwise  the proper transform $\overline{E}$ would be a smooth curve of self-intersection $-2$, not equal to $C$, a contradiction. Thus, if $E$ is the proper transform of $\overline{E}$, then $E$ is an interior exceptional curve on $Y$ and   $C\cdot E =1$. Likewise, if $\overline{Y}$ is $\mathbb{F}_0$ or $\mathbb{F}_2$, then necessarily  $\overline{Y}=\mathbb{F}_2$ since it contains the image of $C$, which is necessarily the negative section, and $Y$ is  $\mathbb{F}_2$ blown up at a point $p$ not on the negative section. But then the proper transform of the fiber through $p$ is an exceptional curve meeting $C$ at exactly one point as claimed. Thus we may assume that every interior exceptional curve on $Y$  has nonempty intersection with $C$.

Next suppose that $F$ is an exceptional curve which is a component of $D$ and let $(\overline{Y}, \overline{D})$ be the anticanonical pair obtained by contracting $F$. As before, the image of $C$ is a smooth rational curve $\overline{C}$ of self-intersection $-2$ on $\overline{Y}$. If there exists an interior exceptional curve $\overline{E}$ on $\overline{Y}$  such that $ \overline{C}\cdot\overline{E}=1$, then   $\overline{E}$ does not contain the point $p$ of $\overline{Y}$ which is the image of $F$, since it is a double point of $\overline{D}$. Thus again $C\cdot E =1$, where   $E$, the proper transform of  $\overline{E}$, is an interior exceptional curve on $Y$. Finally, if $\overline{Y}$ is $\mathbb{F}_0$ or $\mathbb{F}_2$, then as in the previous paragraph  $\overline{Y}=\mathbb{F}_2$  and the image of $C$ is the negative section,   $Y$ is  $\mathbb{F}_2$ blown up at a point $p$, and the proper transform of the fiber through $p$ is an interior exceptional curve meeting $C$ at exactly one point.  

Finally, if there are no exceptional curves on $Y$, i.e.\ if $Y$ is minimal, then, by inspection  of the cases in Theorem~\ref{minimalist}, $Y = \mathbb{F}_2$ and $\Lambda$ has rank one. So we have reduced the proof to the following situation:

\begin{assumption}\label{ass1} We may assume that $(Y,D)$ is an anticanonical pair with a unique smooth rational curve $C$ of self-intersection $-2$ disjoint from $D$ such that  every exceptional curve is interior, there exists an exceptional curve $E$ on $Y$ such that $E\cdot C =d\geq 2$, and, for every other exceptional curve $E'$ on $Y$, $C \cdot E' \geq d$. 
\end{assumption}

We will show that Assumption~\ref{ass1} leads to a contradiction.

We fix the following notation for the rest of the proof: As in Assumption~\ref{ass1}, let $d\geq 2$ be the smallest positive integer of the form $C\cdot E$, where $E$ an exceptional curve  on $Y$, and fix once and for all an exceptional curve $E$ such that $E\cdot C =d$.   Since $E\cdot D =1$, there exists a unique component $D_i$ of $D$ such that $E\cdot D_i =1$, and $E\cdot D_j = 0$ for $j\neq i$. Let us record the following facts about the divisor $C+E$ and the corresponding linear system $|C+E|$:

\begin{lemma}\label{lemma1.3} With $C$, $E$, and $d$ as above,
\begin{enumerate}
\item[\rm(i)] $(C+E) \cdot C = d-2\geq 0$.
\item[\rm(ii)] $(C+E) \cdot E = d-1\geq 1$.
\item[\rm(iii)] $(C+E)^2 = 2d-3 > 0$. Hence the divisor $C+E$ is nef and big.
\item[\rm(iv)] $(C+E) \cdot D = 1$. In fact, $(C+E) \cdot D_i = 1$ and $(C+E) \cdot D_j = 0$ for $j\neq i$.
\item[\rm(v)] $h^0(\scrO_Y(C+E)) = d$ and $h^i(\scrO_Y(C+E)) = 0$ for $i>0$.
\end{enumerate}
\end{lemma}
\begin{proof} (i)--(iv) are all straightforward. To see (v), from the exact sequence 
$$0 \to \scrO_Y \to \scrO_Y(E) \to \scrO_E(E)\cong \scrO_{\Pee^1}(-1) \to 0,$$
we see that $h^0(Y; \scrO_Y(E)) = 1$ and that $h^1(Y; \scrO_Y(E))=h^2(Y; \scrO_Y(E)) = 0$. Then, using
$$0 \to \scrO_Y(E) \to \scrO_Y(C+E) \to \scrO_C(C+E) \cong \scrO_{\Pee^1}(d-2) \to 0,$$
we see that $h^0(\scrO_Y(C+E)) = 1 + d-1 =d$ and that $h^i(\scrO_Y(C+E)) = 0$ for $i>0$.
\end{proof}

\medskip
\noindent \textbf{Step II: The case where $D$ is not negative semidefinite.}
\medskip

\noindent \textbf{Case I:} $Y$ is a (generalized) del Pezzo surface, i.e.\  $D$ is nef and big. Define the divisor $G = G^{(1)}$ by: $G = C+E - D$.

\begin{lemma} With notation  as above,
\begin{enumerate}
\item[\rm(i)] $G \cdot D = 1-D^2$. 
\item[\rm(ii)] $G ^2 = 2d-5 + D^2$.
\item[\rm(iii)] $G$ is linearly equivalent to an effective divisor, possibly zero.
\end{enumerate}
\end{lemma}
\begin{proof} (i) and (ii) are clear since $(C+E) \cdot D = 1$. To see (iii),  it follows from Lemma~\ref{LrestrD} that $H^0(Y; \scrO_Y(C+E)) \to H^0(D; \scrO_D(C+E))$ is surjective. As $h^0(Y; \scrO_Y(C+E))  = d$ and $h^0(\scrO_D(C+E)) = 1$, it follows that
$$h^0(Y; \scrO_Y(G)) = h^0(Y;\scrO_Y(C+E)) -1 = d-1 \geq 1,$$
since $d\geq 2$ by assumption. Thus $G$ is effective.
\end{proof}

Since $D$ is nef and big and $G$ is effective, 
$$0 \leq G \cdot D = 1-D^2 \leq 0.$$
Thus $G \cdot D = 0$ and $D^2 = 1$. By the Hodge index theorem, $G^2 \leq 0$, with equality only if $G = 0$. But $G ^2 = 2d-5 + D^2 = 2d-4 \geq 0$, since $d \geq 2$, and hence $G = 0$, $d=2$,  and $C+E$ is linearly equivalent to $D$. In particular, $C+E \in |-K_Y|$. Since $E$ is exceptional and $C\cdot E = 2$, we can contract $E$, producing a new surface $\overline{Y}$ and an irreducible nodal or cuspidal curve $\overline{C}$ which is a section of $|-K_{\overline{Y}}|$, with $(\overline{C})^2 = 2$. Thus $\overline{Y}$ is not a minimal rational surface, so there exists an exceptional curve $\overline{F}$ on $\overline{Y}$, and necessarily $\overline{F} \cdot \overline{C} = 1$. Then the proper transform $F$ of $\overline{F}$ in $Y$ is an exceptional curve such that $F\cdot C =1$, contradicting our assumptions on $Y$ and $C$. Hence this case does not arise.

\bigskip

\noindent \textbf{Case II:}   $D$ is not nef and big, but is not negative semidefinite. In this case, since we have assumed that no component $D_j$ has self-intersection $-1$, there exists a component $D_j$ such that $D_j^2 \geq 0$, since $D$ is not negative semidefinite, and there must also exist a component $D_k$ such that $D_k^2 \leq -3$, since $D$ is not nef. 

If $D_j^2 = 0$, then the linear system $|D_j|$ defines a ruling, i.e.\ a morphism $\pi \colon Y \to \Pee^1$, and since $D_j \cdot C = 0$, $C$ is contained in a fiber of $\pi$. But the hypothesis that $C$ is the unique curve of self-intersection $-2$, and that there does not exist a curve of self-intersection $\leq -3$, implies that there exists another component of the fiber of $\pi$ containing $C$, say $E'$ such that $C\cdot E' =1$. Necessarily $E'$ is an exceptional curve and $C$   meets $E'$ transversally.  This contradicts our assumptions on $Y$.

If $D_j^2 > 0$ and $E'$ is an exceptional curve on $Y$ with $E'\cdot D_j =0$, then $E'\cdot C =0$, since otherwise $D_j \cdot (C+E') = 0$ and $(C+E')^2 \geq 0$, contradicting the Hodge index theorem. As we have assumed that no such curves $E'$ exist, we see that every exceptional curve on $Y$ meets $D_j$ and no other component of $D$. Note that contracting $E'$ does not create any new exceptional curves on $Y$ since the image of $C$ has self-intersection $\geq 0$. We may successively contract exceptional curves until $Y$ becomes a minimal rational surface $\overline{Y}$. In the case, since none of the exceptional curves meet $D_k$, the image $\overline{D}_k$ of $D_k$ is a component of the anticanonical divisor $\overline{D}$ on $\overline{Y}$, and $(\overline{D}_k)^2 = (D_k)^2 \leq -3$. It follows that $\overline{D}_k$ is a section of the ruling on $\overline{Y}$, which is a minimal ruled surface $\mathbb{F}_a$ with $a = -(\overline{D}_k)^2$. Moreover, as $Y\neq \overline{Y}$ by the assumption that there exists an exceptional curve on $Y$, we must have made at least one blow up to reach $Y$. But after blowing up a point $p$, the proper transform of the fiber containing $p$ is an exceptional curve meeting the proper transform of $\overline{D}_k$, a contradiction. Hence this case does not arise.

\medskip
\noindent \textbf{Step III: The case where $D$ is negative semidefinite.} 
\medskip

Suppose that $D$ is negative semidefinite. Then $D^2 \leq 0$; let $-D^2 = e \geq 0$. Moreover, for all $j$, $D\cdot D_j \leq 0$. In this case, for all $k\geq 0$, we define the divisor $G^{(k)}$ by:
$$G^{(k)} = C+E - kD.$$
Thus $G^{(0)} = C+E$ and $G^{(1)} = G$, in the notation of the previous subsection. We record some straightforward properties of $G^{(k)}$:

\begin{lemma}\label{lemma3.1} With $C$, $E$, and $d$ as above,
\begin{enumerate}
\item[\rm(i)] $(G^{(k)})^2 = 2d-3 -2k -k^2e$.
\item[\rm(ii)] $G^{(k)} \cdot D = 1+ke\geq 1$. More precisely,
$$G^{(k)}\cdot D_j = \begin{cases} 1-k(D\cdot D_j) \geq 1, &\text{if $j=i$;} \\
-k(D\cdot D_j) \geq 0, &\text{if $j\neq i$.}
\end{cases}$$
\item[\rm(iii)]  $G^{(k)}\cdot E = d-k-1$. 
\item[\rm(iv)]  If $E'$ is an exceptional curve and $E' \neq E$, then $G^{(k)}\cdot E' \geq d-k$. 
\item[\rm(v)] $G^{(k)}\cdot C = d-2$. 
\item[\rm(vi)] For $k\geq 1$, neither $G^{(k)}-E$ nor $G^{(k)}-C$ is effective.
\end{enumerate}
\end{lemma}
\begin{proof} (i)--(v) are easy calculations. To see (vi), if $G^{(k)} - E = C-kD\equiv  G$, where $G$ is effective, then $C\equiv G+kD$. But this is impossible since $C^2< 0$ and $k\geq 1$. The case where $G^{(k)}-C$ is effective is similar.
\end{proof}

Still assuming that $D$ is negative semidefinite, we shall show:

\begin{claim}\label{mainclaim} For all $k\geq 0$, $G^{(k)}$ is nef and big.
\end{claim}

This immediately leads to a contradiction: if $H$ is an ample divisor, then $H\cdot G^{(k)} > 0$ for all $k$, but $H\cdot G^{(k)} < 0$ provided that $k > H\cdot (C+E)/H\cdot D$.

We prove Claim~\ref{mainclaim} by induction on $k$, starting with the case $k=0$ where it was shown in Lemma~\ref{lemma1.3}. Thus, assume inductively, for $k\geq 1$,  that $G^{(k-1)}$ is nef and big.

\begin{lemma}\label{lemma3.3} With assumptions as above, 
$$h^0(Y; \scrO_Y(G^{(k)})) = \chi(Y; \scrO_Y(G^{(k)})) = d- k -\frac{k(k-1)}{2}e.$$
Thus, $d -   k = k(k-1)e/2+ h^0(Y; \scrO_Y(G^{(k)}))\geq  h^0(Y; \scrO_Y(G^{(k)})) \geq 0$.
\end{lemma}
\begin{proof} By the Riemann-Roch theorem,
$$\chi(Y; \scrO_Y(G^{(k)})) = \frac12((G^{(k)})^2 + G^{(k)}\cdot D) + 1= d- k -\frac{k(k-1)}{2}e.$$
Thus, it suffices to prove that $h^i(Y; \scrO_Y(G^{(k)})) = 0$ for $i>0$. But $G^{(k)} = G^{(k-1)} -D = G^{(k-1)} + K_Y$, and so $h^1(Y; \scrO_Y(G^{(k)}))= h^1(Y;\scrO_Y(-G^{(k-1)})) =0$ by Ramanujam's vanishing theorem and 
$$h^2(Y; \scrO_Y(G^{(k)}))= h^0(Y;\scrO_Y(-G^{(k-1)})) =0$$ since $G^{(k-1)}$ is nef and big.
\end{proof}

\begin{lemma}\label{lemma3.4} The following are equivalent:
\begin{enumerate}
\item[\rm(i)]  $G^{(k)}$ is nef.
\item[\rm(ii)] $(G^{(k)})^2\geq 0$.
\item[\rm(iii)] $\displaystyle d- k -\frac{k(k-1)}{2}e > 0$.
\item[\rm(iv)] $G^{(k)}$ is effective.  
\end{enumerate}
In this case, $G^{(k)} \cdot E'> 0$ for every exceptional curve $E'\neq E$.
\end{lemma}
\begin{proof} (i) $\implies$ (ii): This is true of every nef divisor. (ii) $\implies$ (iii): As in the proof of Lemma~\ref{lemma3.3}, 
$$\frac12((G^{(k)})^2 + G^{(k)}\cdot D) + 1= d- k -\frac{k(k-1)}{2}e.$$
Moreover, $G^{(k)}\cdot D\geq 1$, so that, if $(G^{(k)})^2\geq 0$, then $d- k -k(k-1)e/2 > 0$. (iii) $\implies$ (iv): This follows from Lemma~\ref{lemma3.3}. (iv) $\implies$ (i): Suppose that $G^{(k)}$ is linearly equivalent to $\sum_{i=1}^na_iG_i$, where the $G_i$ are irreducible curves, $G_i\neq G_j$ for $i\neq j$, and the $a_i$ are positive integers. Clearly $G^{(k)}\cdot G \geq 0$ for every irreducible curve $G\neq G_i$ for some $i$. By   Lemma~\ref{lemma3.1}(vi), no $G_i$ can be $E$ or $C$, and hence $G^{(k)}\cdot E\geq 0$, $G^{(k)}\cdot C\geq 0$. If $G_i$ is an irreducible curve such that $(G_i)^2< 0$ and $G_i \neq E$ or $C$, then $G_i$ is  either a component $D_j$ of $D$  or an exceptional curve $E'\neq E$. In the first case, $G^{(k)}\cdot D_j \geq 0$ by Lemma~\ref{lemma3.1}(ii). In the second case, $G^{(k)} \cdot E'\geq d-k \geq h^0(Y; \scrO_Y(G^{(k)}))>0$ by  Lemma~\ref{lemma3.1}(iv) and  Lemma~\ref{lemma3.3}.  Finally, if $G_j^2 \geq 0$, then clearly $G^{(k)}\cdot G_j \geq 0$. Thus $G^{(k)}$ is nef.
\end{proof}

Next we show that, if $G^{(k)}$ is effective, then in fact $G^{(k)}$ is big, i.e.\ the inequality in Lemma~\ref{lemma3.4} is strict.

\begin{lemma}\label{lemma3.5} Suppose that $G^{(k)}$ is effective and hence nef. Then $(G^{(k)})^2> 0$, i.e.\ $G^{(k)}$ is big.
\end{lemma}
\begin{proof} Suppose by contradiction that $(G^{(k)})^2 =0$.  We can write $G^{(k)} = G_m + G_f$, where $G_m$ is the moving part of $|G^{(k)}|$ and $G_f$ is the fixed component. We first show that $G_f =0$. Using 
$$0 = (G^{(k)})^2 = G^{(k)}\cdot G_m + G^{(k)}\cdot G_f = G^{(k)}\cdot G_f +G_f\cdot G_m + (G_m)^2$$
and the fact that $G^{(k)}\cdot G_f \geq 0$ since $G^{(k)}$ is nef, and that $G_m\cdot G_f$ and $(G_m)^2$ are $\geq 0$ since $G_m$ is nef, we see that $G^{(k)}\cdot G_f =G_f\cdot G_m = (G_m)^2 =0$. 
As $G^{(k)}-C$ is not effective,   $C$   is not an irreducible component of $G_f$, and likewise neither is $E$. Hence the irreducible components of $G_f$ are either exceptional curves $E'\neq E$ or components of $D$. If $E'$ is an exceptional curve not equal to $E$, then $G^{(k)}\cdot E'  > 0$ by Lemma~\ref{lemma3.4}. But then $E'$ cannot be a component of $G_f$, for then we would have $G^{(k)}\cdot G_f\geq G^{(k)} \cdot E' >0$. Thus the only possible components of $G_f$ are components $D_j$ of $D$ with $j\neq i$ (recall that $D_i$ is the unique component of $D$ such that $E\cdot D_i =1$ and hence $G^{(k)} \cdot D_i > 0$). Then $G_f =\sum_{j\in J}a_jD_j$, where $J$ is a proper subset of the index set of the components of $D$, possibly empty, and $a_j > 0$ for all $j\in J$. Then
$$0 = (G^{(k)})^2 =  (G_m+G_f)^2 = G_m^2 + 2G_m\cdot G_f + G_f^2 = G_f^2.$$
But $G_f^2 =  (\sum_{j\in J}a_jD_j)^2 < 0$ unless $J =\emptyset$ and $G_f = 0$. Hence $G^{(k)}$ has no fixed components.

We can then apply Theorem~\ref{Lnotbig} to $L=\scrO_Y(G^{(k)})$, as $|G^{(k)}|$ has no fixed components, and in particular no component of $D$ is a fixed component of $|G^{(k)}|$. Since $G^{(k)}\cdot D >0$,  $G^{(k)}$ is linearly equivalent to $nF$ for some positive integer $n$, where $F$ is a smooth rational curve, $F^2 =0$, and $|F|$ is a pencil, which thus defines a morphism $\pi\colon Y \to \Pee^1$. Since $(Y,D)$ is negative definite, there must exist a reducible fiber of $\pi$, and hence an exceptional curve $E'$  such that $F = E'+G$  where $G$ is effective.  But then $G^{(k)}$ is linearly equivalent to $nE'+nG$. By Lemma~\ref{lemma3.1}(vi), $E'\neq E$. Thus  
$$(G^{(k)} )^2 \geq n(G^{(k)} \cdot E') > 0,$$
 contradicting $(G^{(k)} )^2 =0$. So the assumption that $(G^{(k)})^2=0$ leads to a contradiction, and we must have $(G^{(k)})^2> 0$.
\end{proof}

 Combining Lemma~\ref{lemma3.4} and Lemma~\ref{lemma3.5}, we see that either $G^{(k)}$ is nef and big or $d = k + k(k-1)e/2$. So to prove Claim~\ref{mainclaim}, it will suffice to prove:

\begin{lemma}\label{lemma3.6} The case $d = k + k(k-1)e/2$ is impossible.
\end{lemma}
\begin{proof} For $k=1$, $k+k(k-1)e/2 =1$, but by assumption $d\geq 2$. Thus we may assume $k\geq 2$. By the inductive assumption, $G^{(k-1)}$ is nef and big. Consider the linear series $|G^{(k-1)}|$. By abuse  of notation, we denote by $G^{(k-1)}$ a general element of $|G^{(k-1)}|$. Let $\delta = (G^{(k-1)})^2$. We next compute the dimension $N$ of the linear system $|G^{(k-1)}|$. By the inductive hypothesis applied to $G^{(k-2)}$, or directly in case $k=2$ and $G^{(0)} = C+E$, we see that $h^i(Y; \scrO_Y(G^{(k-1)})) = 0$ for $i>0$, and hence that 
$$N=\dim |G^{(k-1)}| = \frac12 ((G^{(k-1)})^2 + G^{(k-1)}\cdot D) =  \frac12 (\delta + G^{(k-1)}\cdot D).$$
On the other hand, setting $g = p_a(G^{(k-1)})$, we have
$$g = \frac12 ((G^{(k-1)})^2 - G^{(k-1)}\cdot D) +1=  \frac12 (\delta - G^{(k-1)}\cdot D)+1.$$
Thus 
$$\delta - N =  \frac12 (\delta - G^{(k-1)}\cdot D) = g-1.$$
Plugging in from Lemma~\ref{lemma3.1} with $k$ replaced by $k-1$, we have
\begin{align*}
\frac12 (\delta - G^{(k-1)}\cdot D) &= \frac12(2d-1-2k-(k-1)^2e-1-(k-1)e) \\
&= d-k -\frac12k(k-1)e -1 = h^0(Y; \scrO_Y(G^{(k)})) -1.
\end{align*}
Thus  $g = \delta-N +1 = d - k - k(k-1)e/2$. Hence $d = k + k(k-1)e/2$ $\iff$ $g=0$ $\iff$ $\delta - N = -1$. Since $g=p_a(G^{(k-1)}) =0$, and hence 
$$G^{(k-1)}\cdot D = 2+ (G^{(k-1)})^2\geq 3,$$
the linear system $|G^{(k-1)}|$ has no fixed components or base locus, by Theorem~\ref{nefdivisors}.   Hence the general element $G^{(k-1)}$ of $|G^{(k-1)}|$ is a smooth rational curve. From the exact sequence
$$0 \to \scrO_Y \to \scrO_Y(G^{(k-1)}) \to \scrO_{G^{(k-1)}}(G^{(k-1)}) \to 0,$$
and the fact that $h^1(Y; \scrO_Y) = 0$, it follows that $|G^{(k-1)}|$ induces a birational morphism $\varphi \colon Y \to \Pee^N$, whose positive dimensional fibers are exactly the curves $G$ such that $G^{(k-1)}\cdot G = 0$. The image surface $\varphi(Y) =\overline{Y}$ has degree $\delta = N-1$, in other words $\overline{Y}$ is a surface of minimal degree in $\Pee^N$. In particular $\overline{Y}$ is normal, and is either a smooth Hirzebruch surface $\mathbb{F}_a$, the normal surface obtained by contracting the  negative section of $\mathbb{F}_a$,  or the Veronese surface and hence $\cong \Pee^2$. By Lemma~\ref{lemma3.1}(vi), since $k\geq 2$, $G^{(k-1)}-E$ is not effective. If $E'$ is an exceptional curve not equal to $E$, then $G^{(k-1)}\cdot E'\geq d-k+1 >0$. In both cases, no exceptional curve can be contracted via $\varphi$. But then either $\varphi$ is an isomorphism or $\varphi$ is the unique  minimal resolution of singularities of $\overline{Y}$. In either case, $Y$ is either $\mathbb{F}_a$ or $\Pee^2$, contradicting the assumption that $(Y,D)$ is negative definite. Thus we have proved Lemma~\ref{lemma3.6}, and hence Claim~\ref{mainclaim} and Theorem~\ref{maintheorem}. 
\end{proof}

\section{Automorphisms I:  an exact sequence}

\begin{definition}\label{autdef} For an anticanonical pair $(Y,D)$, let $\Lambda=\Lambda(Y,D)$  as in Definition~\ref{defLambda} and set $\widehat{\Lambda}= \widehat{\Lambda}(Y,D) = H^2(Y; \Zee)$, so that there is an inclusion $\Lambda \subseteq \widehat{\Lambda}$, in fact an exact sequence
$$0 \to \Lambda \to \widehat{\Lambda} \xrightarrow{f} \Zee^r \to F \to 0,$$
where $\Zee^r$ is the free $\Zee$-module generated by the components $D_i$, $f$ is defined by: the coefficient of $D_i$ in  $f(\alpha)$  is $\alpha\cdot D_i$, and $F$ is the cokernel of $f$. Thus, if the natural homomorphism $\bigoplus_i\Zee[D_i] \to \widehat{\Lambda}$ is a primitive embedding, then  $F=0$. If $(Y,D)$ is negative definite, then $F$ is finite, but for example if $(Y,D)$ is toric, then $F \cong \Zee^2$. If we need to specify the pair $(Y,D)$, we shall write $F(Y,D)$.
\end{definition}

\begin{definition}\label{defsomeauts} For an anticanonical pair $(Y,D)$, let $\Aut^+(Y,D)$ denote the group of automorphisms $\phi$  of $Y$ such that, for all $i$, $\phi(D_i) = D_i$, and such that $\phi$ preserves the orientation of $D$. Thus $\Aut^+(Y,D) = \Aut(Y,D)$ if $r\geq 3$.  Clearly, we have the following: If $(\widetilde{Y}, \widetilde{D})$ is a corner blowup of $(Y,D)$, then $\Aut^+(Y,D) \cong \Aut^+(\widetilde{Y}, \widetilde{D})$.

Note that there is a natural restriction homomorphism $\Aut^+(Y,D) \to \Aut(\widehat{\Lambda})$. We define
$$K = K(Y,D)= \Ker (\Aut^+(Y,D) \to \Aut(\widehat{\Lambda})).$$
In other words, $K$ is the subgroup of $\Aut^+(Y,D)$ which acts trivially on $\widehat{\Lambda}=H^2(Y; \Zee) =\Pic Y$. Finally, we let $\rho\colon \Aut^+(Y,D) \to \Aut^0D$ be the natural restriction homomorphism, and shall also use $\rho$ to denote the restriction of $\rho$ to $K$.
\end{definition}

We then have the following exact sequence \cite[Proposition 2.6]{GHK}:

\begin{theorem}\label{autexactseq} There is an exact sequence
$$0 \to K \xrightarrow{\rho} \Aut^0D \xrightarrow{f} \Hom (\widehat{\Lambda}, \mathbb{G}_m) \to \Hom (\Lambda, \mathbb{G}_m) \to  0,$$
where $\rho$ is as in the previous definition, the map $\Hom (\widehat{\Lambda}, \mathbb{G}_m) \to \Hom (\Lambda, \mathbb{G}_m)$ is the restriction homomorphism, dual to the inclusion $\Lambda \subseteq \widehat{\Lambda}$, and $f$ is defined by: under the canonical isomorphism $\Aut^0D \cong \mathbb{G}_m^r$ of (i) of Lemma~\ref{autisom}, $f$ is identified with the homomorphism
$$g(\lambda_1, \dots, \lambda_r)(\alpha) = \prod_{i=1}^r\lambda_i^{(\alpha\cdot D_i)}.$$
Finally, $K \cong F\spcheck =\Hom(F, \mathbb{G}_m)$, where  $F$ is as in Definition~\ref{autdef}. In particular, $K$ is a deformation invariant of the pair $(Y,D)$.
\end{theorem} 
\begin{proof} Since $\mathbb{G}_m$ is divisible, it is injective, and hence the functor $\Hom(\cdot, \mathbb{G}_m)$ is exact. Applying this to the exact sequence of Definition~\ref{autdef}, we obtain an exact sequence
$$0\to  F\spcheck \to \Hom (\Zee^r, \mathbb{G}_m)\to \Hom (\widehat{\Lambda}, \mathbb{G}_m) \to \Hom (\Lambda, \mathbb{G}_m) \to  0.$$ 
The homomorphism $g\colon \Hom (\Zee^r, \mathbb{G}_m)\to \Hom (\widehat{\Lambda}, \mathbb{G}_m)$ is given by 
$$g(\lambda_1, \dots, \lambda_r)(\alpha) = \prod_{i=1}^r\lambda_i^{(\alpha\cdot D_i)}.$$ Using Lemma~\ref{autisom} to identify $\Hom (\Zee^r, \mathbb{G}_m)$ with $\Aut^0D$ and the homomorphism $g\colon \Hom (\Zee^r, \mathbb{G}_m)\to \Hom (\widehat{\Lambda}, \mathbb{G}_m)$ with $f$ gives the exactness except at the first stage. So if we show that $\rho$ is injective and identify $\Ker f$ with $\rho(K)$,   then  $K\cong  F\spcheck$ under the given identifications. By Corollary~\ref{blowupanddown}, it is enough to show the following: let $*(Y,D)$ be the statement, for the anticanonical pair $(Y,D)$, that $\rho$ is injective and that $\Ker f =\rho(K)$. Then:
\begin{enumerate}
\item[\rm(I)] If $Y\cong \mathbb{F}_0$ and  $D = f_1+\sigma_1+ f_2+\sigma_2$ is the union of two fibers in each of the rulings, then $*(Y,D)$ holds;
\item[\rm(II)] If $(\widetilde{Y}, \widetilde{D}) \to (Y,D)$ is a corner blowup, then  $*(Y,D)$ holds $\iff$ $*(\widetilde{Y}, \widetilde{D})$ holds;
\item[\rm(III)] If $(\widetilde{Y}, \widetilde{D}) \to (Y,D)$ is an interior  blowup and  $*(Y,D)$ holds, then   $*(\widetilde{Y}, \widetilde{D})$ holds.
\end{enumerate}

\noindent Proof of (I): In this case, $\widehat{\Lambda}= H^2(Y; \Zee) \cong \Zee \cdot f \oplus \Zee\cdot \sigma$ and $\Lambda = 0$. The homomorphism $f\colon \Aut^0D \cong \Hom (\Zee^4, \mathbb{G}_m)\to \Hom (\widehat{\Lambda}, \mathbb{G}_m)$ is given by 
$(\lambda_1, \lambda_2, \lambda_3, \lambda_4) \mapsto  \phi$, where $\phi(f) = \lambda_1\lambda_3$ and 
$\phi(\sigma) = \lambda_2\lambda_4$. Thus $\Ker f \cong \mathbb{G}_m^2$ is identified with the automorphisms of $D$ of the form $(\lambda_1, \lambda_2, \lambda_1^{-1}, \lambda_2^{-1})$, $\lambda_i \in \mathbb{G}_m$. On the other hand, since $\Aut^+(\mathbb{F}_0,D) \to  \Aut \widehat{\Lambda}$ is trivial in this case, $K \cong \Aut^+(\mathbb{F}_0,D) =\mathbb{G}_m^2$, and the map $\rho\colon  \Aut^+(\mathbb{F}_0,D)\to  \Aut^0D \cong  \mathbb{G}_m^4$ is clearly injective with image $\Ker f$.

\smallskip
\noindent Proof of (II): In this case,  $\Aut^+(Y,D) \cong \Aut^+(\widetilde{Y}, \widetilde{D})$ and hence $K(\widetilde{Y}, \widetilde{D}) \cong K(Y,D)$. Next, identifying $\Aut^0\widetilde{D}$ with $\Aut^0D \times \mathbb{G}_m$ in the obvious way, there is a commutative diagram  
$$\begin{CD}
K(Y,D) @>{\rho}>> \Aut^0D\cong \mathbb{G}_m^r\\
@V{\cong}VV @VV{\iota}V\\
K(\widetilde{Y}, \widetilde{D}) @>>{\widetilde{\rho}}> \Aut^0\widetilde{D}\cong \mathbb{G}_m^{r+1},
\end{CD}$$
where $\iota(\lambda_1, \dots, \lambda_r) = (\lambda_1, \dots, \lambda_r, \lambda_1 \lambda_r )$. Thus $\widetilde{\rho} =\iota\circ \rho$, so that $\rho$ is injective $\iff$ $\widetilde{\rho}$ is injective.   Moreover, if we denote by $\widetilde{f}$ the homomorphism $\Aut^0\widetilde{D} \to \Hom (\widehat{\Lambda}(\widetilde{Y},\widetilde{D} ), \mathbb{G}_m)$, then under the  identification  $$\Hom (\widehat{\Lambda}(\widetilde{Y},\widetilde{D} ), \mathbb{G}_m) \cong \Hom(\widehat{\Lambda}(Y,D), \mathbb{G}_m) \times \mathbb{G}_m,$$ we can identify $\widetilde{f}$ with the homomorphism
$$\widetilde{f}(\lambda_1, \dots, \lambda_r, \lambda_{r+1}) = (f(\lambda_1, \dots, \lambda_r), \lambda_1\lambda_r\lambda_{r+1}^{-1}).$$
Thus $\Ker \widetilde{f} =\iota(\Ker f)$. Since $\widetilde{\rho}(K(\widetilde{Y}, \widetilde{D})) = \iota(\rho(K(Y,D)))$, we see that $\Ker f = \rho(K(Y,D)$ $\iff$ $\Ker \widetilde{f} = \widetilde{\rho}(K(\widetilde{Y}, \widetilde{D}))$.  Hence $*(Y,D)$ holds $\iff$ $*(\widetilde{Y}, \widetilde{D})$ holds.
  
 \smallskip
\noindent Proof of (III): In this case, if $(\widetilde{Y}, \widetilde{D}) \to (Y,D)$ is an interior  blowup at the point $p\in D_i^{\text{int}}$ with exceptional set $E$, then from the definition of $F(Y,D)$ as the cokernel 
$$\widehat{\Lambda}(Y,D) \to \bigoplus_{j=1}^r\Zee\cdot [D_j] \to F(Y,D) \to 0,$$
we see that there is a surjection
$$\Zee \to  F(Y,D) \to F(\widetilde{Y}, \widetilde{D}) \to 0,$$
where the image of $\Zee$ in $F(Y,D)$ above corresponds to the image of the factor $\Zee\cdot [D_i]$.  Thus there is a commutative diagram with exact rows:
$$\begin{CD}
\bigoplus_{j=1}^r\Zee[D_j] @>>> F(Y,D) @>>> 0\\
@VVV @VVV \\
\bigoplus_{j\neq i}\Zee[D_j] @>>> F(\widetilde{Y}, \widetilde{D}) @>>> 0
\end{CD}$$
Dually, there is a commutative diagram
$$\begin{CD}
\mathbb{G}_m^r @<<<  F\spcheck(Y,D) @<<< 0\\
@AAA @AAA \\
\mathbb{G}_m^{r-1} @<<<  F\spcheck(\widetilde{Y}, \widetilde{D}) @<<< 0
\end{CD}$$
where the vertical map $\mathbb{G}_m^{r-1} \to \mathbb{G}_m^r$ is inclusion into the subgroup of $\mathbb{G}_m^r$ where the $i^{\text{th}}$ component is $1$.

To see the corresponding change for the group $K(Y,D)$,  first note that  $\Aut^0\widetilde{D} \cong \Aut^0D$, compatibly with  the homomorphisms $\rho\colon K(Y,D) \to \Aut^0D$ and $\widetilde{\rho} \colon K(\widetilde{Y}, \widetilde{D}) \to \Aut^0\widetilde{D}$.
If $\widetilde\phi\in K(\widetilde{Y}, \widetilde{D})$, then $\phi(E) = E$, $\widetilde\phi$ induces the identity on the subspace $\widehat{\Lambda}(Y,D)  = (E)^\perp$ of $\widehat{\Lambda}(\widetilde{Y},D)$, and $\widetilde\phi$ descends to an element $\phi$ of $K(Y,D)$ such that $\phi(p) = p$. Conversely, $\phi\in K(Y,D)$ is in the image of  $K(\widetilde{Y}, \widetilde{D})$ $\iff$ $\phi(p) = p$. Thus 
$$K(\widetilde{Y}, \widetilde{D})\cong  \{\phi \in K(Y,D): \phi(p) = p\}.$$
Thus, if $\rho\colon K(Y,D) \to \Aut^0D$ is injective, then so is $\widetilde{\rho}$.   Since any $\phi \in K(Y,D)$ already fixes two non-interior  points of $D_i$, the subgroup $\widetilde{\rho}(K(\widetilde{Y}, \widetilde{D}))$ can be identified with 
$$\{(\lambda_1, \dots, \lambda_r) \in \rho(K(Y,D)): \lambda_i = 1\}.$$ 
This exactly says that, if $\rho(K(Y,D))$ is identified with $ F\spcheck(Y,D)$, then the subgroup $\widetilde{\rho}(K(\widetilde{Y}, \widetilde{D}))$ is identified with $ F\spcheck(\widetilde{Y}, \widetilde{D})$. Hence $*(Y,D)$ implies $*(\widetilde{Y}, \widetilde{D})$, completing the proof of the theorem.
\end{proof}

\begin{corollary}\label{trivauts} For every $i$, $1\leq i\leq r$, choose $p_i\in D_i^{\text{\rm{int}}}$. Suppose that $\phi\in K(Y,D)$, i.e.\ that  $\phi\in \Aut^+(Y,D)$ and that $\phi^*\colon \widehat{\Lambda} \to \widehat{\Lambda}$ is the identity, and that $\phi(p_i) = p_i$ for every $i$. Then $\phi =\Id$.
\end{corollary} 
\begin{proof} The hypothesis $\phi(p_i) = p_i$ for every $i$ implies that $\phi|D_i=\Id$ for every $i$, and hence that $\phi$ is in the kernel of the restriction homomorphism $\rho\colon K \to \Aut^0D$. By Theorem~\ref{autexactseq}, $\rho$ is injective, and hence $\phi=\Id$.
\end{proof}

Motivated by Corollary~\ref{trivauts}, we make the following definition:

\begin{definition}\label{defrigidified} Let $(Y,D)$ be an anticanonical pair. A \textsl{rigidification} of $(Y,D)$ is a choice of $p_i\in D_i^{\text{\rm{int}}}$ for every $i$.  Setting $p=(p_1, \dots, p_r)$, where $p_i\in D_i^{\text{\rm{int}}}$ for every $i$, we shall refer to the triple $(Y,D, p)$ as a \textsl{rigidified} triple. Isomorphisms   of rigidified triples are defined in the obvious way; note that  Corollary~\ref{trivauts} implies that every automorphism of a rigidified triple acting trivially on $\widehat{\Lambda}$ is the identity. 
\end{definition}

\section{Proof of the Torelli theorem}

Let $(Y,D)$ be an anticanonical pair and let $\varphi_Y\colon \Lambda \to \mathbb{G}_m$ be the period map. We begin by describing the possible lifts of the period map, which is an element of $\Hom(\Lambda, \mathbb{G}_m)$, to an element  of $\Hom(\widehat{\Lambda}, \mathbb{G}_m)$.

\begin{lemma}\label{lifts} For each $i$, $1\leq i\leq r$, let $p_i\in D_i^{\text{\rm{int}}}$ and let $p=(p_1, \dots, p_r)$, so that $(Y,D, p)$ is a rigidified triple. Define $\hat\varphi_{Y;p} \colon \widehat{\Lambda} \to \mathbb{G}_m$ by 
$$\hat\varphi_{Y;p}(\alpha) = \psi(L_\alpha|D \otimes \scrO_D(-\sum_{i=1}^r(\alpha\cdot D_i)p_i)), $$
where $\psi \colon \Pic^0D \to \mathbb{G}_m$ is the isomorphism defined by the orientation. Then:
\begin{enumerate}
\item[\rm(i)] $\hat\varphi_{Y;p}$ is a lift of $\varphi_Y \in \Hom(\Lambda, \mathbb{G}_m)$ to $\Hom(\widehat{\Lambda}, \mathbb{G}_m)$. 
\item[\rm(ii)]  Every lift of $\varphi_Y$ to $\Hom(\widehat{\Lambda}, \mathbb{G}_m)$ is equal to $\hat\varphi_{Y;p}$ for some choice of $p_i$.
\item[\rm(iii)] If $p_i'\in D_i^{\text{\rm{int}}}$, $1\leq i\leq r$,  then  $\hat\varphi_{Y;p} = \hat\varphi_{Y;p'}$ $\iff$ there exists a (unique) $\phi\in K(Y,D)$ such that $\phi(p_i)= p_i'$ for every $i$.
 \end{enumerate}
\end{lemma}
\begin{proof} Clearly $\hat\varphi_{Y;p}$ is a homomorphism from $\widehat{\Lambda}$ to $\mathbb{G}_m$ and, if $\alpha \in \Lambda$, then by definition $\hat\varphi_{Y;p}(\alpha) = \varphi_Y(\alpha)$. To see that every lift is of this form, let $\varphi'$ be another such lift. Then $\varphi' \cdot \hat\varphi_{Y;p}^{-1}$ is in the kernel of the restriction map $\Hom (\widehat{\Lambda}, \mathbb{G}_m) \to \Hom (\Lambda, \mathbb{G}_m)$,  and hence $\varphi' \cdot\hat\varphi_{Y;p} ^{-1}= f(\lambda)$, in the notation of Theorem~\ref{autexactseq}. Thus there exist $\lambda_1, \dots, \lambda_r \in \mathbb{G}_m$ such that $\varphi' \cdot\hat\varphi_{Y;p} ^{-1}(\alpha) = \prod_i\lambda_i^{(\alpha\cdot D_i)}$. On the other hand, given $p_i'\in D_i^{\text{\rm{int}}}$ and $p'=(p_1', \dots, p_r')$, clearly 
\begin{align*}
\hat\varphi_{Y;p'}(\alpha) \cdot \hat\varphi_{Y;p}^{-1}(\alpha) &= \psi(\scrO_D((\alpha\cdot D_i)(p_i'-p_i))) = \prod_{i=1}^r(\psi(\scrO_D(p_i'-p_i)))^{(\alpha\cdot D_i)} \\
&= \prod_{i=1}^r\lambda_i^{(\alpha\cdot D_i)},
\end{align*}
where $\lambda_i = \psi(\scrO_D(p_i'-p_i))$. It thus suffices to show that the morphism 
$$D_1^{\text{\rm{int}}} \times \cdots \times D_r^{\text{\rm{int}}}\to \mathbb{G}_m^r$$
defined by
$$(t_1, \dots, t_r) \mapsto (\psi(\scrO_D(t_1-p_1)), \dots, \psi(\scrO_D(t_r-p_r)))$$
is a bijection, which can be seen as follows: By (v) of Lemma~\ref{Loncycle}, given $i$ and the point $p_i$, every line bundle on $D$ of multidegree $(0,\dots, 0)$ on $D$ is of the form $\scrO_D(t_i-p_i)$ for a unique $t_i\in D_i^{\text{\rm{int}}}$.  Hence, given $\lambda_i\in \mathbb{G}_m$, there is a unique $t_i\in D_i^{\text{\rm{int}}}$ such that $\lambda_i = \psi(\scrO_D(t_i-p_i))$. 

Before proving (iii), we note:

\begin{claim}\label{8.2} Let $\phi \in \Aut^0D$. In the notation of Theorem~\ref{autexactseq}, $f(\phi)\cdot \hat\varphi_{Y;p} = \hat\varphi_{Y;\phi(p)}$, where $\phi(p) = (\phi(p_1) , \dots, \phi(p_r))$. 
\end{claim}
\begin{proof} By (iv) of Lemma~\ref{autisom} and the definition of $f$, for all $\alpha \in\widehat{\Lambda}$, 
\begin{align*}
f(\phi)(\alpha) &= \prod_{i=1}^r\psi(\scrO_D(-\phi(p_1)+ p_1))^{(\alpha\cdot D_1)} \cdots  \psi(\scrO_D(-\phi(p_r)+ p_r))^{(\alpha\cdot D_r)} \\
&= \psi(\scrO_D(\sum_{i=1}^r(\alpha\cdot D_i)(-\phi(p_i)+ p_i))).
\end{align*}
Thus, $(f(\phi) \cdot \hat\varphi_{Y;p})(\alpha)$ is equal to
\begin{gather*}
\psi(\scrO_D(\sum_{i=1}^r(\alpha\cdot D_i)(-\phi(p_i)+ p_i)))\cdot \psi(L_\alpha|D \otimes \scrO_D(-\sum_{i=1}^r(L_\alpha\cdot D_i)p_i))\\
 = \psi(L_\alpha|D \otimes \scrO_D(-\sum_{i=1}^r(L_\alpha\cdot D_i)\phi(p_i))) =  \hat\varphi_{Y;\phi(p)}(\alpha).
\end{gather*}
Hence $f(\phi)\cdot \hat\varphi_{Y;p} = \hat\varphi_{Y;\phi(p)}$ as claimed.
\end{proof}

Now suppose that $\hat\varphi_{Y;p} = \hat\varphi_{Y;p'}$.  By (ii) of Lemma~\ref{autisom}, given $p$ and $p'$, there is a unique $\phi\in \Aut^0D$  such that $\phi(p_i) = p_i'$ for all $i$. Then, by the claim above, 
$$f(\phi)\cdot \hat\varphi_{Y;p} = \hat\varphi_{Y;\phi(p)} = \hat\varphi_{Y;p'}=\hat\varphi_{Y;p}.$$
Hence $f(\phi) = 1$, i.e.\ $\phi \in \Ker f$. It follows from Theorem~\ref{autexactseq} that $\phi\in K(Y,D)$.
\end{proof}

 Next we consider the following situation: $(Y, D)$ is obtained from $(\overline{Y}, \overline{D})$ by a sequence of interior blowups, where $(\overline{Y}, \overline{D})$ is some anticanonical pair. We suppose that there exist  nonnegative integers $a_1, \dots, a_r$ such that $Y$ is obtained from $\overline{Y}$ by making $a_i$ interior blowups along $D_i$ for $i=1, \dots r$, and let $\pi\colon Y \to \overline{Y}$ be the blowup morphism. For simplicity, we first consider the case where the blowups are at distinct points $q_1, \dots, q_N$, where $N = a_1+\cdots + a_r$, i.e.\ there are no infinitely near blowups. Given $q_k$, we denote by $i(k)$ the unique $i$ such that $q_k \in D_i^{\text{\rm{int}}}$. Let $E_k$ be the exceptional divisor corresponding to blowing up the point $q_k$.  Then $H^2(Y; \Zee) \cong H^2(\overline{Y}; \Zee) \oplus \Zee^N$, where the ordering of the points $q_i$ gives a natural basis $[E_1], \dots , [E_N]$ of $\Zee^N$. 
 
 In case $(Y,D)$ is obtained from $(\overline{Y}, \overline{D})$ by blowups, some of which are infinitely near, let $q\in D_i^{\text{int}}$ and suppose that we make $b$ infinitely near blowups at $q$. Thus by convention $b$ of the points $q_k$, say  $q_{k_1}, \dots, q_{k_b}$, will be equal to the point $q\in D_i^{\text{\rm{int}}}$, where $i =i(k_j)$ for $j =1, \dots, b$. The result on $Y$ is a generalized exceptional   curve  $C_1+ \cdots + C_{b-1}+ E$ as in Definition~\ref{defgenexcep}. In this case, we again get generalized exceptional curves $E_k = C_k + C_{k+1} +\dots + E$ for $1\leq k\leq b$, with $E_b=E$, $E_k\cdot E_j = 0$, $k\neq j$, and $E_k^2=-1$. Here $E_k$ is the pullback to $Y$ of the exceptional curve in the $k^{\text{th}}$ blowup. Note that $\varphi_Y([C_j]) = 1$ for all $j$, and hence, for any lift $\hat\varphi_{Y;p}$ of $\varphi_Y$, $\hat\varphi_{Y;p}(E_k) =  \hat\varphi_{Y;p}(E_b)=  \hat\varphi_{Y;p}(E)$ for all $i$, $1\leq k \leq b$.
 
 In the above situation, let $(Y,D,p)$ be a rigidification of $(Y,D)$, where $p_i\in D_i^{\text{\rm{int}}}$,  $1\leq i\leq r$, and  $p=(p_1, \dots, p_r)$, and let   $\hat\varphi_{Y;p}$ be the extended period map of Lemma~\ref{lifts}. Restricting $\hat\varphi_{Y;p}$ to the subspace $\Zee^N$ defined above gives a homomorphism $\Zee^N \to \mathbb{G}_m$, and hence by applying $\psi^{-1}$, a homomorphism $\varepsilon \colon \Zee^N \to \Pic^0D$, or equivalently a point,  also denoted by $\varepsilon(\hat\varphi_{Y;p})$, in $(\Pic^0D)^N$. Moreover, by Corollary~\ref{Pic0toDint}, given the point $p_i \in D_i^{\text{\rm{int}}}$, there is an isomorphism $\tau_{p_i} \colon \Pic^0D \to D_i^{\text{\rm{int}}}$. Define   
 $$\tau_p\colon (\Pic^0D)^N \to D_{i(1)}^{\text{\rm{int}}}\times \cdot \times D_{i(N)}^{\text{\rm{int}}}$$
 by mapping the $k^{\text{th}}$ factor of $(\Pic^0D)^N$ to $D_{i(k)}^{\text{\rm{int}}}$ via $\tau_{p_{i(k)}}$. Summarizing, we have the following morphisms
 $$[E_k] \in \Zee^N \mapsto \psi^{-1}\circ \hat\varphi_{Y;p}([E_k]) \in \Pic^0D \mapsto \tau_{p_i(k)}\circ \psi^{-1}\circ \hat\varphi_{Y;p}([E_k]) \in  D_{i(k)}^{\text{\rm{int}}},$$
 and these fit together to give the point $\tau_p(\varepsilon(\hat\varphi_{Y;p}) )\in D_{i(1)}^{\text{\rm{int}}}\times \cdot \times D_{i(N)}^{\text{\rm{int}}}$.

 \begin{lemma}\label{calcextper} In the above notation,
 $$\tau_p(\varepsilon(\hat\varphi_{Y;p}) ) = (q_1, \dots, q_N).$$
 In particular, $\tau_p(\varepsilon(\hat\varphi_{Y;p}) )$ does not depend on $p$.
 \end{lemma}
\begin{proof} By the definition of $\hat\varphi_{Y;p}$, the $k^{\text{th}}$ component of $\varepsilon(\hat\varphi_{Y;p})$ is the line bundle $\scrO_D(q_k - p_{i(k)})$ (this includes the case where some of the blowups are infinitely near). Then $\tau_{p_{i(k)}}(\scrO_D(q_k - p_{i(k)})) = q_k$, since
$$\scrO_D(q_k - p_{i(k)}) \otimes \scrO_D(p_{i(k)})  = \scrO_D(q_k ).$$
Clearly, then,  $\tau_p(\varepsilon(\hat\varphi_{Y;p}) )$ does not depend on $p$.
\end{proof}

Now suppose that we have two anticanonical pairs $(Y, D)$ and $(Y', D')$, both obtained from $(\overline{Y}, \overline{D})$ by a sequence of interior blowups, where $(\overline{Y}, \overline{D})$ is a  taut anticanonical pair, with $\pi\colon Y \to \overline{Y}$ and $\pi'\colon Y' \to \overline{Y}$ the blowdowns. We suppose that there exist  nonnegative integers $a_1, \dots, a_r$ such that $Y$ is obtained from $\overline{Y}$ by making $a_i$ interior blowups along $\overline{D}_i^{\text{\rm{int}}}$ for $i=1, \dots r$, possibly infinitely near, and similarly for $Y'$ (for the same values of $a_i$). As before, we write the points for $Y$ as $q_1, \dots, q_N$, with $q_k\in  \overline{D}_{i(k)}^{\text{\rm{int}}}$, and similarly we denote the points in $Y'$ by $q_1', \dots, q_N'$. The ordering of the points defines a unique isomorphism $\gamma\colon H^2(Y'; \Zee) \cong H^2(Y; \Zee)$ with the property that $\gamma$ identifies the subspace $H^2(\overline{Y}; \Zee)$ of $H^2(Y'; \Zee)$ with the corresponding subspace of $H^2(Y;\Zee)$ and $\gamma([E_i']) = [E_i]$ for every $i$, where $E_i$ is the exceptional curve corresponding to $q_i$ and similarly for $E_i'$.

\begin{lemma}\label{prelimTorellilemma} In the above situation, suppose that the period maps are identified via $\gamma$, i.e.\ that $\varphi_{Y}\circ \gamma =\varphi_{Y'}$. Then there exist an isomorphism $\phi\in K(\overline{Y},\overline{D})$ and an isomorphism $\rho\colon Y\to Y'$ inducing the isomorphism $\gamma$ such that the following commutes:
$$\begin{CD}
Y @>{\rho}>> Y'\\
@V{\pi}VV @VV{\pi'}V\\
\overline{Y} @>{\phi}>> \overline{Y}.
\end{CD}$$
In particular, the pairs $(Y,D)$ and $(Y', D')$ are isomorphic by an isomorphism preserving the orientations of $D$ and $D'$.
\end{lemma}
\begin{proof} Note that the morphisms $\pi$ and $\pi'$ identify $D$ and $D'$ with $\overline{D}$. Choose points $p_i \in D_i^{\text{\rm{int}}}= \overline{D}_i^{\text{\rm{int}}}$, giving a lift $\hat\varphi_{Y;p}$ of $\varphi_Y$. Via the isomorphism $\gamma\colon H^2(Y; \Zee) \cong H^2(Y'; \Zee)$, we can view  $\hat\varphi_{Y;p}$ as a lift of $\varphi_{Y'}$, necessarily of the form $\hat\varphi_{Y';p'}$ for some points $p_i'\in (D_i')^{\text{\rm{int}}}$. Under the identifications of $D$ and $D'$, we can view the $p_i$ and $p_i'$ as points of $\overline{D}_i$. Clearly, restricting $\hat\varphi_{Y;p}$ to the subspace $H^2(\overline{Y}; \Zee)$ gives $\hat\varphi_{\overline{Y};p}$, and similarly for $\hat\varphi_{Y';p'}$. By construction, 
$$\hat\varphi_{Y;p}\circ \gamma|H^2(\overline{Y}; \Zee) = \hat\varphi_{Y';p'}|H^2(\overline{Y}; \Zee),$$
so that $\hat\varphi_{\overline{Y};p} = \hat\varphi_{\overline{Y};p'}$ after viewing $p_i$ and $p_i'$ as points of $\overline{D}_i$. By (iii) of Lemma~\ref{lifts} applied to the pair $(\overline{Y},\overline{D})$, there exists a unique $\phi\in K(\overline{Y},\overline{D})$ such that $\phi(p_i) = p_i'$ for every $i$. We will construct an isomorphism $\rho \colon Y \to Y'$ such that $\pi'\circ \rho =\phi \circ\pi$.  Clearly, it suffices to show the following: if $q_1, \dots, q_N$ is the ordered set  of $a_i$ points blown up in $\overline{D}_i$ to obtain $Y$, and  $q_1', \dots, q_N'$ is the corresponding set for $Y'$, then, for every $i$, $q_i' =\phi(q_i)$. By Lemma~\ref{calcextper} and the fact that $\gamma([E_i']) = [E_i]$ for every $i$,
$$(q_1', \dots, , q_N') = \tau_{p'}(\varepsilon(\hat\varphi_{Y';p'})) = \tau_{p'}(\varepsilon(\hat\varphi_{Y;p})).$$
On the other hand, by (v) of Lemma~\ref{autisom},
$$\tau_{p'}(\varepsilon(\hat\varphi_{Y;p})) = \tau_{\phi(p)}(\varepsilon(\hat\varphi_{Y;p}))= \phi(\tau_{p}(\varepsilon(\hat\varphi_{Y;p}))) =(\phi(q_1), \dots,\phi(q_N)).$$
Thus $q_i' = \phi(q_i)$ as claimed. 
\end{proof}

\begin{theorem}[Torelli Theorem I]\label{firstTorelli} Let $(Y,D)$ and $(Y', D')$ be two labeled anticanonical pairs with $r(D) = r(D')$. Suppose that 
$$\gamma\colon   H^2(Y';\Zee) \to  H^2(Y;\Zee)$$ is an admissible integral isometry such that
\begin{enumerate}
\item[\rm(i)]  For all $i$,  $\gamma([D_i'])=[D_i]$.
\item[\rm(ii)]  Denoting the extension of $\gamma$ to an isometry $H^2(Y';\Ar)\to H^2(Y;\Ar)$ by $\gamma$ as well, we have $\gamma(\overline{\mathcal{A}}(Y'))= \overline{\mathcal{A}}(Y)$. 
\item[\rm(iii)] $\varphi_Y\circ \gamma =\varphi_{Y'}$.
\end{enumerate}
Then there is an isomorphism of labeled pairs $\rho\colon (Y,D) \to (Y',D')$, compatible with the orientations,  such that $\rho^* = \gamma$. Moreover,  if $\rho$ and $\rho'$ are two such isomorphisms, then there exists a $\phi \in K(Y', D')$ such that $\rho' =\phi \circ \rho$. Conversely, if $\phi \in K(Y', D')$, then $\rho' =\phi \circ \rho$ is an isomorphism from $Y$ to $Y'$ such that $(\rho')^* = \gamma$.
\end{theorem}
\begin{proof} We shall just write down the proof under the assumption that $(Y',D')$ is an interior blowup of a taut pair $(\overline{Y}', \overline{D}')$. The only remaining case is $Y'=\mathbb{F}_0$ or $Y'=\mathbb{F}_2$, with self-intersection sequence $(8)$ or $(2,2)$, and hence $Y$ is also either $\mathbb{F}_0$ or $\mathbb{F}_2$, and these cases can be handled directly, or reduced to the case where $(Y,D)$ is an interior blowup of a taut pair $(\overline{Y}, \overline{D})$ after making a single corner blowup.

First assume that $(Y',D')$ is taut, and hence that $(Y, D)\cong (Y',D')$ is also taut. By Remark~\ref{tautrigid},  $\Lambda(Y,D) =0$. The statement then reduces to the assertion that every integral isometry $\gamma\colon  H^2(Y';\Zee) \to  H^2(Y;\Zee)$ satisfying (i) and (ii) of the statement is induced by an isomorphism $\rho\colon (Y,D)\to (Y',D')$. This is easily reduced to the case where $(Y,D)$ and $(Y', D')$ are minimal ruled. In this case it is clear by inspection that there is exactly one integral isometry $\gamma\colon   H^2(Y';\Zee) \to  H^2(Y;\Zee)$ such   $\gamma([D_i'])=[D_i]$ for every $i$, and such a $\gamma$ is realized by an isomorphism of pairs 
$(Y', D')\to (Y,D)$, unique up to the action of $K(Y,D)$. Thus we may assume that $(Y',D')$ blows down to a taut pair, but is not itself taut.

The assumption that $\gamma(\overline{\mathcal{A}}(Y'))= \overline{\mathcal{A}}(Y)$ implies that $\alpha$ is the class of an interior exceptional curve $E'$ on $Y'$ $\iff$ $\gamma(\alpha)$ is the class of an interior exceptional curve $E$ on $Y$, and that $\beta$ is the class of a $-2$ curve on $Y'$ $\iff$ $\gamma(\beta)$ is the class of a  $-2$ curve on $Y$. Moreover, if $\alpha =[E']$ is an exceptional curve such that $E'\cdot D_i' =1$, then the corresponding exceptional curve $E$ on $Y$ satisfies $E\cdot D_i = 1$. If $Y_1'$ is the surface obtained by contracting such an interior exceptional curve $E'$ with $E'\cdot D_i'=1$, and $Y_1$ is the surface obtained by contracting the exceptional curve corresponding to $\gamma([E'])$, then $\gamma$ defines an integral isometry $\gamma_1$ from $H^2(Y';\Zee)$ to $H^2(Y; \Zee)$. If $\alpha_1$ is a wall of $\overline{\mathcal{A}}(Y_1')$ corresponding to an exceptional curve $F_1'$ on $Y_1'$, then the proper transform $F'$ of $F_1'$ in $Y'$ is either an exceptional curve disjoint from $E'$ or a $-2$ curve $C'$ such that $C'\cdot E' = 1$. Thus $\gamma([F'])$ is either the class of an exceptional curve $F$ disjoint from $E$ or a $-2$ curve $C$ such that $C\cdot E = 1$. Hence $\gamma(\alpha_1)$ is a wall of $\overline{\mathcal{A}}(Y_1)$. Likewise, if $\beta_1$ is a wall of $\overline{\mathcal{A}}(Y_1')$ corresponding to a $-2$ curve $C_1'$ on $Y_1'$, then the proper transform $C'$ of $C_1'$ in $Y'$ is a $-2$ curve, and $\gamma([C']) = [C]$, where $C$ is a $-2$ curve on $Y$ disjoint from $E$. It follows that $\gamma_1$ identifies the walls of $\overline{\mathcal{A}}(Y_1')$ with a subset of the walls of $\overline{\mathcal{A}}(Y_1)$ and a symmetric argument with $\gamma^{-1}$ shows that 
$$(\gamma_1)(\overline{\mathcal{A}}(Y_1')) = \overline{\mathcal{A}}(Y_1).$$

By an obvious inductive argument, given the blowdown $\pi'\colon (Y', D') \to (\overline{Y}', \overline{D}')$, there is an analogous blowdown  $\pi \colon (Y,D)\to(\overline{Y}, \overline{D})$ which has the same self-intersection sequence as $(\overline{Y}', \overline{D}')$, and hence is isomorphic to $(\overline{Y}', \overline{D}')$. Moreover, $\gamma$ induces an integral isometry $H^2(\overline{Y}'; \Zee) \to H^2(\overline{Y}; \Zee)$ satisfying the assumptions of Theorem~\ref{firstTorelli}. By the taut case described in the second paragraph of the proof, there exists an isomorphism $(\overline{Y}', \overline{D}') \to (\overline{Y}, \overline{D})$ inducing $\gamma$, and which we may then use to identify $(\overline{Y}', \overline{D}')$ with $(\overline{Y}, \overline{D}))$.  We are then in the situation described prior to Lemma~\ref{prelimTorellilemma}, where $Y$ and $Y'$ are both interior blowups of the same pair $(\overline{Y}, \overline{D})$ at $a_i$ points on $\overline{D}_i^{\text{\rm{int}}}$, and $\gamma$ induces an isometry from the image of $H^2(\overline{Y};\Zee)$ in $H^2(Y';\Zee)$ to the image of $H^2(\overline{Y};\Zee)$ in $H^2(Y;\Zee)$.   By construction,  the given isometry $\gamma$  from  $H^2(Y'; \Zee)$ to $H^2(Y; \Zee)$ is the one constructed just prior to the statement of Lemma~\ref{prelimTorellilemma}. By assumption, $\varphi_{Y}\circ \gamma =\varphi_{Y'}$. Thus, by Lemma~\ref{prelimTorellilemma}, there is an isomorphism $\rho\colon (Y,D) \to (Y', D')$, which satisfies $\rho^* =\gamma$.

Finally, if $\rho$ and $\rho'$ are  isomorphisms from $Y$ to $Y'$ such that $\rho^* = (\rho')^*$, then $\rho'\circ \rho^{-1} =\phi$ is an element of $\Aut^+(Y', D')$ acting trivially on $H^2(Y'; \Zee)$, and hence lies in $K(Y', D')$, and $\rho' =\phi\circ \rho$. Conversely, if $\phi \in K(Y', D')$, then $\phi\circ \rho$ is an isomorphism such that $(\phi\circ \rho)^* = \rho^*\circ \phi^* =\rho^*=\gamma$, as claimed.
\end{proof}

\begin{remark}\label{tautremark} In the proof of Theorem~\ref{firstTorelli}, we could assume that $(\overline{Y}, \overline{D})$ satisfies a weaker notion than tautness, since we know not only the self-intersection sequence of $\overline{D}$ but also the  ample cone of $\overline{Y}$, and can then reconstruct the pair $(\overline{Y}, \overline{D})$ from this data.
\end{remark}

\begin{theorem}[Torelli Theorem II]\label{secondTorelli} Let $(Y,D)$ and $(Y', D')$ be two labeled anticanonical pairs with $r(D) = r(D')$. Suppose that 
$$\gamma\colon  H^2(Y';\Zee) \to  H^2(Y;\Zee)$$ is an integral isometry such that\begin{enumerate}
\item[\rm(i)]  For all $i$,  $\gamma([D_i'])=[D_i]$.
\item[\rm(ii)]   $\gamma(\overline{\mathcal{A}}_{\text{\rm{gen}}}(Y'))= \overline{\mathcal{A}}_{\text{\rm{gen}}}(Y)$ (i.e.\ $\gamma$ is admissible). 
\item[\rm(iii)] $\varphi_Y\circ \gamma =\varphi_{Y'}$.
\end{enumerate}
Then there exists a unique   $w\in \mathsf{W}({\Delta_Y})$ and an isomorphism of labeled pairs $\rho\colon (Y,D) \to (Y',D')$, compatible with the orientations and unique up to composing with an element of $K(Y', D')$, such that $\rho^* = w\circ \gamma$.
\end{theorem}
\begin{proof} Since $\gamma(\overline{\mathcal{A}}_{\text{\rm{gen}}}(Y'))= \overline{\mathcal{A}}_{\text{\rm{gen}}}(Y)$, $\gamma(R_{Y'}) = R_Y$. Similarly, if $\beta' \in R_{Y'}$ and $\varphi_{Y'}(\beta') = 1$, then $\gamma(\beta') =\beta \in R_Y$ and $\varphi_Y(\beta) = 1$. By (iii) of Corollary~\ref{preserveamp},  the condition that $\beta' \in R_{Y'}$ and $\varphi_{Y'}(\beta') = 1$ is equivalent to the condition that $\beta'\in R^{\text{nod}}_{Y'}$ and hence $\beta' \in R^{\text{nod}}_{Y'}$ $\iff$ $\beta \in R^{\text{nod}}_Y$. Moreover, $\gamma$ is equivariant with respect to the actions of $\mathsf{W}({\Delta_{Y'}})$ on $H^2(Y';\Ar)$ and $\mathsf{W}({\Delta_Y})$ on $H^2(Y; \Ar)$, in the sense that, if $\beta' \in R^{\text{nod}}_{Y'}$ and  $\gamma(\beta') =\beta$, then  
$$\gamma\circ r_{\beta'} = r_\beta \circ \gamma.$$ 
In particular, since $\overline{\mathcal{A}}(Y')$ is a fundamental domain for the action of $\mathsf{W}({\Delta_{Y'}})$ on $\overline{\mathcal{A}}_{\text{\rm{gen}}}(Y')$, there exists a unique $w\in \mathsf{W}({\Delta_Y})$ such that $w\circ \gamma (\overline{\mathcal{A}}(Y'))= \overline{\mathcal{A}}(Y)$. By Lemma~\ref{permapinvar}, it is still the case that $\varphi_Y\circ (w\circ\gamma) =\varphi_{Y'}$. Thus, by the  Torelli Theorem I, there exists an isomorphism $\rho\colon Y \to Y'$, unique up to composing with an element of $K(Y', D')$ such that $\rho^* = w\circ \gamma$.
\end{proof}

As a corollary, there is a rigidified version using a choice of points $p_i\in D_i^{\text{\rm{int}}}$ as well as the corresponding lift of the period map.

\begin{theorem}[Torelli Theorem III]\label{thirdTorelli} Let $(Y,D)$ and $(Y', D')$ be two labeled anticanonical pairs with $r(D) = r(D')$. For each $i$, $1\leq i \leq r$, let $p_i\in D_i^{\text{\rm{int}}}$, with $p=(p_1, \dots, p_r)$,  and similarly let $p_i'\in (D_i')^{\text{\rm{int}}}$. Suppose that 
$$\gamma\colon  H^2(Y';\Zee) \to  H^2(Y;\Zee)$$ is an integral isometry such that
\begin{enumerate}
\item[\rm(i)]  For all $i$,  $\gamma([D_i'])=[D_i]$.
\item[\rm(ii)] $\gamma(\overline{\mathcal{A}}(Y'))= \overline{\mathcal{A}}(Y)$ resp.\   $\gamma(\overline{\mathcal{A}}_{\text{\rm{gen}}}(Y'))= \overline{\mathcal{A}}_{\text{\rm{gen}}}(Y)$. 
\item[\rm(iii)] $\hat\varphi_{Y;p}\circ \gamma =\hat\varphi_{Y';p'}$.
\end{enumerate}
Then there is a unique isomorphism of labeled pairs $\rho\colon (Y,D) \to (Y',D')$, compatible with the orientations,  such that $\rho(p_i) =p_i'$ and $\rho^* = \gamma$, resp.\    there exists a unique isomorphism $\rho\colon (Y,D) \to (Y',D')$  as above  such that $\rho(p_i) =p_i'$ and a unique  $w\in \mathsf{W}({\Delta_Y})$  such that  $\rho^* = w\circ \gamma$.
\end{theorem}
\begin{proof} We shall just write out the case where  $\gamma(\overline{\mathcal{A}}_{\text{\rm{gen}}}(Y'))= \overline{\mathcal{A}}_{\text{\rm{gen}}}(Y)$. By Theorem~\ref{secondTorelli}, there exists an isomorphism $\rho\colon (Y,D) \to (Y',D')$ of labeled pairs, compatible with the orientations, and a unique $w\in \mathsf{W}({\Delta_Y})$ such that $\rho^* = w\circ \gamma$. Using $\rho$ to identify $(Y,D)$ and  $(Y', D')$, we have points $p_i, \rho^{-1}(p_i')\in D_i^{\text{\rm{int}}}$ such that $\hat\varphi_{Y;p}  =\hat\varphi_{Y;\rho^{-1}(p')}$. By Lemma~\ref{lifts}(iii), there exists a unique $\phi\in K(Y,D)$ such that $\phi(p_i) = \rho^{-1}(p_i')$ for every $i$. Replacing $\rho$ by $\rho\circ \phi$ then gives an isomorphism of labeled pairs as desired, and it is unique by Corollary~\ref{trivauts}.
\end{proof}

We turn next to a version of the Torelli theorem in families. Because the group $K(Y,D)$ may well be nontrivial, even in the negative definite case, we cannot expect the Torelli theorem to hold in families. For example, the period map could be constant, so that all fibers are isomorphic to a fixed pair $(Y,D)$, but the total space could be induced from a nontrivial principal $K(Y,D)$-bundle. 
Thus it is essential to use the rigidified version of the Torelli theorem, Theorem~\ref{thirdTorelli}.
There are then analogues of results for $K3$ surfaces first established by Burns-Rapoport \cite{BurnsRapoport} (see also \cite{LooijengaPeters} or \cite{Beauville}), and we shall just sketch the corresponding arguments in our case.  

Fix a deformation type of pairs $(Y,D)$, and hence a lattice $\widehat{\Lambda}$,   a generic ample cone $\overline{\mathcal{A}}_{\text{\rm{gen}}}\subseteq \widehat{\Lambda} \otimes _\Zee\Ar$, and a corresponding set $R$ of roots.  As in \S3, for a reduced connected analytic space $S$, we consider  rigidified  families $(\mathcal{Y}, \mathcal{D}, \sigma)$ over $S$ within the given deformation equivalence class. 

\begin{definition} A \textsl{rigidified} family   $(\mathcal{Y}, \mathcal{D}, \sigma)$ over $S$ is a family  $\pi \colon (\mathcal{Y}, \mathcal{D} )\to S$ as in \S3 and an $r$-tuple $\sigma = (\sigma_1, \dots, \sigma_r)$ such that, for every $i$, $\sigma_i$ is a section of the morphism $\pi|\mathcal{D}_i \colon \mathcal{D}_i\to S$, whose image lies in the smooth locus of $\mathcal{D}$. Note that, as there are three labeled disjoint sections of $\pi|\mathcal{D}_i$, there is a canonical $S$-isomorphism $\mathcal{D}_i \cong S\times \Pee^1$, for $r\geq 2$, and similarly for the normalization of $\mathcal{D}$ when $r=1$. 
\end{definition}

We can also consider admissible markings $\theta$ of the family $(\mathcal{Y}, \mathcal{D}, \sigma)$ in the sense of Definition~\ref{defadmmarking} (for the underlying family over $S^{\text{red}}$).  Such a quadruple $(\mathcal{Y}, \mathcal{D}, \sigma, \theta)$ will be called \textsl{rigidified and admissibly marked}. Given a (not necessarily reduced or connected) analytic space $S$, the set of such quadruples $(\mathcal{Y}, \mathcal{D}, \sigma, \theta)$ over $S$    defines an (analytic) stack $\widehat{\mathbf{M}}$ in the obvious way, and there is a morphism from this stack to the stack $\mathbf{M}$ of triples $(\mathcal{Y}, \mathcal{D}, \theta)$, whose fiber over a family $\pi\colon \mathcal{Y}\to S$ is a principal homogeneous space over the group $\Hom(S, \mathbb{G}_m^r)/\Hom (S, K(Y,D))$.  In terms of functors of Artin rings, we have the corresponding functor $\mathbf{Def}_{Y; D_1, \dots,D_r;p}$, and it is prorepresented by a Kuranishi space $(\widehat {T},0)$. Here, the germ $(\widehat {T},0)$ is smooth, there is a smooth morphism $(\widehat {T},0)\to (T, 0)$, where $(T,0)$ is the Kuranishi space for the functor $\mathbf{Def}_{Y; D_1, \dots,D_r}$, and the fiber over $0$ is the germ $(\Sigma, \Id)$, where $\Sigma \subseteq \Aut^0D \cong \mathbb{G}_m^r$ is any germ at $\Id$ of a submanifold which is a slice to the quotient homomorphism $\Aut^0D \to \Aut^0D/K(Y,D)$.

Given a rigidified and admissibly marked family $(\mathcal{Y}, \mathcal{D}, \sigma, \theta)$ over a reduced connected base $S$, we can define the extended period map 
$$\widehat{\Phi}_S \colon S \to \Hom(\widehat{\Lambda}, \mathbb{G}_m),$$
 where $\widehat{\Lambda} =H^2(Y_s;\Zee)$ for a fixed fiber of $\pi$.   

\begin{lemma} $\widehat{\Phi}_S$ is holomorphic. There exists a smooth connected base $S$ and a rigidified and admissibly marked family over $S$ for which $\widehat{\Phi}_S$ is surjective. For the Kuranishi space $(\widehat {T},0)$ described above, there is a commutative diagram
$$\begin{CD}
(\Sigma, \Id) @>>> \mathbb{G}_m^r/K(Y,D)\\
@VVV @VVV\\
(\widehat {T},0) @>{\widehat{\Phi}_{\widehat {T}}}>> \Hom(\widehat{\Lambda}, \mathbb{G}_m)\\
@VVV @VVV\\
(T,0) @>{\Phi_T}>> \Hom(\Lambda, \mathbb{G}_m),
\end{CD}$$
where the morphism $(\Sigma, \Id) \to \mathbb{G}_m^r/K(Y,D)$ is identified with the local isomorphism given by projection of the slice to the quotient. 
Hence, viewing $\widehat{\Phi}_{\widehat {T}}$ as a map from $(\widehat {T},0)$ to $\Hom(\widehat{\Lambda}, \mathbb{G}_m)$, the differential of $\widehat{\Phi}_{\widehat {T}}$ is injective, i.e.\ the local Torelli theorem holds.  
\end{lemma}
\begin{proof} One checks that $\widehat{\Phi}_{\widehat {T}}$ is holomorphic by an argument similar to that used to prove Theorem~\ref{perholom}. The surjectivity for an appropriate $S$ follows from Lemma~\ref{lifts} and the  surjectivity of the usual period map. Finally, to check the commutativity of the diagram and the rest of the assertions of the lemma, note that $\Aut^0D$ acts transitively on the fiber over $0$. More precisely, given the lift $(Y,D,p)$ of the pair $(Y,D)$ and $\phi\in \Aut^0D$, every possible lift is of the form $\phi\cdot (Y,D,p) = (Y,D, \phi(p))$, and the isotropy group at $(Y,D,p)$  is $K(Y,D)$. By Claim~\ref{8.2}, $\hat{\varphi}_{Y;\phi(p)} =f(\phi)\cdot \hat{\varphi}_{Y;p}$, identifying the fiber with $\mathbb{G}_m^r/K(Y,D)$, viewed as a subset of $\Hom(\widehat{\Lambda}, \mathbb{G}_m)$ via the action of multiplication.
\end{proof}

Let $\widehat{M}$ be the set of (isomorphism classes of) rigidified and admissibly marked pairs $(Y,D, p, \theta)$.

\begin{lemma} There is a natural structure of a (non-separated) complex manifold on $\widehat{M}$ and a universal family over $\widehat{M}$ so that $\widehat{M}$ is a fine moduli space, i.e.\  the stack $\widehat{\mathbf{M}}$ is representable by the space $\widehat{M}$. 
\end{lemma}
\begin{proof} This follows from the arguments of \cite[(2.1)]{BurnsRapoport}, \cite[p.\ 142]{Beauville}, or \cite[(10.2)]{LooijengaPeters}, using the fact that rigidified and admissibly marked pairs have no nontrivial automorphisms (Corollary~\ref{trivauts}).
\end{proof}

\begin{remark} Given a taut pair $(\overline{Y}, \overline{D})$ such that $(Y,D)$ is an interior blowup of $(\overline{Y}, \overline{D})$, there are clearly elementary constructions (as in the proof of Theorem~\ref{surjper}) of a separated moduli space which is a product of copies of $\mathbb{G}_m$ and a ``universal" family over this space. It is not a fine moduli space if there are $-2$-curves on some deformation of $Y$, since, as in the $K3$ case,  the fine moduli space $\widehat{M}$ is  not separated if there are $-2$-curves on $Y$. 
\end{remark}

Let $\widehat{\Omega} =\Hom(\widehat{\Lambda}, \mathbb{G}_m)$ be the corresponding period space. Then via the construction of \cite[p.\ 243]{BurnsRapoport}, \cite[p.\ 145]{Beauville}, or \cite[p.\ 183]{LooijengaPeters}, there is a (non-separated) complex manifold $\widetilde{\widehat{\Omega}}$ and a holomorphic, \'etale map  $\widetilde{\widehat{\Omega}} \to \widehat{\Omega}$, whose fiber over a point $\varphi\in \Hom(\widehat{\Lambda}, \mathbb{G}_m)$ consists of the connected components of $\overline{\mathcal{A}}_{\text{\rm{gen}}} - \bigcup_{\beta\in R^{\text{nod}}_{\varphi}}W^\beta$, where 
$$R^{\text{nod}}_{\varphi} =\{\beta \in R: \varphi(\beta) = 1\}.$$
 (By  Corollary~\ref{preserveamp}(iii), if $\varphi =\varphi_Y$ for some pair $(Y,D)$, then $R^{\text{nod}}_{\varphi} = R^{\text{nod}}_{Y}$.) More precisely, let $K\widehat{\Omega}$ be the set of pairs $(\varphi, x)\in \widehat{\Omega}\times \mathcal{A}_{\text{\rm{gen}}}$ such that $x$ is not orthogonal to any $\beta \in R$ such that $\varphi(\beta) = 1$. Thus the fiber of $K\widehat{\Omega} \to \widehat{\Omega}$ over $\varphi$ is $\mathcal{A}_{\text{\rm{gen}}}(Y) - \bigcup_{\beta\in R^{\text{nod}}_{\varphi}}W^\beta$. It follows easily from the local finiteness of the walls $W^\beta$ that $K\widehat{\Omega}$ is an open subset of $\widehat{\Omega}\times \mathcal{A}_{\text{\rm{gen}}}$. Define $\widetilde{\widehat{\Omega}}$ to be the quotient of $K\widehat{\Omega}$ by the equivalence relation
  $(\varphi_1, x_1) \sim (\varphi_2, x_2)$ $\iff$ $\varphi_1 = \varphi_2$ and $x_1, x_2$ are in the same connected component of the fiber of $K\widehat{\Omega} \to \widehat{\Omega}$.  Note that the fiber of $\widetilde{\widehat{\Omega}} \to \widehat{\Omega}$ over $\varphi$ consists of one point $\iff$ there are no $\beta \in R$ such that $\varphi(\beta) = 1$; if $\varphi =\varphi_Y$, this is equivalent to the condition that there are no $-2$-curves on $Y$.

The period map $\widehat{\Phi}_{\widehat{M}}=\widehat{\Phi}\colon \widehat{M} \to \widehat{\Omega}$ lifts to a function $\widetilde{\widehat{\Phi}}\colon \widehat{M} \to \widetilde{\widehat{\Omega}}$, by sending $(Y, D, p, \theta)$ to the point $\hat{\varphi}_{Y;p}$ and to the connected component  of $\mathcal{A}_{\text{\rm{gen}}}(Y) - \bigcup_{\beta\in R^{\text{nod}}_{\varphi}}W^\beta$ defined by $ \mathcal{A}(Y)$.

\begin{theorem}\label{extendedpermap}  The lifted extended period map $\widetilde{\widehat{\Phi}}$ is a morphism, and in fact an isomorphism of non-separated complex manifolds.   
\end{theorem}
\begin{proof} The main point is to check that $\widetilde{\widehat{\Phi}}$ is continuous in the given topologies. This follows from the ``openness of the K\"ahler cone" (\cite[pp.\ 118--119]{Beauville}, \cite[\S8]{LooijengaPeters}):

\begin{lemma}\label{Kconeopen} Let $\pi\colon (\mathcal{Y}, \mathcal{D}, \sigma, \theta)\to S$ be a rigidified, admissibly marked family over the reduced complex space  $S$ and let $Y_s =\pi^{-1}(s)$. Via the marking, identify $R^2\pi_*\Ar$ with $\widehat{\Lambda}_\Ar \times S$, and set $KS \subseteq \mathcal{A}_{\text{\rm{gen}}} \times S$ to be    
$$\{(x,s): x\in \mathcal{A}(Y_s)\}.$$
Then $KS$ is an open subset of $\mathcal{A}_{\text{\rm{gen}}} \times S$ and the projection $KS \to S$ is an open map. 
\end{lemma}
\begin{proof} Given a point $(x,s_0) \in KS$, there are only finitely many $\beta\in R$ such that $x\in W^\beta$, say $\beta_1, \dots, \beta_k$. Moreover $\varphi_{Y_{s_0}}(\beta_i) \neq 1$ for all $i$. Choose an open subset $V_1$ of $\mathcal{A}_{\text{\rm{gen}}}$ containing $x$ such that , for all $\beta \in R$, $W^\beta\cap V_1 \neq \emptyset$ $\iff$ $\beta = \beta_i$ for some $i$. Let $V_2$ be an open subset of $S$ such that $\varphi_{Y_s}(\beta_i) \neq 1$ for all $i$ and for all $s\in V_2$. 

There exists an $h\in \widehat{\Lambda}_\Q \cap V_1$ such that $Nh$ is ample on $Y_{s_0}$ for some $N\in \mathbb{N}$. After shrinking $V_2$, we can assume that the class $Nh$ defines an ample divisor on $Y_s$ for all $s\in V_2$. Then, for every $s\in V_2$ and every $y\in V_1$, $h\in \mathcal{A}(Y_s)$ and $h$ and $y$ are not separated by a wall $W^\beta$, $\beta \in R$, such that $\varphi_{Y_s}(\beta) = 1$. It follows that $y\in \mathcal{A}(Y_s)$ for all $y\in V_1$. Hence $V_1\times V_2\subseteq KS$, so that $KS$ is  open in $\mathcal{A}_{\text{\rm{gen}}} \times S$. The final statement is then clear.
\end{proof}

To see that $\widetilde{\widehat{\Phi}}$ is continuous, it suffices to check that, given a rigidified, admissibly marked family $\pi\colon (\mathcal{Y}, \mathcal{D}, \sigma, \theta)\to S$ over a  reduced complex space  $S$ the corresponding function $\widetilde{\widehat{\Phi}}_S \colon S \to \widetilde{\widehat{\Omega}}$ is continuous. There is a lifted continuous map $KS \to K\widehat{\Omega}$, defined in the natural way, and it induces the function $\widetilde{\widehat{\Phi}}_S \colon S \to \widetilde{\widehat{\Omega}}$ because the fibers of $KS \to S$ are all equivalent under the equivalence relation defining $\widetilde{\widehat{\Omega}}$. The continuity of $\widetilde{\widehat{\Phi}}_S$ now follows: if $U$ is an open subset of $\widetilde{\widehat{\Omega}}$, then   $\widetilde{\widehat{\Phi}}_S^{-1}(U)$ is obtained by taking the preimage of $U$ in $K\widehat{\Omega}$, pulling back to $KS$, and then projecting to $S$.  Thus $\widetilde{\widehat{\Phi}}^{-1}_S(U)$ is open by the definition of the quotient topology and the fact that projection $KS\to S$ is an open surjective map (Lemma~\ref{Kconeopen}). 

 It   follows similarly that $\widetilde{\widehat{\Phi}}$ is a morphism (cf.\  \cite[\S2]{BurnsRapoport}). Moreover, $\widetilde{\widehat{\Phi}}$ is surjective, by the surjectivity of the period map and the fact that $\mathsf{W}(R^{\text{nod}}_\varphi)$ acts (simply) transitively on the fiber of $\widetilde{\widehat{\Omega}} \to \widehat{\Omega}$ over $\varphi$.  It is injective by the Torelli theorem III, and it is everywhere \'etale. Hence $\widetilde{\widehat{\Phi}}$ is an isomorphism (between two non-separated complex manifolds).
\end{proof}

\begin{remark} There are various versions of the previous statements where we get rid of the rigidification and/or the marking, in terms of quotient stacks. The reader can consult \cite{GHK}, Section 6.
\end{remark}

 \begin{corollary}\label{Torelliinfamilies}  Let $\pi\colon (\mathcal{Y}, \mathcal{D}, \sigma) \to S$ and $\pi'\colon (\mathcal{Y}', \mathcal{D}', \sigma') \to S$ be two rigidified families over the reduced, connected analytic space $S$. Suppose that $$\gamma\colon  R^2(\pi')_*\Zee  \to  R^2\pi_*\Zee$$ is an isomorphism of local systems, preserving the intersection form, such that, for every $s\in S$,
\begin{enumerate}
\item[\rm(i)]  For all $i$,  $\gamma_s([D_i'])=[D_i]$.
\item[\rm(ii)] $\gamma_s(\overline{\mathcal{A}}(Y_s'))= \overline{\mathcal{A}}(Y_s)$. 
\item[\rm(iii)] $\hat\varphi_{Y_s;\sigma(s)}\circ \gamma_s =\hat\varphi_{Y_s';\sigma'(s)}$.
\end{enumerate}
Then there is a unique isomorphism   $\rho\colon (\mathcal{Y}, \mathcal{D}) \to (\mathcal{Y}', \mathcal{D}')$, compatible with the orientations,  such that $\rho\circ \sigma_i  =\sigma_i'$ and $\rho^* = \gamma$. 
\end{corollary}
\begin{proof} The result is local on $S$ and hence we may assume that the families are compatibly and admissibly marked. In this case, the hypotheses imply that the period morphisms $\widetilde{\widehat{\Phi}}_S$ and $\widetilde{\widehat{\Phi}}'_S$ from $S$ to $\widetilde{\widehat{\Omega}}$ coincide. By  Theorem~\ref{extendedpermap}, the classfying morphisms from $S$ to $\widehat{M}$ coincide as well. Thus the families $(\mathcal{Y}, \mathcal{D}, \sigma)$ and $(\mathcal{Y}', \mathcal{D}', \sigma')$ are isomorphic via a unique $\rho$ such that $\rho^* = \gamma$.
\end{proof}

\section{Automorphisms II: admissible maps}

\begin{definition} Let $(Y, D)$ be an anticanonical pair. We define the group $O^+(\Lambda(Y,D)) = O^+(\Lambda)$ to be the group of integral isometries $\gamma$ of $H^2(Y;\Zee)$ such that  $\gamma([D_i]) = [D_i]$ for all $i$ and $\gamma(\mathcal{C}^+) =\mathcal{C}^+$. Note that such a $\gamma$ induces an integral isometry of $\Lambda$. The real points of the corresponding algebraic group will be denoted by  $O^+(\Lambda)_\Ar$. The element  $\gamma\in O^+(\Lambda)$ is \textsl{admissible} if in addition
$$\gamma (\overline{\mathcal{A}}_{\text{\rm{gen}}}(Y)) = \overline{\mathcal{A}}_{\text{\rm{gen}}}(Y),$$
i.e.\ $\gamma$ is admissible as an integral isometry from $H^2(Y;\Zee)$ to itself.
We denote the subgroup of all admissible isometries by $\Gamma= \Gamma(Y,D)$. Note that the groups $O^+(\Lambda)$ and $\Gamma(Y,D)$  only depend on the deformation type of the pair $(Y,D)$, and that $\mathsf{W}(R)$ is a subgroup of $\Gamma(Y,D)$.
\end{definition}

\begin{remark} Suppose that $(Y,D)$ is generic and that $\gamma\in O^+(\Lambda)$. Then $\gamma\in \Gamma(Y,D)$ $\iff$ $\gamma$ permutes the set of interior exceptional curves on $Y$.

In general, if $Y$ is allowed to have $-2$-curves and $\gamma \in \Gamma(Y,D)$, then there is a unique $w\in \mathsf{W}(\Delta_Y)$ such that $w\gamma(\overline{\mathcal{A}}(Y)) = \overline{\mathcal{A}}(Y)$. Then $w\gamma$ permutes the set of interior exceptional curves on $Y$ as well as the set of $-2$-curves 
(cf.\ the proof of Theorem~\ref{firstTorelli}).
\end{remark}

The group $\Gamma$ acts on $\Omega =\Hom(\Lambda, \mathbb{G}_m)$ and $\widehat{\Omega} =\Hom(\widehat{\Lambda}, \mathbb{G}_m)$ in the natural way, via $\gamma\cdot \varphi = \varphi\circ \gamma^{-1}$. However, this action fails in general to be properly discontinuous. Moreover $\Gamma$ acts on the space $K\widehat{\Omega}$ defined prior to the statement of Theorem~\ref{extendedpermap} via $\gamma\cdot (\varphi,x) = (\varphi\circ \gamma^{-1}, \gamma \cdot x)$, and this action preserves the equivalence relation $\sim$. Hence there is an induced $\Gamma$-action on $\widetilde{\widehat{\Omega}}$, and similarly for the non-rigidified analogue $\widetilde{\Omega}$ which is a non-separated \'etale cover of $\Omega$. Thus, there is a $\Gamma$-action on the fine moduli space $\widehat{M}$, given on the level of points by: if $\gamma\in \Gamma$, then 
$$\gamma \cdot(Y,D, p, \theta)= (Y,D, p, \gamma\circ\theta).$$
Under this action, the period map is $\Gamma$-equivariant.  Moreover, the $\Gamma$-action on $\widehat{M}$ lifts to an action on the universal bundle. On the level of functors, the action is given by 
$$\gamma \cdot(\mathcal{Y}, \mathcal{D},\sigma, \theta)= (\mathcal{Y}, \mathcal{D}, \sigma, \gamma\circ\theta).$$
In other words, $\gamma$ is the identity on the fibers of the family but changes the marking.
We then have:

\begin{theorem}\label{liftGamma} Let $(\widehat{\mathcal{Y}}, \widehat{\mathcal{D}}, \hat\sigma, \theta) \to \widehat{M}$ be the universal family of pairs over $\widehat{M}$. Then the $\Gamma$-action on $\widehat{M}$ lifts to a $\Gamma$-action on $(\widehat{\mathcal{Y}}, \widehat{\mathcal{D}}, \hat\sigma, \theta)$. \qed
\end{theorem}

\begin{remark} There exist nonempty $\Gamma$-invariant open subsets of $\widehat{\Lambda}_\Ar$ for which the $\Gamma$-action is properly discontinuous and for which the set of walls $\{W^\beta: \beta \in R\}$ is locally finite. For example, $\mathcal{C}^+$ and $\mathcal{A}_{\text{\rm{gen}}}$ have this property. Using this remark, one can find connected nonempty open subsets $U$ of $\widehat{\Omega}$ (in the classical topology but  \textbf{not} in general in the Zariski topology) such that (i) the \'etale morphism $\widetilde{\widehat{\Omega}} \to \widehat{\Omega}$ is an isomorphism over $U$ and (ii) $U$ is $\Gamma$-invariant  and the $\Gamma$-action on $U$ is properly discontinuous. After further shrinking $U$, we can assume that the $\Gamma$-action has no fixed points. For such a set $U$, we can identify $U$ with its preimage in $\widetilde{\widehat{\Omega}}$ and in $\widehat{M}$, and $U$ and its preimage in $\widehat{M}$ are separated. If $\widehat{\mathcal{Y}}|U \to U$ is the corresponding family, then there is an induced family 
$$(\widehat{\mathcal{Y}}|U)/\Gamma  \to U/\Gamma,$$
which is a family of pairs over $U$ whose monodromy group is $\Gamma$.

For more details on the possible open sets $U$, see \cite[II]{Looij} and \cite[\S7]{GHK}.
\end{remark} 

We turn now to the connection between $\Gamma(Y,D)$ and the automorphism group of the pair $(Y,D)$.

\begin{definition} We define the group of \textsl{Hodge isomorphisms} of $(Y,D)$ by
$$\operatorname{Hodge}(Y,D) = \{\gamma \in \Gamma(Y,D): \varphi_Y\circ \gamma = \varphi_Y\}.$$
In other words, $\operatorname{Hodge}(Y,D)$ is the subgroup of $\Gamma(Y,D)$ fixing the period homomorphism $\varphi_Y$.
\end{definition}

Then the following is a corollary of the Torelli theorem:

\begin{theorem}\label{autsofY} The group $\operatorname{Hodge}(Y,D)$ is isomorphic to a semidirect product:
 $$\operatorname{Hodge}(Y,D) \cong \mathsf{W}({\Delta_Y})\rtimes (\Aut^+(Y,D)/K) .$$
 \end{theorem}
 \begin{proof} Clearly,   $\mathsf{W}({\Delta_Y})$ and $\Aut^+(Y,D)/K$ are subgroups of $\operatorname{Hodge}(Y,D)$ and it is easy to check that $\mathsf{W}({\Delta_Y})$ is a normal subgroup of $\operatorname{Hodge}(Y,D)$. By Theorem~\ref{secondTorelli}, if $\gamma \in \operatorname{Hodge}(Y,D)$, then there exists a unique $w\in \mathsf{W}({\Delta_Y})$ and a $\psi \in \Aut^+(Y,D)$, unique up to an element of $K$, such that $\gamma = w\circ\psi$.   Thus, every element of $\operatorname{Hodge}(Y,D)$ is uniquely written as $w\circ \overline{\psi}$, where $\overline{\psi}$ is the image of $\psi$ in $\Aut^+(Y,D)/K$, and so the group $\operatorname{Hodge}(Y,D)$ is isomorphic to the semi-direct product of $\mathsf{W}({\Delta_Y})$ and $\Aut^+(Y,D)/K$
 \end{proof} 
 
 \begin{corollary} Given a deformation type of anticanonical pairs, let  $(Y_0, D_0)$ be the unique isomorphism class within the given deformation type such that $\varphi_{Y_0}=1$. Then
 $$\Gamma(Y,D) = \Gamma(Y_0, D_0) \cong \mathsf{W}({R_{Y_0}})\rtimes (\Aut^+(Y_0,D_0)/K).$$
 \end{corollary}
 \begin{proof} This follows from Theorem~\ref{autsofY} since $\operatorname{Hodge}(Y_0,D_0)=\Gamma(Y_0, D_0)$ because $\varphi_{Y_0}$ is the constant homomorphism and $\mathsf{W}({\Delta_{Y_0}})= \mathsf{W}({R^{\text{nod}}_{Y_0}})= \mathsf{W}(R_{Y_0})$ by (iii) of Corollary~\ref{preserveamp}.
 \end{proof}
 
For the rest of this section, we describe various results and examples pertaining to the group $\Gamma(Y,D)$ of admissible isometries. To get a feel for the size of $\Gamma(Y,D)$, we prove the following:
 
 \begin{theorem}\label{finiteorbits} Let $(Y,D)$ be a generic  anticanonical pair with self-inter\-sec\-tion sequence $(d_1, \dots, d_r)$ and let $\mathcal{E}(Y,D)$ be the set of interior exceptional curves of $Y$, or equivalently the set of walls of $\overline{\mathcal{A}}_{\text{\rm{gen}}}(Y)$ not corresponding to components of $D$. Then $\Gamma(Y,D)$ acts on $\mathcal{E}(Y,D)$ and the number of $\Gamma(Y,D)$-orbits for this action is finite. More precisely, if $\mathcal{X}_i$ denotes the finite set of deformation types for anticanonical pairs with self-intersection sequence $(d_1, \dots, d_{i-1}, d_i + 1, d_{i+1}, \dots d_r)$, then there is an injection
 $$\mathcal{E}(Y,D)/\Gamma(Y,D) \hookrightarrow \coprod_{i=1}^r\mathcal{X}_i.$$
 \end{theorem}
 \begin{proof} It suffices to prove the existence of the injection in the final sentence of the statement of Theorem~\ref{finiteorbits}. If $E$ is an interior exceptional curve on $Y$, let $(\overline{Y}, \overline{D})$ be the anticanonical pair obtained by contracting $E$. If $E\cdot D_i=1$, then $(\overline{Y}, \overline{D})$ has self-intersection sequence $(d_1, \dots, d_{i-1}, d_i + 1, d_{i+1}, \dots d_r)$ and thus corresponds to a point of $X_i$. In particular, there is a well-defined map $\mathcal{E}(Y,D)\to \coprod_{i=1}^r\mathcal{X}_i$. First, we claim that this map factors through the action of $\Gamma(Y,D)$. In fact, for $\gamma\in \Gamma(Y,D)$,  let $E'$ be the exceptional curve corresponding to $\gamma([E])$ and let $(\overline{Y}', \overline{D}')$ be the anticanonical pair obtained by blowing down $E'$. Then $\gamma$ induces an isometry $\bar\gamma\colon H^2(\overline{Y}; \Zee) \to H^2(\overline{Y}'; \Zee)$ which takes $\overline{\mathcal{A}}_{\text{\rm{gen}}}(\overline{Y}) =  \overline{\mathcal{A}}_{\text{\rm{gen}}}(Y)\cap [E]^\perp$ to $\overline{\mathcal{A}}_{\text{\rm{gen}}}(\overline{Y}')$. By Theorem~\ref{defequiv}, $(\overline{Y}, \overline{D})$ and $(\overline{Y}', \overline{D}')$ are deformation equivalent, and hence define the same point of $\mathcal{X}_i$. Thus we get a well-defined function $\mathcal{E}(Y,D)/\Gamma(Y,D) \to \coprod_{i=1}^r\mathcal{X}_i$.

 To see that this function is an injection, let $E_1$ and $E_2$ be two exceptional curves on $Y$, and suppose that $(\overline{Y}_1, \overline{D}_2)$ is deformation equivalent to $(\overline{Y}_2, \overline{D}_2)$, where $(\overline{Y}_i, \overline{D}_i)$ is the anticanonical pair obtained by blowing down $E_i$.  Then, again by Theorem~\ref{defequiv}, there exists an integral isometry $\bar\gamma \colon 
 H^2(\overline{Y}_1; \Zee) \to H^2(\overline{Y}_2; \Zee)$ taking $\overline{\mathcal{A}}_{\text{\rm{gen}}}(\overline{Y}_1)$ to $\overline{\mathcal{A}}_{\text{\rm{gen}}}(\overline{Y}_2)$ and identifying the corresponding components of the cycles. Identifying $H^2(\overline{Y}_i; \Zee)$ with $[E_i]^\perp \subseteq H^2(Y;\Zee)$, we can then extend $\bar\gamma$ to an integral isometry $\gamma\colon H^2(Y; \Zee) \to H^2(Y; \Zee)$ taking $[E_1]$ to $[E_2]$. Clearly it suffices to prove that $\gamma$ is admissible, i.e.\ that $\gamma(\overline{\mathcal{A}}_{\text{\rm{gen}}}(Y))= \overline{\mathcal{A}}_{\text{\rm{gen}}}(Y)$. But, if $x_1\in \overline{\mathcal{A}}_{\text{\rm{gen}}}(\overline{Y}_1)$ is the class of an ample divisor  on $\overline{Y}_1$, then the corresponding point of $H^2(Y;\Zee)$ is the class of a nef divisor $H_1$ on $Y$ such that $H_1\cdot D > 0$. By Lemma~\ref{chareff},
 the effective numerical exceptional curves on $Y$ are   the numerical exceptional curves $\alpha$ such that $H_1\cdot  \alpha \geq 0$. Clearly, $\bar\gamma (x_1)=x_2\in H^2(\overline{Y}_2; \Zee)$ is also the class of an ample divisor, so that $\alpha$ is an effective  numerical exceptional curve $\iff$ $\alpha\cdot H_1 \geq 0$ $\iff$ $\gamma(\alpha)\cdot \gamma(H_1) \geq 0$ $\iff$ $\gamma(\alpha)$  is an effective  numerical exceptional curve. Thus $\gamma(\overline{\mathcal{A}}_{\text{\rm{gen}}}(Y))= \overline{\mathcal{A}}_{\text{\rm{gen}}}(Y)$ and so $\gamma \in \Gamma(Y,D)$. It follows that $E_1$ and $E_2$ are in the same $\Gamma(Y,D)$-orbit, and hence that $\mathcal{E}(Y,D)/\Gamma(Y,D) \to \coprod_{i=1}^r\mathcal{X}_i$ is injective.  
  \end{proof}
 
 \begin{corollary} For a generic anticanonical pair $(Y,D)$, $\Gamma(Y,D)$ is infinite $\iff$ there are infinitely many interior exceptional curves on $Y$, or equivalently, there are infinitely many walls of $\overline{\mathcal{A}}_{\text{\rm{gen}}}(Y)$.
 \end{corollary}
 \begin{proof} $\impliedby$: This is clear from Theorem~\ref{finiteorbits}.
 
\smallskip
 \noindent $\implies$: It suffices the show that, if there are only finitely many interior exceptional curves on $Y$, then $\Gamma(Y,D)$ is finite. We argue by induction on the number of interior blowdowns of $Y$ to reach a minimal model, where there are no exceptional curves and $\Gamma =\{\Id\}$. Suppose that there are only finitely many interior exceptional curves on $Y$, say $E_1, \dots, E_k$. Then the stabilizer $\overline{\Gamma}$ of $E_1$ has finite index in $\Gamma(Y,D)$ and it will suffice to show that $\overline{\Gamma}$ is finite. But, if $(\overline{Y}, \overline{D})$ is the anticanonical pair obtained by contracting $E_1$, then there are only finitely many interior exceptional curves on $\overline{Y}$, corresponding to the exceptional curves on $Y$ not meeting $E_1$. Hence, by the inductive hypothesis, $\Gamma(\overline{Y}, \overline{D})$ is finite. The method of proof of Theorem~\ref{finiteorbits} then identifies $\Gamma(\overline{Y}, \overline{D})$ with $\overline{\Gamma}$. Hence $\overline{\Gamma}$ and therefore $\Gamma(Y,D)$ are finite.
 \end{proof}
 
  \begin{corollary} Let $(Y,D)$ be a generic anticanonical pair and let $\mathcal{E}_k(Y,D)$ be the set of ordered $k$-tuples $(E_1, \dots, E_k)$, where each $E_i$ is an interior exceptional curve and $E_i\cdot E_j =0$ for all $i\neq j$. Then $\Gamma(Y,D)$ acts on $\mathcal{E}_k(Y,D)$ and the number of $\Gamma(Y,D)$-orbits for this action is finite.  
 \end{corollary}
 \begin{proof} This follows by the same method as the proof of Theorem~\ref{finiteorbits}.
 \end{proof}
 
  \begin{corollary} The group $\Gamma(Y,D)$ acts on the set of roots $R_Y$ and the number of orbits for this action is finite.
 \end{corollary}
 \begin{proof} The corollary is clearly true in case $Y =\mathbb{F}_0$ or $\mathbb{F}_2$. In all other cases, let $\mathcal{X}$ be the set of ordered pairs $(E_1, E_2)$, where $E_1$ and $E_2$ are disjoint exceptional curves and $[E_1]-[E_2]$ is a root, or equivalently $E_1\cdot D_i = 1$ $\iff$ $E_2\cdot D_i =1$. By Theorem~\ref{maintheorem}, the natural map $\mathcal{X}\to R_Y$ is surjective, and it is clearly equivariant for the action of $\Gamma(Y,D)$. By the previous corollary, the set $\mathcal{X}/\Gamma(Y,D)$ is finite, and hence $R_Y/\Gamma(Y,D)$ is finite as well.
 \end{proof}
 
 \begin{remark} In particular, if $R_Y$ is infinite, then $\Gamma(Y,D)$ is infinite as well. However, an example due to Gross-Hacking-Keel shows that, at least in the strictly negative semidefinite case, the converse need not hold (Corollary~\ref{negsemiexs}). In this example, $R_Y=\emptyset$. It is natural to look for negative definite examples, and to ask if (by analogy with the case of $K3$ surfaces), in case $R_Y$ is finite and nonempty,  must $\Gamma(Y,D)$ be finite as well?
 \end{remark}

  The next result is due to Looijenga \cite{Looij}:
 
 \begin{theorem} Suppose that $(Y,D)$ is negative definite and that $r(D)\leq 5$. Then $\Gamma(Y,D) = \mathsf{W}(R)$.
 \end{theorem}
 \begin{proof} Let $\gamma \in \Gamma(Y,D)$. By \cite{Looij} or \cite[Theorem 3.13]{Fried3}, $\gamma \in \Gamma(Y,D)$ $\iff$ $\gamma(\mathcal{C}^+) =  \mathcal{C}^+$, $\gamma([D_i]) = [D_i]$ for all $i$, and $\gamma(R) = R$. Hence, if $\mathcal{O}$ is a fundamental domain for $\mathsf{W}(R)$, corresponding to a root basis $B$ as in \cite{Looij}, there exists a $w\in \mathsf{W}(R)$ such that $w\gamma(\mathcal{O}) =\mathcal{O}$, and hence $w\gamma$ is given by a diagram automorphism of the diagram corresponding to $B$. Moreover $w\gamma([D_i] =[D_i]$ for every $i$. We claim that this implies that $w\gamma=\Id$, and hence that $\gamma = w^{-1} \in \mathsf{W}(R)$.
 
 First note that, by the discussion of the root bases $B$ in \S2 of Part I of \cite{Looij}, the diagram of $B$ has $r$ branches  corresponding to the $r$ components of $D$, and the $i^{\text{th}}$ branch has length depending on the self-intersection number $D_i^2$. If there is a nontrivial diagram automorphism, then there exist $i\neq j$ such that $D_i^2 = D_j^2$ and the diagram automorphism maps the  $i^{\text{th}}$ branch to the $j^{\text{th}}$ branch. Hence, if $\varepsilon_i$ is the element in the finite discriminant form $\Lambda^*/\Lambda$ which is dual to the end component of the $i^{\text{th}}$ branch, then $\gamma(\varepsilon_i) = \varepsilon_j$. On the other hand, by the explicit descriptions in \cite[Part I, \S2]{Looij}, it is easy to see that the element $\varepsilon_i \in  \Lambda^*/\Lambda$ is identified with the dual to the element $[D_i]$ in the finite discriminant form associated to the span of the $[D_i]$. Since $\gamma([D_i]) = [D_i]$ for every $i$, this is only possible if $\varepsilon_i  = \varepsilon_j$. But a straightforward if tedious calculation in case $r\leq 5$ shows that, for $i\neq j$, the dual elements $\varepsilon_i$ and $\varepsilon_j$ are distinct elements of the finite discriminant form. Hence $w\gamma$ is the trivial diagram automorphism, and so is the identity on $\Lambda$. Since $w\gamma([D_i]) =[D_i]$ for every $i$, $w\gamma=\Id$, and hence $\gamma   \in \mathsf{W}(R)$.
 \end{proof}
 
 The next result shows that, in certain circumstances, $\Gamma(Y,D)$ is an arithmetic group:
 
 \begin{theorem}\label{automadm} Suppose that one of the following conditions holds:
 \begin{enumerate}
 \item[\rm(a)] $(Y,D)$ is not negative definite.
 \item[\rm(b)] $D^2=-1$.
 \end{enumerate}
 Then:
 \begin{enumerate}
 \item[\rm(i)]  $\Gamma(Y,D)=O^+(\Lambda)$. 
 \item[\rm(ii)] $R_Y = \{\beta \in \Lambda: \beta^2=-2\}$.
 \end{enumerate}
\end{theorem}
 \begin{proof} Note that, if (i) holds and $\beta \in \Lambda$ satisfies $\beta^2=-2$, then the reflection $r_\beta\in O^+(\Lambda)$. Thus, by (i), $r_\beta \in \Gamma(Y,D)$, and hence $\beta \in R_Y$ by definition. Thus it suffices to prove (i). Moreover, it suffices to prove that, for every $\gamma\in O^+(\Lambda)$ and for every effective numerical exceptional curve $\alpha$, $\gamma(\alpha)$ is again an effective numerical exceptional curve. Clearly $\gamma(\alpha)^2 = \gamma(\alpha)\cdot [K_Y] = -1$, since $\gamma([K_Y]) = [K_Y]$. Thus it suffices to prove in either case (a) or case (b) that $\gamma(\alpha)$ is effective.  By a straightforward reduction,   we may further assume that no component of $D$ is an exceptional curve.  
 
 First suppose that $(Y,D)$ is not negative semidefinite. Then there is a component $D_i$ of $D$ such that $D_i^2\geq 0$. By (i) of Lemma~\ref{caseslemma}, a numerical exceptional curve $\alpha$ is effective $\iff$ $[D_i] \cdot \alpha \geq 0$ $\iff$ $[D_i] \cdot \gamma(\alpha) \geq 0$, since $\gamma([D_i]) =[D_i]$. Thus $\alpha$ is effective $\iff$ $\gamma(\alpha)$ is effective. 
  A similar argument works in case $D$ is strictly negative semidefinite, using (ii) of Lemma~\ref{caseslemma}.
  
  Finally assume that $(Y,D)$ is negative definite and that $D^2=-1$. As in the proof of Lemma~\ref{detbyres}, it suffices to show that $\alpha$ preserves the set of all effective numerical exceptional curves which are not in the $\Zee$-span of the $[D_i]$. Fix a nef and big  $\Ar$-divisor $x$ in $\Lambda_\Ar$ such that $x\cdot [C]> 0$ for all irreducible curves $C$ not equal to $D_i$ for some $i$. By (iii) of Lemma~\ref{caseslemma}, such a numerical exceptional curve $\alpha$ is effective $\iff$ $\alpha \cdot x \geq 0$. Now 
  $$(\alpha+[D])^2 = -1+2-1 =0,$$
  and hence $\alpha + [D]$ lies on the boundary of $\mathcal{C}$, the positive cone. Moreover, by the light cone lemma, $\alpha + [D]$ lies on the boundary of $\mathcal{C}^+$  $\iff$ $(\alpha + [D])\cdot x > 0$ $\iff$ $(\alpha + [D])\cdot x \geq  0$ $\iff$ $\alpha\cdot x \geq 0$ $\iff$ $\alpha$ is effective. If $\gamma\in O^+(\Lambda)$, $\alpha + [D]$ lies on the boundary of $\mathcal{C}^+$  $\iff$ $\gamma(\alpha + [D])$ lies on the boundary of $\mathcal{C}^+$. Hence $\gamma$ preserves the set of all effective numerical exceptional curves which are not in the $\Zee$-span of the $[D_i]$, so that $\gamma (\overline{\mathcal{A}}_{\text{\rm{gen}}} ) = \overline{\mathcal{A}}_{\text{\rm{gen}}}$.
  \end{proof}
 
 We next consider the case where $D$ is strictly negative semidefinite: $D =\sum_{i=0}^{r-1}D_i$,  where $D_i^2=-2$ for all $i$, and $r>1$. Then $D^2=0$, and either $\scrO_Y(D)|D$ is nontrivial (i.e.\ $\varphi_Y([D]) \neq 1$) and $h^0(Y; \scrO_Y(D)) = 1$, or $\scrO_Y(D)|D \cong \scrO_D$ (i.e.\ $\varphi_Y([D]) = 1$), $h^0(Y; \scrO_Y(D)) = 2$,  and $Y$ is a rational elliptic surface. In any case it is easy to check that $Y$ is the blowup of $\Pee^2$ at $9$ points and that $[D]^\perp/\Zee[D] \cong (-E_8)$. In fact, if $\sigma_0$ is an exceptional curve and $\sigma _0 \cdot D_0 = 1$, then 
$$H^2(Y, \Zee) \cong \Zee[D] \oplus \Zee[\sigma_0] \oplus \{[D], [\sigma_0]\}^\perp,$$
where $\{[D], [\sigma_0]\}^\perp \cong -E_8$ is the unique even, negative definite unimodular lattice of rank $8$ up to isometry.

\medskip
\noindent\textbf{Convention:} To avoid an endless number of signs in what follows, we shall take as a convention that root lattices are \textbf{negative definite}. In particular, every root in a root lattice has square $-2$. Thus, we will denote by $E_8$ what was denoted $-E_8$ above.

\medskip
The possible deformation types of the pairs $(Y,D)$ correspond to embeddings of the lattice $A_{r-1}$ into $E_8$:

\begin{proposition} Let $(Y,D)$ be an anticanonical pair of length $r$ and such that $D_i^2 = -2$ for all $i$. Then there is a bijection between the set of deformation types of the pairs $(Y,D)$ and the set of embeddings of the root lattice  of type $A_{r-1}$ into that of type $E_8$ modulo the action of the Weyl group of $E_8$.
\end{proposition}
\begin{proof} (Sketch.) Given a deformation type of $(Y,D)$, it is easy to check that $[D]\in \Lambda$ is a primitive isotropic vector and that $[D]^\perp/\Zee\cdot [D] \cong E_8$. Thus, there is an associated embedding of the lattice 
$$(\Zee[D_1]+ \cdots + \Zee[D_r])/\Zee\cdot [D] \cong A_{r-1}$$
 into $E_8$, well-defined up to the choice of an isomorphism from $[D]^\perp/\Zee\cdot [D]$ to $E_8$, and hence well-defined up to the group of isometries of $E_8$, which is $\mathsf{W}(E_8)$.

To go the other way, we must show that an embedding of the root lattice of type $A_{r-1}$ into $E_8$, modulo isometries of $E_8$, determines a deformation type of pairs $(Y,D)$ consistent with the above construction. Note that, if $(Y,D)$ is such a pair and $\varphi_Y([D]) =1$, then  $Y$ is a rational elliptic surface and  $D$ is a fiber of type I${}_r$ in Kodaira's notation. Hence the $j$-invariant of $Y$ is non-constant.  Fixing a smooth elliptic curve $F$ with generic $j$-invariant, the space of rational elliptic surfaces with a fiber isomorphic to $F$ is identified with $\Hom(E_8, F)/\mathsf{W}(E_8)$ by analogy with the period map for anticanonical pairs: Identify $E_8$ with $[F]^\perp/\Zee\cdot [F]$, where $[F]\in H^2(Y; \Zee)$. Then there is a well defined homomorphism $\varphi_{(Y,F)}\colon [F]^\perp \to J^0(F) \cong F$, defined on line bundles $L$ such that $\deg(L|F) =0$ by sending $L$ to $L|F$, and it descends to  $[F]^\perp/\Zee\cdot [F]$ since $\scrO_Y(F)|F$ is trivial. 

Let $\iota$ be an embedding of the root lattice  of type $A_{r-1}$ into $E_8$, with image $\Lambda'$. If $\varphi\in \Hom(E_8, F)$ satisfies $\Ker \varphi =\Lambda'$, then it is straightforward to check that the corresponding elliptic surface has a (unique) fiber of type I${}_r$, and thus determines a deformation type of pairs $(Y,D)$. Moreover, there is an exact sequence
$$0\to \Hom(E_8/\Lambda', F) \to \Hom(E_8, F) \to \Hom (\Lambda', F).$$
If $\Lambda'$ is a primitive sublattice of $E_8$, then $\Hom(E_8/\Lambda', F)$ is connected and the space of   $\varphi\in \Hom(E_8, F)$ such that $\Ker \varphi =\Lambda'$ is a dense open subset of $\Hom(E_8/\Lambda', F)$ and is thus connected as well. In this case, the embedding $\iota$ determines a unique deformation type. If $\Lambda'$ is not a primitive sublattice of $E_8$, then in any event the torsion subgroup of $E_8/\Lambda'$ is cyclic of order $m$, say. A somewhat more involved argument with the  action of $SL_2(\Zee)$, the monodromy group for the ``universal" family of elliptic curves,
on the set of torsion points of $F$ of order exactly $m$   then handles this case.
\end{proof}

We next classify the embeddings of the root lattice  of type $A_{r-1}$ into that of type $E_8$.

\begin{proposition}\label{semidefcomps} {\rm(i)} For $r\leq 9$ and $r\neq 8$, there is a unique embedding of the root lattice corresponding to $A_{r-1}$ into the root lattice of $E_8$ up to the action of $\mathsf{W}(E_8)$, the automorphism group of the $E_8$ lattice. For $r=8$, there are two possible embeddings of the root lattice of $A_7$ in the root lattice of $E_8$ up to the action of $\mathsf{W}(E_8)$, exactly one of which is primitive.

\smallskip
\noindent {\rm(ii)} The set of elements of square $-2$ in the orthogonal complement to the root lattice of $A_{r-1}$ into the root lattice of $E_8$ spans the orthogonal complement over $\Zee$, unless $r=7$ or $r=8$ and for the primitive embedding of the root lattice of type $A_7$ into the root lattice of $E_8$. The possible root systems are as follows: 
\medskip

\begin{tabular}{|l|c|}\hline
r, type of embedding & type of root system in $A_{r-1}^\perp$ \\ 
\hline
\hline
$r=2$ & $E_7$ \\ \hline
$r=3$ & $E_6$ \\ \hline
$r=4$ & $D_5$ \\ \hline
$r=5$ & $A_4$ \\ \hline
$r=6$ & $A_1+A_2$ \\ \hline
$r=7$ & $A_1$ \\ \hline
$r=8$, the primitive embedding & $\emptyset$ \\ \hline
$r=8$, the imprimitive embedding & $A_1$ \\ \hline
$r=9$  & $\emptyset$ \\ \hline
\end{tabular}

\medskip
\noindent {\rm(iii)} The torsion subgroup of the quotient of $E_8$ by the image of the root lattice of $A_{r-1}$ is isomorphic to $\Zee/2\Zee$, for $r=8$ and the imprimitive embedding, to $\Zee/3\Zee$, for $r=9$, and is trivial in all other cases.
\medskip

\end{proposition}
\begin{proof} By the Borel-de Siebenthal procedure, every embedding  of the root lattice corresponding to $A_{r-1}$ in that of $E_8$ is obtained by an inclusion of the Dynkin diagram of $A_{r-1}$ in the extended Dynkin diagram $\widetilde{E}_8$. For $r\neq 8$, using the easy fact that every two inclusions of a root lattice of type $A_k$ in one of type $A_\ell$ are Weyl equivalent, it is easy to check that every two such inclusions are equivalent under the action of $\mathsf{W}(E_8)$.

In case $r=8$,  there are two possible embeddings of the root lattice of $A_7$ in the root lattice of $E_8$. The first one corresponds to the subdiagram  $A_7$ of $E_8$ obtained by deleting the vertex corresponding to $\alpha_2$ (in the notation of Bourbaki) and is a primitive embedding. The second corresponds to the subdiagram $A_7 + A_1$ of the extended Dynkin diagram $\widetilde{E}_8$, obtained by deleting the vertex corresponding to $\alpha_3$. In the first case, there is no root in the orthogonal complement of the $A_7$ lattice, while in the second case the orthogonal complement of the $A_7$ lattice contains the roots $\pm \alpha_1$. Note also that, in the first case, the embedding of the $A_7$ lattice is primitive, since it is spanned by a subset of a $\Zee$-basis for the $E_8$ lattice. In the second case, the inclusion of the $A_7$ lattice factors through the inclusion of the $A_7$ lattice in the $E_7$ lattice (again given by the Borel-de Siebenthal procedure) and is thus of index two in its saturation. Finally, when $r=9$, the embedding of the $A_8$ lattice in the $E_8$ lattice has index three. 
\end{proof}

\begin{remark} In the case where $Y$ is elliptic, the list of possibilities in Proposition~\ref{semidefcomps} has been enumerated by Miranda-Persson in a more precise form \cite[Theorem 4.1]{MirandaPersson}.
\end{remark}

Next we describe the relevant lattice theory. Fixing an exceptional curve $\sigma_0$ as above, let $\Lambda_0 = \{[D], [\sigma_0]\}^\perp \cong E_8$. Let $\delta_0, \dots, \delta_{r-1}$ be the images of $[D_i]$ in $\Lambda_0$ under the projection, so that $\delta_0+\dots + \delta_{r-1} = 0$, and let $\overline{\Lambda}$ be the orthogonal complement in $\Lambda_0$ of $\delta_1, \dots, \delta_{r-1}$.

\begin{lemma}\label{detailsinsemidef} With notation as above,
\begin{enumerate}
\item[\rm(i)] $\Lambda =\Lambda(Y,D) = \Zee [D]\oplus \overline{\Lambda}$ as an orthogonal direct sum.
\item[\rm(ii)] The group $K = K(Y,D)$ of Definition~\ref{defsomeauts} is trivial unless $r=8$ and the embedding of the $E_7$ lattice is not primitive, in which case $K\cong \Zee/2\Zee$, or $r=9$, in which case $K\cong \Zee/3\Zee$. 
\item[\rm(iii)] If $\Gamma = \Gamma(Y,D)$ is the group of admissible isometries of $Y$, then there is an injection $\iota \colon \overline{\Lambda} \to \Gamma$, such that, after identifying $\overline{\Lambda}$ with its image in $\Gamma$, 
$$\Gamma = \overline{\Lambda} \rtimes G,$$
where $G$ is the (finite) group of integral isometries of $\overline{\Lambda}$ which are induced from an integral isometry of $\Lambda_0$, and the action of $G$ on $\overline{\Lambda}$ is the standard one.
\end{enumerate}
\end{lemma}
\begin{proof}  (i): Clearly $\Zee [D]\oplus \overline{\Lambda} \subseteq \Lambda$. Given $\lambda \in \Lambda$, there exists an $n\in \Zee$ such that $\lambda -n[D] \in \{[\sigma_0], [D]\}^\perp$, and hence $\lambda -n[D] \in \overline{\Lambda}$. Thus  $\Lambda  = \Zee [D]\oplus \overline{\Lambda}$.

\smallskip
\noindent (ii): By definition, $K$ is dual to the group $F$ which  is the kernel of the natural map $\widehat{\Lambda} \to \Zee^r$ which sends $\lambda \in \lambda$ to $(\lambda\cdot [D_0], \dots, \lambda\cdot [D_{r-1}])$. Thus $F =\{0\}$ $\iff$ the $\Zee$-span of $[D_0], \dots, [D_{r-1}]$ is primitively embedded in $\widehat{\Lambda}$, and in general $F \cong \operatorname{Ext}^1(L, \Zee)$, where $L = \widehat{\Lambda}/\Zee[D_0] + \cdots + \Zee[D_{r-1}]$. Since the $[D_i]$ are contained in $\Lambda$, which is a primitive sublattice of $\widehat{\Lambda}$, it is enough to look at the torsion in $\Lambda/\Zee[D_0] + \cdots + \Zee[D_{r-1}]$. Now an elementary argument shows that $\Lambda/\Zee[D_0] + \cdots + \Zee[D_{r-1}]\cong  \Lambda_0/\Zee[D_1] + \cdots + \Zee[D_{r-1}]$. The result then follows from Proposition~\ref{semidefcomps}(iii).

\smallskip
\noindent (iii): Note that every integral isometry of $\widehat{\Lambda} = H^2(Y; \Zee)$ which fixes the classes $[D_i]$ fixes $[D]$ and hence the positive cone $\mathcal{C}^+$, since $[D]^2 = 0$. Then automatically such isometries are elements of $\Gamma$, by Theorem~\ref{automadm}.  Clearly $G$ can be viewed as a subgroup of such isometries, hence $G \subseteq \Gamma$. Next, given $\lambda \in \overline{\Lambda}$, define $a_\lambda \colon \widehat{\Lambda} \to \widehat{\Lambda}$ via:
\begin{align*}
a_\lambda([\sigma_0]) & = [\sigma_0] + \lambda - \frac{\lambda^2}{2}[D];\\
a_\lambda([D]) &= [D];\\
a_\lambda(\alpha) & =\alpha - (\alpha \cdot \lambda) [D], \text{ if $\alpha \in \Lambda_0$}.
\end{align*}
It is easy to check that $a_\lambda$ is an isometry of $\widehat{\Lambda}$ and that $\lambda \mapsto a_\lambda$ defines an injective homomorphism $\iota \colon \overline{\Lambda} \to \Gamma$.

Given an automorphism $\psi\colon \widehat{\Lambda} \to \widehat{\Lambda}$ such that $\psi([D_i]) = [D_i]$ for every $i$, let $\psi(\sigma_0) = \sigma$. Then $\sigma^2 =\sigma_0^2 = -1$, $\sigma\cdot [D] = \sigma_0 \cdot [D] = 1$, $\sigma\cdot [D_i] = \sigma_0 \cdot [D_i]$ for all $i$, $\sigma \cdot \sigma_0 = a$ for some $a\in \Zee$. Let $\lambda = \sigma -\sigma_0 -(a+1)[D]$.  Then $\lambda \in \Lambda_0$ and $\lambda\cdot [D_i] =0$ for all $i$, so that $\lambda \in \widehat{\Lambda}$. Moreover, $a_\lambda^{-1}\circ \psi = \eta$, where $\eta(\sigma_0) = \sigma_0$ and $\eta([D]) = [D]$. Hence $\eta \in G$, and every $\psi \in \Gamma$ can be written as $a_\lambda \cdot \eta$ for $\lambda \in \overline{\Lambda}$ and $\eta\in G$. It is easy to see that this decomposition is unique, i.e.\ that $\iota(\overline{\Lambda}) \cap G =\{\Id\}$. Finally, from
$$\eta \circ a_\lambda \circ \eta^{-1} = a_{\eta(\lambda)},$$
we see that $\iota(\overline{\Lambda})$ is a normal subgroup of $\Gamma$ and that the action of $G$ on $\iota(\overline{\Lambda})\cong \overline{\Lambda}$ is the standard one.
\end{proof}

\begin{remark} Note that $a_\lambda$ is a Hodge isometry of $Y$ for every $\lambda \in \overline{\Lambda}$  $\iff$ $\varphi_Y([D]) =1$, i.e.\ $\iff$ $Y$ is elliptic and $D$ is a fiber. In this case, the Mordell-Weil group of $Y$ is the quotient of $\Lambda_0$ by the classes of the irreducible components of reducible fibers not meeting $\sigma_0$. Thus, the Mordell-Weil group is isomorphic to a quotient of $\Lambda_0/\Zee[D_1] + \cdots + \Zee[D_{r-1}]$, and for generic elliptic surfaces of this type, the Mordell-Weil group is isomorphic to $\Lambda_0/\Zee[D_1] + \cdots + \Zee[D_{r-1}]$ (which has torsion exactly when $r=7$ and the $A_7$ lattice is not primitively embedded in the $E_8$ lattice or when $r=9$). Up to a finite discrepancy, then, the generic Mordell-Weil group is $\overline{\Lambda}$.
\end{remark}

We then have the following   \cite[Example 5.6 and Example 5.3]{GHK}.

\begin{corollary}\label{negsemiexs} If $(Y,D)$ is strictly negative semidefinite, then $\Gamma(Y,D) = \mathsf{W}(R_Y)$  except for the cases $r=7$ or $r=8$ and the embedding of $A_7$ is primitive. In both of these cases, $\mathsf{W}(R_Y)$ has infinite index in $\Gamma(Y,D)$. In the first case, $R_Y$ is given by $\widetilde{A}_1$ and  $\mathsf{W}(R_Y)$ is the affine Weyl group of $A_1$. In the second case, $R_Y =\emptyset$ and hence $\mathsf{W}(R_Y) =\{1\}$, but $\Gamma(Y,D)$ is infinite. 
\end{corollary}
\begin{proof} We shall just write out the argument in the two exceptional cases where $\mathsf{W}(R_Y)$ has infinite index in $\Gamma(Y,D)$. Clearly, if $\beta\in \Lambda$ and $\beta^2=-2$, then the image $\bar{\beta}$ of $\beta$ in $\overline{\Lambda}$ satisfies $(\bar\beta)^2=-2$. Conversely, if  $\bar\beta \in \overline{\Lambda}$ satisfies $(\bar\beta)^2=-2$, then, for every $k\in \Zee$, $\bar\beta + k[D] =\beta\in \Lambda$ satisfies $\beta^2=-2$. Thus, if there are no such elements $\bar\beta \in \overline{\Lambda}$, then $R_Y =\emptyset$. This handles the case $r=8$ for the primitive embedding. For $r=7$, it is clear that $\mathsf{W}(R_Y)$ is the affine Weyl group of $A_1$, which is the semidirect product $\Zee\bar\beta \rtimes \{\pm 1\}$, where $\bar\beta$ is a root in $\overline{\Lambda}$.  Moreover this semidirect product decomposition is compatible with that of Lemma~\ref{detailsinsemidef}(iii). Since $\overline{\Lambda}$ has rank two, $\Zee\bar\beta$ has infinite index in $\overline{\Lambda}$, and hence  $\mathsf{W}(R_Y)$ has infinite index in $\Gamma(Y,D)$. 
\end{proof}

\begin{example} Finally, we give a negative definite example where $\mathsf{W}(R_Y)$ has infinite index in $\Gamma(Y,D)$ (and is itself infinite). For this example, $D^2=-1$ and $r=8$, so that the self-intersection sequence of $(Y,D)$ is $(-3, -2, -2, -2, -2, -2, -2, -2)$. Although we shall not need to know this, as in the case of Proposition~\ref{semidefcomps}, the possible deformation types of such pairs $(Y,D)$ correspond to embeddings of $A_7$ into $[D]^\perp$, which is a lattice $\Upsilon$ of type $E_{10}$ with our sign conventions and hence has signature $(1, 9)$. As a lattice, $\Upsilon$ is isomorphic to $U\oplus  E_8$, where $U$ is the standard rank two even unimodular hyperbolic lattice, and there is a diagram of vectors of square $-2$ of type $T_{2,3,7}$ in the notation of \cite{Gabrielov}. The embeddings of the $A_7$ lattice in $\Upsilon$ can be classified. There are two such embeddings up to the action of $O^+(\Upsilon)$ (which is in fact generated by reflections in the elements of square $-2$), as can easily be checked directly by hand. The first is a primitive embedding corresponding to realizing $A_7$ as a subdiagram of  $T_{2,3,7}$. The second is imprimitive, and corresponds to first realizing $A_7$ as an index two sublattice of $E_7$. By Theorem~\ref{automadm}, since $D^2 =-1$,   $\Gamma(Y,D)=O^+(\Lambda)$ and every element of square $-2$ in $\Lambda$ is a root, for both deformation types of pairs $(Y,D)$.

We shall only be concerned with the primitive case. In this case, we can realize the pairs $(Y,D)$ concretely as follows: beginning with a nodal cubic in $\Pee^2$, make $7$ infinitely near blowups at the node, then blow up a point on the last exceptional curve and blow up two more points which are smooth points of the proper transform of the cubic. Thus there are $7$ $-2$-curves $D_1, \dots, D_7$, with the class of $D_i=E_i-E_{i+1}$, and the proper transform $D_0$ of the nodal cubic  is $3H - 2E_1-E_2 - \cdots - E_7 - E_9-E_{10}$. (Here and in what follows, the $E_i$ refer to exceptional curves, not lattices.) Thus 
$$D= D_0 + D_1 + \cdots + D_7 = 3H -  E_1-E_2 - \cdots - E_7 - E_8 -E_9-E_{10},$$ 
$[D]^\perp = \Upsilon$, and  $\Lambda$  is the orthogonal complement in $\Upsilon=E_{10}$ of $E_1-E_2, \dots, E_7-E_8$, 

\begin{proposition} The group generated by reflections about the elements of square $-2$ in $\Lambda$ has infinite index in $O^+(\Lambda)$. Hence $\mathsf{W}(R_Y)$ has infinite index in $\Gamma(Y,D)=O^+(\Lambda)$.
\end{proposition}
\begin{proof} Let $\gamma_1 =[3H - \sum_{i=1}^8E_i - E_9]$, 
$\gamma_2=[3H - \sum_{i=1}^8E_i - E_{10}]$, and 
$\gamma_3=[8H - 3\sum_{i=1}^8E_i]$. 
A calculation shows that $\gamma_i \in \Lambda$, $1\leq i\leq 3$,  $\gamma_1^2 =\gamma_2^2 =0$, $\gamma_1\cdot \gamma_2 =1$, $\gamma_1\cdot \gamma_3 = \gamma_2\cdot \gamma_3 = 0$, and $\gamma_3^2=-8$. Since $\gamma_3$ is clearly primitive, $\Lambda$ is isomorphic to $U\oplus (-8)$, where $(n)$ denotes the  rank one lattice generated by a vector of square $n$. But a result due to Vinberg \cite{Vinberg} (see also \cite[Lemma 4.8]{Dolgachev}) shows that, for a lattice $\Lambda$ isomorphic to $U\oplus (n)$ with $n<0$, the reflections about the vectors of square $-2$ generate a finite index subgroup of $O^+(\Lambda)$ $\iff$ $n=-2$.  
\end{proof} 

\begin{remark} It is shown in \cite[Example 5.5]{GHK} that, for every integer $k$, the element $(4k^2-1)\gamma_1 + \gamma_2 + k\gamma_3$ is the class of a $-2$-curve on $Y_0$, where $(Y_0, D_0)$ is the unique anticanonical pair in the above deformation type for which $\varphi_{Y_0} =1$ (and for an appropriate labeling of the generalized exceptional curves on $Y_0$). In particular, there are infinitely many $-2$-curves on $Y_0$.
\end{remark}

\begin{remark} For the imprimitive embedding of the $A_7$ lattice in $E_{10}$, $\Lambda \cong  U\oplus (-2)$ and in fact $\Gamma(Y,D) = \mathsf{W}(R_Y)$.
\end{remark}

\end{example}

\bigskip
\noindent
Department of Mathematics \\
Columbia University \\
New York, NY 10027 \\
USA

\bigskip
\noindent
{\tt rf@math.columbia.edu}

\end{document}